\documentclass[reqno,10pt,centertags,draft]{amsart}
\usepackage{amsmath,amsthm,amscd,amssymb,latexsym,upref,enumerate}

\newcommand{\bbA}{{\mathbb{A}}}
\newcommand{\bbB}{{\mathbb{B}}}
\newcommand{\bbC}{{\mathbb{C}}}

\newcommand{\bbN}{{\mathbb{N}}}
\newcommand{\bbR}{{\mathbb{R}}}

\newcommand{\bbZ}{{\mathbb{Z}}}

\newcommand{\cC}{{\mathcal C}}

\newcommand{\cM}{{\mathcal M}}

\newcommand{\cU}{{\mathcal U}}

\newcommand{\dott}{\,\cdot\,}
\newcommand{\no}{\notag}
\newcommand{\lb}{\label}
\newcommand{\f}{\frac}

\newcommand{\ol}{\overline}

\newcommand{\wti}{\widetilde}

\newcommand{\oh}{o}

\newcommand{\Arc}{\text{\rm{Arc}}}

\newcommand{\supp}{\text{\rm{supp}}}

\newcommand{\bi}{\bibitem}
\newcommand{\hatt}{\widehat}

\newcommand{\spn}{{\text{\rm span}}}

\newcommand{\st}{\;|\;}

\renewcommand{\Re}{\text{\rm Re}}
\renewcommand{\Im}{\text{\rm Im}}

\newcommand{\abs}[1]{\lvert#1\rvert}
\newcommand{\norm}[1]{\left\Vert#1\right\Vert}
\newcommand{\ph}[1]{\phantom{#1}}

\newcommand{\Om}{\Omega}

\newcommand{\si}{\sigma}
\newcommand{\om}{\omega}
\newcommand{\la}{\lambda}
\newcommand{\al}{\alpha}

\newcommand{\ga}{\gamma}
\newcommand{\De}{\Delta}
\newcommand{\de}{\delta}
\newcommand{\te}{\theta}
\newcommand{\Te}{\Theta}
\newcommand{\ze}{\zeta}

\newcommand{\veps}{\varepsilon}

\newcommand{\C}{\mathbb{C}}
\newcommand{\Cz}{{\C\backslash\{0\}}}

\newcommand{\D}{\mathbb{D}}
\newcommand{\dD}{{\partial\hspace*{.2mm}\mathbb{D}}}
\newcommand{\Z}{{\mathbb{Z}}}
\newcommand{\N}{{\mathbb{N}}}

\newcommand{\U}{{\mathbb{U}}}
\newcommand{\V}{{\mathbb{V}}}
\newcommand{\W}{{\mathbb{W}}}
\newcommand{\T}{{\mathbb{T}}}
\renewcommand{\L}{{\mathbb{L}}}

\newcommand{\Cm}{{\mathbb{C}^{m\times m}}}

\newcommand{\s}[1]{{\mathrm{s}(#1)}}
\newcommand{\sm}[1]{{\mathrm{s}(#1)^m}}
\newcommand{\smn}[1]{{\mathrm{s}(#1)^{m\times n}}}
\newcommand{\smm}[1]{{\mathrm{s}(#1)^{m\times m}}}
\newcommand{\ltm}[1]{{\ell^2(#1)^m}}
\newcommand{\ltzm}[1]{{\ell_0^2(#1)^{m}}}
\newcommand{\ltmn}[1]{{\ell^2(#1)^{m\times n}}}
\newcommand{\ltmm}[1]{{\ell^2(#1)^{m\times m}}}

\newcommand{\Ltm}[1]{{L^2(\dD;d\Om_{#1}(\cdot,k_0))}}

\newtheorem{theorem}{Theorem}[section]

\newtheorem{lemma}[theorem]{Lemma}
\newtheorem{corollary}[theorem]{Corollary}
\newtheorem{hypothesis}[theorem]{Hypothesis}
\theoremstyle{definition}
\newtheorem{definition}[theorem]{Definition}
\newtheorem{remark}[theorem]{Remark}

\allowdisplaybreaks \numberwithin{equation}{section}

\begin{document}

\title[Weyl--Titchmarsh Theory and Uniqueness results for CMV operators]
{Weyl--Titchmarsh Theory and Borg--Marchenko-type Uniqueness Results for CMV Operators with Matrix-Valued Verblunsky Coefficients}

\author[S.\ Clark, F.\ Gesztesy, and M.\ Zinchenko]
{Stephen Clark, Fritz Gesztesy, and Maxim Zinchenko}

\address{Department of Mathematics \& Statistics,
University of Missouri, Rolla, MO 65409, USA}
\email{sclark@umr.edu}
\urladdr{http://web.umr.edu/\~{}sclark/index.html}

\address{Department of Mathematics,
University of Missouri, Columbia, MO 65211, USA}
\email{fritz@math.missouri.edu}
\urladdr{http://www.math.missouri.edu/personnel/faculty/gesztesyf.html}

\address{Department of Mathematics,
California Institute of Technology, Pasadena, CA 91125, USA}
\email{maxim@caltech.edu}
\urladdr{http://www.math.caltech.edu/\~{}maxim}

\dedicatory{Dedicated with great pleasure to Eduard Tsekanovskii on the occasion
of his 70th birthday}

\thanks{Based upon work supported by the US National Science
Foundation under Grants No.\ DMS-0405526 and DMS-0405528.}
\thanks{{\it Operators and Matrices.} {\bf 1}, 535--592 (2007).}

\date{\today}

\begin{abstract}
We prove local and global versions of Borg--Marchenko-type uniqueness theorems for
half-lattice and full-lattice CMV operators (CMV for Cantero, Moral, and Vel\'azquez 
\cite{CMV03}) with matrix-valued Verblunsky coefficients. While our half-lattice results are formulated in terms of matrix-valued Weyl--Titchmarsh functions, our full-lattice results involve the diagonal and main off-diagonal Green's matrices.

We also develop the basics of Weyl--Titchmarsh theory for CMV operators with
matrix-valued Verblunsky coefficients as this is of independent interest and an essential ingredient in proving the corresponding Borg--Marchenko-type uniqueness theorems.
\end{abstract}

\maketitle

\section{Introduction}\lb{s1}

Since Borg--Marchenko-type uniqueness theorems were first formulated in the context of scalar Schr\"odinger operators on half-lines, we start with a brief review of these results: Let $H_j = -\f{d^2}{dx^2} + V_j$, $V_j\in L^1 ([0,R]; dx)$ for all $R>0$, $V_j$
real-valued, $j=1,2$, be two self-adjoint operators in $L^2 ([0,\infty); dx)$ which, just for simplicity, have a Dirichlet boundary condition at $x=0$ (and possibly a self-adjoint boundary condition at infinity). Let $m_j(z)$, $z\in\bbC\backslash\bbR$, be the
Weyl--Titchmarsh $m$-functions associated with $H_j$, $j=1,2$. Then the celebrated Borg--Marchenko uniqueness theorem, in this particular context, reads as follows:

\begin{theorem} \lb{t1.1} 
Suppose
\begin{equation}
m_1(z) = m_2(z), \;\; z\in\bbC\backslash\bbR, \, \text{ then } \,
V_1(x) = V_2(x) \, \text{ for a.e.\  $x\in [0,\infty)$.}   \lb{1.1}
\end{equation}
\end{theorem}

This result was published by Marchenko \cite{Ma50} in 1950. Marchenko's extensive treatise on spectral theory of one-dimensional Schr\"odinger operators \cite{Ma52}, repeating the proof of his uniqueness theorem, then appeared in 1952, which also marked the appearance of Borg's proof of the uniqueness theorem \cite{Bo52} (apparently, based on his lecture at the 11th Scandinavian Congress of Mathematicians held at Trondheim, Norway in 1949).

We emphasize that Borg and Marchenko also treat the general case of
non-Dirichlet boundary conditions at $x=0$ (in which equality of the two $m$-functions also identifies the two boundary conditions), moreover, Marchenko also simultaneously discussed the half-line and the finite interval case. For brevity we chose to illustrate the simplest possible case only.

To the best of our knowledge, the only alternative approaches to Theorem \ref{t1.1}
are based on the Gelfand--Levitan solution \cite{GL51} of the inverse spectral problem published in 1951 (see also Levitan and Gasymov \cite{LG64}) and alternative variants due to M.~Krein \cite{Kr51}, \cite{Kr53}. For over 45 years, Theorem \ref{t1.1} stood the test of time and resisted any improvements. Finally, in 1998, Simon \cite{Si98} proved the following spectacular result, a local Borg--Marchenko theorem (see part $(i)$ below) and a significant improvement of the original Borg--Marchenko theorem (see part $(ii)$ below):

\begin{theorem}\lb{t1.2} ${}$ \\
$(i)$ Let $a>0$, $0<\veps<\pi/2$ and
suppose that
\begin{equation} \lb{1.4}
|m_1 (z) - m_2(z)| \underset{|z|\to\infty}{=}  O(e^{-2\Im (z^{1/2})a})
\end{equation}
along the ray $\arg(z) = \pi-\veps$. Then
\begin{equation} \lb{1.5}
V_1(x) = V_2 (x) \, \text{ for a.e.\ $x\in [0,a]$.}
\end{equation}
$(ii)$ Let $0<\veps <\pi/2$ and suppose
that for all $a>0$,
\begin{equation}
|m_1(z) - m_2(z)| \underset{|z|\to\infty}{=} O(e^{-2\Im (z^{1/2})a})   \lb{1.6}
\end{equation}
along the ray $\arg(z) = \pi -\veps$. Then
\begin{equation}
V_1(x) = V_2(x) \, \text{ for a.e.\  $x\in [0,\infty)$.}    \lb{1.7}
\end{equation}
\end{theorem}

The ray $\arg(z) = \pi -\veps$, $0<\veps < \pi/2$ chosen in
Theorem \ref{t1.2} is of no particular importance. A limit taken along any non-self-intersecting curve $\cC$ going to infinity in the sector
$\arg(z)\in ((\pi/2)+\veps, \pi -\veps)$ is permissible. For simplicity we only discussed the Dirichlet boundary condition $u(0)=0$ thus far. However, everything extends to the case of general boundary conditions $u'(0) + h u(0) = 0$, $h\in\bbR$.  Moreover, the case of a finite interval problem on $[0,b]$, $b\in (0,\infty)$, instead of the half-line $[0,\infty)$ in Theorem \ref{t1.2}\,$(i)$, with $0<a<b$, and a self-adjoint boundary condition at $x=b$ of the type $u'(b) + h_b u(b) = 0$,  $h_b \in \bbR$, can be handled as well. All of this is treated in detail in \cite{GS00a}.

Remarkably enough, the local Borg--Marchenko theorem proven by Simon \cite{Si98} was just a by-product of his new approach to inverse spectral theory for half-line
Schr\"odinger operators. Actually, Simon's original result in \cite{Si98} was obtained under a bit weaker conditions on $V$; the result as stated in Theorem \ref{t1.2} is taken from \cite{GS00a} (see also \cite{GS00}). While the original proof of the local
Borg--Marchenko theorem in \cite{Si98} relied on the full power of a new formalism in inverse spectral theory, a short and fairly elementary proof of Theorem \ref{t1.2} was presented in \cite{GS00a}. Without going into further details at this point, we also mention that \cite{GS00a} contains the analog of the local Borg--Marchenko uniqueness result, Theorem \ref{t1.2} for Schr\"odinger operators on the real line. In addition, the case of half-line Jacobi operators and half-line matrix-valued Schr\"odinger operators was dealt with in \cite{GS00a}.

We should also mention some work of Ramm \cite{Ra99}, \cite{Ra00}, who provided a  proof of Theorem~\ref{t1.2}\,$(i)$ under the additional assumption that $V_j$ are
short-range potentials satisfying $V_j\in L^1([0,\infty); (1+|x|)dx)$, $j=1,2$.
A very short proof of Theorem \ref{t1.2}, close in spirit to Borg's original paper
\cite{Bo52}, was subsequently found by Bennewitz \cite{Be01}. Still other proofs were presented in \cite{Ho01} and \cite{Kn01}. Various local and global uniqueness results for matrix-valued Schr\"odinger, Dirac-type, and Jacobi operators were considered in
\cite{CG02}, \cite{FKRS07}, \cite{GKM02}, \cite{Sa90}, \cite{Sa02}, and \cite{Sa06}. A local
Borg--Marchenko theorem for complex-valued potentials has been proved in
\cite{BPW02}; the case of semi-infinite Jacobi operators with complex-valued coefficients was studied in \cite{We04}. This circle of ideas has been reviewed in \cite{Ge07}.

After this review of Borg--Marchenko-type uniqueness results for Schr\"odinger operators, we now turn to the principal object of our interest in this paper, the so-called CMV operators. CMV operators are a special class of unitary semi-infinite five-diagonal matrices. But for simplicity, we confine ourselves in this introduction to a discussion of CMV operators on $\bbZ$, that is, doubly infinite CMV operators. Let $\alpha$ be a sequence of $m\times m$ matrices, $m\in\bbN$, with entries in $\bbC$,
$\alpha=\{\al_k\}_{k \in \Z}$ such that $\|\alpha_k \|_{\bbC^{m\times m}} < 1$,
$k\in\bbZ$. The unitary operator $\U$ on $\ltm{\Z}$ then can be written as a special five-diagonal doubly infinite matrix in the standard basis of $\ell^2(\bbZ)^m$ as in 
\eqref{3.18}. For the corresponding half-lattice CMV operators $\U_{+,k_0}$, in 
$\ell^2([k_0,\infty)\cap\bbZ)^m$ we refer to \eqref{3.33} and \eqref{3.34}.

The actual history of CMV operators (with scalar coefficients $\alpha_k \in\bbC$, $k\in\bbZ$) is quite interesting: The corresponding unitary semi-infinite five-diagonal matrices were first introduced in 1991 by
Bunse--Gerstner and Elsner \cite{BGE91}, and subsequently discussed in detail by Watkins \cite{Wa93} in 1993 (cf.\ the recent discussion in Simon \cite{Si06}). They were subsequently rediscovered by Cantero, Moral, and Vel\'azquez (CMV) in \cite{CMV03}. In \cite[Sects.\ 4.5, 10.5]{Si04}, Simon introduced the corresponding notion of unitary doubly infinite five-diagonal matrices and coined the term ``extended'' CMV matrices. For simplicity, we will just speak of CMV operators whether or not they are half-lattice or full-lattice operators. We also note that in a context different from orthogonal polynomials on the unit circle, Bourget, Howland, and Joye \cite{BHJ03} introduced a family of doubly infinite matrices with three sets of parameters which, for special choices of the parameters, reduces to two-sided CMV matrices on $\bbZ$. Moreover, it is possible to connect unitary block Jacobi matrices to the trigonometric moment problem (and hence to CMV matrices) as discussed by Berezansky and Dudkin \cite{BD05}, \cite{BD06}.

The relevance of this unitary operator $\U$ on $\ell^2(\bbZ)^m$, more precisely, the relevance of the corresponding half-lattice CMV operator $\U_{+,0}$ in
$\ell^2(\bbN_0)^m$ is derived from its intimate relationship with the trigonometric moment problem and hence with finite measures on the unit circle $\dD$. (Here
$\bbN_0=\bbN\cup\{0\}$.)
This will be reviewed in some detail in Section \ref{s3} but we also refer to the
monumental two-volume treatise by Simon \cite{Si04} (see also \cite{Si04b} and
\cite{Si05}) and the exhaustive bibliography therein. For classical results on orthogonal polynomials on the unit circle we refer, for instance, to \cite{Ak65},
\cite{Ge46}--\cite{Ge61}, \cite{Kr45}, \cite{Sz20}--\cite{Sz78},
\cite{Ve35}, \cite{Ve36}. More recent references relevant to the spectral theoretic content of this paper are \cite{De07}, \cite{GJ96}--\cite{GT94},
\cite{GZ06}, \cite{GZ06a}, \cite{GN01}, \cite{PY04}, and \cite{Si04a}. The full-lattice CMV operators $\U$ on $\bbZ$ are closely related to an important, and only recently intensively studied, completely integrable nonabelian version of the defocusing nonlinear
Schr\"odinger equation (continuous in time but discrete in space), a special case of the Ablowitz--Ladik system. Relevant references in this context are, for instance,
\cite{AL75}--\cite{APT04}, \cite{GGH05}, \cite{GH05}--\cite{GHMT07a}, \cite{Li05},
\cite{MEKL95}--\cite{Ne06}, \cite{Sc89}, \cite{Ve99}, and the literature cited therein. We emphasize that the case of matrix-valued coefficients $\alpha_k$ is considerably less studied than the case of scalar coefficients.

We note that our discussion of CMV operators will be undertaken in the spirit of
\cite{GKM02}, where (local and global) uniqueness theorems for full-line (resp.,
full-lattice) problems are  formulated in terms of diagonal Green's matrices $g(z,x_0)$ and their $x$-derivatives $g^\prime (z,x_0)$ at some fixed $x_0\in\bbR$, for matrix-valued Schr\"odinger and Dirac-type operators on $\bbR$ and similarly for matrix-valued Jacobi operators on $\bbZ$. While we prove half-lattice and full-latice uniqueness results in our principal Section \ref{s5}, we now confine ourselves in this introduction to just two typical results for CMV operators on $\bbZ$ with matrix-valued coefficients:

We use the following notation for the diagonal and for the neighboring off-diagonal entries of the Green's matrix of $\U$ (i.e., the discrete integral kernel of $(\U-zI)^{-1}$),
\begin{align}
g(z,k) = (\U-Iz)^{-1}(k,k),   \quad
h(z,k) = \begin{cases}
(\U-Iz)^{-1}(k-1,k), & k \text{ odd}, \\
(\U-Iz)^{-1}(k,k-1), & k \text{ even},
\end{cases}\quad k\in\Z,\; z\in\D.      \lb{1.9}
\end{align}

The next uniqueness result then holds for the full-lattice CMV operator $\U$.

\begin{theorem}  \lb{t1.3}
Let $m\in\N$ and assume $\al=\{\al_k\}_{k\in\Z}$ be a sequence of
$m \times m$ matrices with complex entries such that
$\norm{\al_k}_{\Cm} < 1$ and let $k_0\in\Z$. Then any of the
following two sets of data
\begin{enumerate}[$(i)$]
\item $g(z,k_0)$ and $h(z,k_0)$ for all $z$ in a sufficiently small neighborhood
of the origin under the assumption that $h(0,k_0)$ is invertible;
\item $g(z,k_0-1)$ and $g(z,k_0)$ for all $z$ in a sufficiently small neighborhood
of the origin and $\al_{k_0}$ under the assumption $\al_{k_0}$ is invertible;
\end{enumerate}
uniquely determine the matrix-valued Verblunsky coefficients $\{\al_k\}_{k\in\Z}$, and
hence the full-lattice CMV operator $\U$ defined in \eqref{3.18}.
\end{theorem}

In the subsequent local uniqueness result, $g^{(j)}$ and $h^{(j)}$ denote the
corresponding quantities in \eqref{1.9} associated with the
matrix-valued Verblunsky coefficients $\alpha^{(j)}$, $j=1,2$.

\begin{theorem}  \lb{t1.4}
Let $m\in\N$ and assume $\al^{(\ell)}=\{\al_k^{(\ell)}\}_{k\in\Z}$ be sequences of
$m \times m$ matrices with complex entries such that
$\norm{\al_k^{(\ell)}}_{\Cm} < 1$, $k\in\Z$, $\ell=1,2$. Moreover, assume
$k_0\in\Z$, $N\in\N$. Then for the full-lattice
problems associated with $\al^{(1)}$ and $\al^{(2)}$ the following
local uniqueness results hold:
\begin{enumerate}[$(i)$]
\item
If either $h^{(1)}(0,k_0)$ or $h^{(2)}(0,k_0)$ is invertible and
\begin{align}
\begin{split}
&\big\|g^{(1)}(z,k_0)-g^{(2)}(z,k_0)\big\|_\Cm
+ \big\|h^{(1)}(z,k_0)-h^{(2)}(z,k_0)\big\|_\Cm \underset{z\to 0}{=} \oh(z^N), \lb{1.10} \\
& \, \text{then } \, \al^{(1)}_k = \al^{(2)}_k \,\text{ for }\,
k_0-N \leq k\leq k_0+N+1.
\end{split}
\end{align}
\item
If $\al^{(1)}_{k_0}=\al^{(2)}_{k_0}$, $\al^{(1)}_{k_0}$ is
invertible, and
\begin{align}
\begin{split}
&\big\|g^{(1)}(z,k_0-1)-g^{(2)}(z,k_0-1)\big\|_\Cm +
\big\|g^{(1)}(z,k_0)-g^{(2)}(z,k_0)\big\|_\Cm \underset{z\to 0}{=} \oh(z^N), \lb{1.11}  \\
& \, \text{then } \, \al^{(1)}_k = \al^{(2)}_k \,\text{ for }\,
k_0-N-1 \leq k\leq k_0+N+1.
\end{split}
\end{align}
\end{enumerate}
\end{theorem}

The special case of CMV operators with scalar Verblunsky coefficients has recently been discussed in \cite{CGZ07}.

Finally, a brief description of the content of each section in this paper: In Section \ref{s3} we develop the basic Weyl--Titchmarsh theory for half-lattice CMV operators with matrix-valued Verblunsky coefficients. The analogous theory for full-line CMV operators is developed in Section \ref{s4}. Weyl--Titchmarsh theory for CMV operators with matrix-valued Verblunsky coefficients is a subject of independent interest and of fundamental importance in the remainder of this paper. Section \ref{s5}  contains our new Borg--Marchenko-type uniqueness results for half-lattice and full-lattice CMV operators with matrix-valued Verblunsky coefficients. Appendix \ref{sA} summarizes basic facts on matrix-valued Caratheodory and Schur functions relevant to this paper.

\section{Weyl--Titchmarsh Theory for Half-Lattice CMV Operators \\ with
Matrix-Valued Verblunsky Coefficients}
\lb{s3}

In this section we present the basics of Weyl--Titchmarsh theory for half-lattice CMV operators with matrix-valued Verblunsky coefficients. We closely follow the corresponding treatment of scalar-valued Verblunsky coefficients in \cite{GZ06}.

We should note that while there is an extensive literature on orthogonal matrix-valued polynomials on the real line and on the unit circle, we refer, for instance, to
\cite{AN84}, \cite{BC92}, \cite[Ch.\ VII]{Be68}, \cite{BG90}, \cite{CFMV03}, \cite{CG06},
\cite{DG92}--\cite{DV95}, \cite{Ge81}, \cite{Ge82}, \cite{Kr49}, \cite{Kr71}, \cite{Le47},
\cite{Lo99}, \cite{Os97}--\cite{Os02}, \cite{Ro90}, \cite{YM01}--\cite{YK78}, and the literature therein, the case of CMV operators with matrix-valued Verblunsky coefficients appears to be a much less explored frontier. The only references we are aware of in this context are Simon's treatise \cite[Part 1, Sect.\ 2.13]{Si04} and a recent preprint by Simon \cite{Si06}.

In the remainder of this paper, $\Cm$ denotes the space of $m\times m$ matrices with complex-valued entries endowed with the operator norm $\norm{\cdot}_{\Cm}$ (we use the standard Euclidean norm in $\bbC^m$). The adjoint of an element $\ga\in\Cm$ is denoted by $\ga^*$, and the real and imaginary parts of $\ga$ are defined as usual by
$\Re(\ga)=(\ga+\ga^*)/2$ and $\Im(\ga) =(\ga-\ga^*)/(2i)$.

\begin{remark} \lb{r3.0} 
For simplicity of exposition, we find it convenient to use the
following conventions: We denote by $\s{\Z}$ the vector space of all
$\C$-valued sequences, and by $\sm{\Z}=\s{\Z}\otimes\C^m$ the vector
space of all $\C^m$-valued sequences; that is,
\begin{align} \lb{3.5A}
\phi=\{\phi(k)\}_{k\in\Z}=
\begin{pmatrix}
\vdots\\\phi(-1)\\\phi(0)\\\phi(1)\\\vdots
\end{pmatrix}\in\sm{\Z}, \quad 
\phi(k)=
\begin{pmatrix}
(\phi(k))_1\\(\phi(k))_2\\\vdots\\(\phi(k))_m
\end{pmatrix}\in\C^m, \;  k\in\Z.
\end{align}
Moreover, we introduce $\smn{\Z}=\sm{\Z}\otimes\C^n$, $m,n\in\N$, that is,
$\Phi=(\phi_1,\dots,\phi_n)\in\smn{\Z}$, where $\phi_j\in\sm{\Z}$
for all $j=1,\dots,n$.

We also note that $\smn{\Z}=\s{\Z}\otimes\C^{m\times n}$, $m,n\in\N$;
which is to say that the elements of $\smn{\Z}$ can be
identified with the $\C^{m\times n}$-valued sequences,
\begin{align} \lb{3.6A}
\Phi=\{\Phi(k)\}_{k\in\Z}=
\begin{pmatrix}
\vdots\\\Phi(-1)\\\Phi(0)\\\Phi(1)\\\vdots
\end{pmatrix},
\quad \Phi(k)=
\begin{pmatrix}
(\Phi(k))_{1,1} & \hdots & (\Phi(k))_{1,n}\\
\vdots & & \vdots \\
(\Phi(k))_{m,1} & \hdots & (\Phi(k))_{m,n}
\end{pmatrix}\in\C^{m\times n}, \; k\in\Z,
\end{align}
by setting $\Phi=(\phi_1,\dots,\phi_n)$, where
\begin{align} \lb{3.7A}
\phi_j=
\begin{pmatrix}
\vdots\\\phi_j(-1)\\\phi_j(0)\\\phi_j(1)\\\vdots
\end{pmatrix}\in\sm{\Z}, \quad \phi_j(k)=
\begin{pmatrix}
(\Phi(k))_{1,j}\\\vdots\\(\Phi(k))_{m,j}
\end{pmatrix}\in\C^m, \; j=1,\dots,n,\; k\in\Z.
\end{align}

For the elements of $\smn{\Z}$ we define the right-multiplication by
$n\times n$ matrices with complex-valued entries by
\begin{align} \lb{3.3A}
\Phi C = (\phi_1,\dots,\phi_n)
\begin{pmatrix}c_{1,1} & \dots & c_{1,n}\\
\vdots&&\vdots\\c_{n,1} & \dots & c_{n,n}\end{pmatrix} =
\left(\sum_{j=1}^n \phi_j c_{j,1},\dots,\sum_{j=1}^n \phi_j
c_{j,n}\right)\in\smn{\Z}
\end{align}
for all $\Phi\in\smn{\Z}$ and $C\in\C^{n\times n}$. In addition, for
any linear transformation $\bbA: \sm{\Z}\to\sm{\Z}$,  we define
$\bbA\Phi$ for all $\Phi=(\phi_1,\dots,\phi_n)\in\smn{\Z}$ by
\begin{align} \lb{3.4A}
\bbA\Phi = (\bbA\phi_1,\dots,\bbA\phi_n)\in\smn{\Z}.
\end{align}

Given the above conventions, we note the subspace
containment: $\ltm{\Z}=\ell^2(\Z)\otimes\C^m\subset\sm{\Z}$ and
$\ltmn{\Z}=\ell^2(\Z)\otimes\C^{m\times n}\subset\smn{\Z}$. We also
note that $\ltm{\Z}$ represents a Hilbert space with scalar
product given by
\begin{align} \lb{3.2A}
(\phi,\psi)_{\ltm{\Z}} = \sum_{k=-\infty}^\infty\sum_{j=1}^m
\ol{(\phi(k))_j}(\psi(k))_j, \quad \phi,\psi\in\ltm{\Z}.
\end{align}
Finally, we note that a straightforward modification of the
above definitions also yields the Hilbert space $\ltm{J}$ as well
as the sets $\ltmn{J}$, $\sm{J}$, and $\smn{J}$ for any
$J\subset\Z$.
\end{remark}

We start by introducing our basic assumption:

\begin{hypothesis} \lb{h3.1}
Let $m\in\N$ and assume $\al=\{\al_k\}_{k\in\Z}$ is a sequence of
$m \times m$ matrices with complex entries\footnote{We emphasize that $\al_k\in\Cm$, $k\in\Z$, are general (not necessarily normal) matrices.}
and such that
\begin{equation} \lb{3.7}
\norm{\al_k}_{\Cm} < 1, \quad k\in\Z.
\end{equation}
\end{hypothesis}

Given a sequence $\al$ satisfying \eqref{3.7}, we define two sequences
of positive self-adjoint $m\times m$ matrices $\{\rho_k\}_{k\in\bbZ}$
and $\{\wti\rho_k\}_{k\in\bbZ}$ by
\begin{align}
\rho_k &= \sqrt{I_m-\al_k^*\al_k}, \quad k\in\bbZ, \lb{3.8}
\\
\wti\rho_k &= \sqrt{I_m-\al_k\al_k^*}, \quad k\in\bbZ, \lb{3.9}
\end{align}
and two sequences of $m\times m$ matrices with positive real parts,
$\{a_k\}_{k\in\bbZ}\subset \bbC^{m\times m}$ and
$\{b_k\}_{k\in\bbZ}\subset \bbC^{m\times m}$ by
\begin{align}
a_k &= I_m+\al_k, \quad k\in\bbZ, \lb{3.10}
\\
b_k &= I_m-\al_k, \quad k \in \Z. \lb{3.11}
\end{align}
Then \eqref{3.7} implies that $\rho_k$ and $\wti\rho_k$ are
invertible matrices for all $k\in\Z$, and using elementary power series
expansions one verifies the following identities for all $k\in\Z$,
\begin{align}
&\wti\rho_k^{\pm1}\al_k = \al_k\rho_k^{\pm1}, \quad
\al_k^*\wti\rho_k^{\pm1} = \rho_k^{\pm1}\al_k^*, \lb{3.12}
 \\
&a_k^*\wti\rho_k^{-2}a_k = a_k\rho_k^{-2}a_k^*, \quad
b_k^*\wti\rho_k^{-2}b_k = b_k\rho_k^{-2}b_k^*, \quad
a_k^*\wti\rho_k^{-2}b_k + a_k\rho_k^{-2}b_k^* =
b_k^*\wti\rho_k^{-2}a_k + b_k\rho_k^{-2}a_k^* = 2I_m. \lb{3.13}
\end{align}

According to Simon \cite{Si04}, we call $\al_k$ the Verblunsky
coefficients in honor of Verblunsky's pioneering work in the theory
of orthogonal polynomials on the unit circle \cite{Ve35},
\cite{Ve36}.

Next, we introduce a sequence of $2\times 2$ block unitary matrices $\Te_k$
with $m\times m$ matrix coefficients by
\begin{equation} \lb{3.14}
\Te_k = \begin{pmatrix} -\al_k & \wti\rho_k \\ \rho_k & \al_k^*
\end{pmatrix},
\quad k \in \Z,
\end{equation}
and two unitary operators $\V$ and $\W$ on $\ltm{\Z}$ by their
matrix representations in the standard basis of $\ltm{\Z}$ by 
\begin{align} \lb{3.15}
\V &= \begin{pmatrix} \ddots & & &
\raisebox{-3mm}[0mm][0mm]{\hspace*{-5mm}\Huge $0$}  \\ & \Te_{2k-2} &
& \\ & & \Te_{2k} & & \\ &
\raisebox{0mm}[0mm][0mm]{\hspace*{-10mm}\Huge $0$} & & \ddots
\end{pmatrix}, \quad
\W = \begin{pmatrix} \ddots & & &
\raisebox{-3mm}[0mm][0mm]{\hspace*{-5mm}\Huge $0$}
\\ & \Te_{2k-1} &  &  \\ &  & \Te_{2k+1} &  & \\ &
\raisebox{0mm}[0mm][0mm]{\hspace*{-10mm}\Huge $0$} & & \ddots
\end{pmatrix},
\end{align}
where
\begin{align}
\begin{pmatrix}
\V_{2k-1,2k-1} & \V_{2k-1,2k} \\ \V_{2k,2k-1}   & \V_{2k,2k}
\end{pmatrix} =  \Te_{2k},
\quad
\begin{pmatrix}
\W_{2k,2k} & \W_{2k,2k+1} \\ \W_{2k+1,2k}  & \W_{2k+1,2k+1}
\end{pmatrix} =  \Te_{2k+1},
\quad k\in\Z. \lb{3.16}
\end{align}
Moreover, we introduce the unitary operator $\U$ on
$\ltm{\Z}$ as the product of the unitary operators $\V$ and $\W$ by
\begin{equation} \lb{3.17}
\U = \V\W,
\end{equation}
or in matrix form in the standard basis of $\ltm{\Z}$, by
\begin{align}
\U = \begin{pmatrix} \ddots &&\hspace*{-8mm}\ddots
&\hspace*{-10mm}\ddots &\hspace*{-12mm}\ddots &\hspace*{-14mm}\ddots
&&& \raisebox{-3mm}[0mm][0mm]{\hspace*{-6mm}{\Huge $0$}}
\\
&0& -\al_{0}\rho_{-1} & -\al_{0}\al_{-1}^* & -\wti\rho_{0}\al_{1} &
\wti\rho_{0}\wti\rho_{1}
\\
&& \rho_{0}\rho_{-1} &\rho_{0}\al_{-1}^* & -\al_{0}^*\al_{1} &
\al_{0}^*\wti\rho_{1} & 0
\\
&&&0& -\al_{2}\rho_{1} & -\al_{2}\al_{1}^* & -\wti\rho_{2}\al_{3} &
\wti\rho_{2}\wti\rho_{3}
\\
&&\raisebox{-4mm}[0mm][0mm]{\hspace*{-6mm}{\Huge $0$}} &&
\rho_{2}\rho_{1} & \rho_{2}\al_{1}^* & -\al_{2}^*\al_{3} &
\al_{2}^*\wti\rho_{3}&0
\\
&&&&&\hspace*{-14mm}\ddots &\hspace*{-14mm}\ddots
&\hspace*{-14mm}\ddots &\hspace*{-8mm}\ddots &\ddots
\end{pmatrix}. \lb{3.18}
\end{align}
Here terms of the form $-\al_{2k}\al_{2k-1}^*$ and
$-\al_{2k}^*\al_{2k+1}$, $k\in\Z$, represent the diagonal entries
$\U_{2k-1,2k-1}$ and $\U_{2k,2k}$ of the infinite matrix $\U$ in
\eqref{3.18}, respectively. We continue to call the operator $\U$ on
$\ltm{\Z}$ the CMV operator since \eqref{3.14}--\eqref{3.18} in the
context of the scalar-valued semi-infinite (i.e., half-lattice) case
were obtained by Cantero, Moral, and Vel\'azquez in \cite{CMV03} in
2003, but we refer to the discussion in the introduction about the
involved history of these operators.

\begin{lemma} \lb{l3.2}
Let $z\in\bbC\backslash\{0\}$ and $\{U(z,k)\}_{k\in\bbZ},
\{V(z,k)\}_{k\in\bbZ}$ be two $\Cm$-valued sequences. Then the
following items $(i)$--$(iii)$ are equivalent:
\begin{align}
(i) & \quad  (\U U(z,\cdot))(k) = z U(z,k), \quad (\W U(z,\cdot))(k)=z
V(z,k), \quad k\in\Z. \lb{3.19}
\\
(ii) & \quad (\W U(z,\cdot))(k) = z V(z,k), \quad (\V V(z,\cdot))(k) =
U(z,k), \quad k\in\Z. \lb{3.20}
\\
(iii) & \quad \binom{U(z,k)}{V(z,k)} = \T(z,k)
\binom{U(z,k-1)}{V(z,k-1)}, \quad k\in\Z.  \lb{3.21}
\end{align}
Here $\U$, $\V$, and $\W$  are understood in the sense of difference
expressions on $\smm{\Z}$ rather than difference operators on
$\ltm{\Z}$ $($cf.\ Remark \ref{r3.0}$)$ and the transfer matrices
$\T(z,k)$, $z\in\Cz$, $k\in\Z$, are defined by
\begin{equation}
\T(z,k) = \begin{cases}
\begin{pmatrix}
\wti\rho_{k}^{-1}\al_{k} & z\wti\rho_{k}^{-1} \\
z^{-1}\rho_{k}^{-1} & \rho_{k}^{-1}\al_{k}^*
\end{pmatrix},  & \text{$k$ odd,}
\\
\begin{pmatrix}
\rho_{k}^{-1}\al_{k}^* & \rho_{k}^{-1} \\
\wti\rho_{k}^{-1} & \wti\rho_{k}^{-1}\al_{k}
\end{pmatrix}, & \text{$k$ even.}
\end{cases} \lb{3.22}
\end{equation}
\end{lemma}
\begin{proof}
The equivalence of \eqref{3.19} and \eqref{3.20} is a consequence of 
\eqref{3.17} and equivalence of \eqref{3.20} and \eqref{3.21} is implied by 
the following computations: 

Assuming $k$ to be odd and utilizing \eqref{3.8}, \eqref{3.9}, and
\eqref{3.12}, one verifies equivalence of the following items
$(i)$--$(v)$:
\begin{align}
 (i) & \quad  \binom{U(z,k)}{V(z,k)} = \T(z,k)
\binom{U(z,k-1)}{V(z,k-1)}.
\\
 (ii) & \quad
\begin{cases}
\wti\rho_k U(z,k) = \al_k U(z,k-1) + zV(z,k-1), \\
\rho_k zV(z,k) = U(z,k-1) + \al_k^* zV(z,k-1).
\end{cases}
\\
 (iii) & \quad
\begin{cases}
zV(z,k-1) = - \al_k U(z,k-1) + \wti\rho_k U(z,k), \\
\rho_k zV(z,k) = U(z,k-1) + \al_k^*\big(-\al_k U(z,k-1) + \wti\rho_k
U(z,k)\big).
\end{cases}
\\
 (iv) & \quad
\begin{cases}
zV(z,k-1) = - \al_k U(z,k-1) + \wti\rho_k U(z,k), \\
\rho_k zV(z,k) = \rho_k^2 U(z,k-1) + \rho_k\al_k^* U(z,k).
\end{cases}
\\
 (v) & \quad z \binom{V(z,k-1)}{V(z,k)} = \Te_k
\binom{U(z,k-1)}{U(z,k)}.
\intertext{Similarly, assuming $k$ to be even, one verifies that the items
$(vi)$--$(viii)$ are equivalent:}
 (vi) & \quad \binom{U(z,k)}{V(z,k)} = \T(z,k)
\binom{U(z,k-1)}{V(z,k-1)}.
\\
 (vii) & \quad
\begin{cases}
\wti\rho_k V(z,k) = \al_k V(z,k-1) + U(z,k-1), \\
\rho_k U(z,k) = V(z,k-1) + \al_k^* U(z,k-1).
\end{cases}
\\
 (viii) & \quad \binom{U(z,k-1)}{U(z,k)} = \Te_k
\binom{V(z,k-1)}{V(z,k)}.
\end{align}
Finally, taking into account \eqref{3.15} and \eqref{3.16}, one
concludes that
\begin{align}
\Te_{2k+1} \binom{U(z,2k)}{U(z,2k+1)} = z \binom{V(z,2k)}{V(z,2k+1)},
\quad \Te_{2k} \binom{V(z,2k-1)}{V(z,2k)} = z
\binom{U(z,2k-1)}{U(z,2k)}, \quad k\in\Z
\end{align}
is equivalent to
\begin{align}
(\W U(z,\cdot))(k) = z V(z,k), \quad (\V V(z,\cdot))(k) = U(z,k),
\quad k\in\Z.
\end{align}
\end{proof}

We note that in studying solutions of $(\U U(z,\cdot))(k)=zU(z,k)$
as in Lemma \ref{l3.2}\,$(i)$, the purpose of the additional
relation $(\W U(z,\cdot))(k)=zV(z,k)$ in \eqref{3.19} is to
introduce a new variable $V$ that improves our understanding of the
structure of such solutions $U$.

If one sets $\al_{k_0} = I_m$ for some reference point $k_0\in\Z$,
then the operator $\U$ splits into a direct sum of two half-lattice
operators $\U_{-,k_0-1}$ and $\U_{+,k_0}$ acting on
$\ltm{(-\infty,k_0-1]\cap\Z}$ and on $\ltm{[k_0,\infty)\cap\Z}$,
respectively. Explicitly, one obtains
\begin{align}
\U = \U_{-,k_0-1} \oplus \U_{+,k_0} \, \text{ in } \,
\ltm{(-\infty,k_0-1]\cap\Z} \oplus \ltm{[k_0,\infty)\cap\Z}.   \lb{3.33}
\end{align}
(Strictly, speaking, setting $\al_{k_0} = I_m$ for some reference
point $k_0\in\Z$ contradicts our basic Hypothesis \ref{h3.1}.
However, as long as the exception to Hypothesis \ref{h3.1} refers to
only one site, we will safely ignore this inconsistency in favor of
the notational simplicity it provides by avoiding the introduction of
a properly modified hypothesis on $\{\al_k\}_{k\in\bbZ}$.) Similarly,
one obtains $\W_{-,k_0-1}$, $\V_{-,k_0-1}$ and $\W_{+,k_0}$,
$\V_{+,k_0}$ such that
\begin{equation}
\U_{\pm,k_0} = \V_{\pm,k_0} \W_{\pm,k_0}. \lb{3.34}
\end{equation}

\begin{lemma} \lb{l3.3}
Let $z\in\bbC\backslash\{0\}$, $k_0\in\bbZ$, and
$\{\hatt P_+(z,k,k_0)\}_{k\geq k_0}$, $\{\hatt R_+(z,k,k_0)\}_{k\geq
k_0}$ be two
$\Cm$-valued sequences. Then the following items $(i)$--$(vi)$ are
equivalent:
\begin{align}
(i)& \quad (\U_{+,k_0} \hatt P_+(z,\cdot,k_0))(k) = z \hatt
P_+(z,k,k_0), \quad
(\W_{+,k_0} \hatt P_+(z,\cdot,k_0))(k) = z \hatt R_+(z,k,k_0), \quad
k\geq k_0.  \lb{3.35}\\
(ii)& \quad (\W_{+,k_0} \hatt P_+(z,\cdot,k_0))(k) = z \hatt
R_+(z,k,k_0), \quad
(\V_{+,k_0} \hatt R_+(z,\cdot,k_0))(k) = \hatt P_+(z,k,k_0), \quad k\geq
k_0.
\lb{3.36}
\\
(iii)& \quad \binom{\hatt P_+(z,k,k_0)}{\hatt R_+(z,k,k_0)} = \T(z,k)
\binom{\hatt P_+(z,k-1,k_0)}{\hatt R_+(z,k-1,k_0)}, \quad k > k_0,
\notag \\
& \hspace*{11mm}  \text{with initial condition } \,
\hatt P_+(z,k_0,k_0) =
\begin{cases} z \hatt R_+(z,k_0,k_0), & \text{$k_0$ odd}, \\
\hatt R_+(z,k_0,k_0), & \text{$k_0$ even}. \end{cases}\lb{3.37}
\end{align}

Next, consider $\Cm$-valued sequences $\{\hatt P_-(z,k,k_0)\}_{k\leq
k_0}$,
$\{\hatt R_-(z,k,k_0)\}_{k\leq k_0}$. Then the following items
$(iv)$--$(vi)$ are equivalent:
\begin{align}
(iv)& \quad (\U_{-,k_0} \hatt P_-(z,\cdot,k_0))(k) = z \hatt
P_-(z,k,k_0), \quad
(\W_{-,k_0} \hatt P_-(z,\cdot,k_0))(k) = z \hatt R_-(z,k,k_0), \quad
k\leq k_0.
\\
(v)& \quad (\W_{-,k_0} \hatt P_-(z,\cdot,k_0))(k) = z \hatt
R_-(z,k,k_0),
\quad (\V_{-,k_0} \hatt R_-(z,\cdot,k_0))(k) = \hatt P_-(z,k,k_0), \quad
k\leq
k_0.
\\
(vi)& \quad \binom{\hatt P_-(z,k-1),k_0}{\hatt R_-(z,k-1,k_0)}  =
\T(z,k)^{-1}
\binom{\hatt P_-(z,k,k_0)}{\hatt R_-(z,k,k_0)}, \quad k \leq k_0, \notag
\\
& \hspace*{11mm}   \text{with initial condition } \,
\hatt  P_-(z,k_0,k_0) =\begin{cases} - \hatt R_-(z,k_0,k_0), &
\text{$k_0$ odd,}
\\ -z \hatt R_-(z,k_0,k_0), & \text{$k_0$ even.} \end{cases}\lb{3.40}
\end{align}

Here $\U_{\pm,k_0}$, $\V_{\pm,k_0}$, and $\W_{\pm,k_0}$  are
understood in the sense of difference expressions on the set
$\smm{\Z\cap[k_0,\pm\infty)}$ rather than difference operators on
$\ltm{\Z\cap[k_0,\pm\infty)}$ $($cf.\ Remark \ref{r3.0}$)$.
\end{lemma}
\begin{proof}
Equivalence of \eqref{3.35} and \eqref{3.36} is a consequence of 
\eqref{3.34}.

Next, repeating the proof of Lemma \ref{l3.2} one obtains that
\begin{align}
(\W_{+,k_0} \hatt P_+(z,\cdot,k_0))(k) &= z \hatt R_+(z,k,k_0), \quad (\V_{+,k_0}
\hatt R_+(z,\cdot,k_0))(k) = \hatt P_+(z,k,k_0), \quad k > k_0,  \lb{3.41}
\end{align}
is equivalent to
\begin{align}
\binom{\hatt P_+(z,k,k_0)}{\hatt R_+(z,k,k_0)} &=  \T(z,k)
\binom{\hatt P_+(z,k-1,k_0)}{\hatt R_+(z,k-1,k_0)}, \quad k > k_0. \lb{3.42}
\end{align}
Moreover, in the case $k_0$ is odd, the matrices $\V_{+,k_0}$ and
$\W_{+,k_0}$ have the structure,
\begin{align}
\V_{+,k_0} =
\begin{pmatrix}
\Te_{k_0+1} & & \raisebox{-1mm}[0mm][0mm]{\hspace*{-2mm}{\huge $0$}}
\\ & \Te_{k_0+3} & \\
\raisebox{-1mm}[0mm][0mm]{\hspace*{0mm}{\huge $0$}} && \ddots
\end{pmatrix},
\quad \W_{+,k_0} =
\begin{pmatrix}
I_m & & \raisebox{-1mm}[0mm][0mm]{\hspace*{-2mm}{\huge $0$}}
\\ & \Te_{k_0+2} & \\ &
\raisebox{-1mm}[0mm][0mm]{\hspace*{-14mm}{\huge $0$}} &\ddots
\end{pmatrix},
\end{align}
and hence,
\begin{equation}
(\W_{+,k_0} \hatt P_+(z,\cdot,k_0))(k_0) = z \hatt R_+(z,k_0,k_0) \lb{3.44}
\end{equation}
is equivalent to
\begin{equation}
\hatt P_+(z,k_0,k_0) = z \hatt R_+(z,k_0,k_0). \lb{3.45}
\end{equation}
In the case $k_0$ is even, the matrices $V_{+,k_0}$ and $W_{+,k_0}$
have the structure,
\begin{align}
\V_{+,k_0} =
\begin{pmatrix}
I_m & & \raisebox{-1mm}[0mm][0mm]{\hspace*{-2mm}{\huge $0$}}
\\ & \Te_{k_0+2} & \\ &
\raisebox{-1mm}[0mm][0mm]{\hspace*{-14mm}{\huge $0$}} &\ddots
\end{pmatrix},
\quad \W_{+,k_0} =
\begin{pmatrix}
\Te_{k_0+1} & & \raisebox{-1mm}[0mm][0mm]{\hspace*{-2mm}{\huge $0$}}
\\ &
\Te_{k_0+3} & \\
\raisebox{-1mm}[0mm][0mm]{\hspace*{0mm}{\huge $0$}} & & \ddots
\end{pmatrix},
\end{align}
and hence,
\begin{equation}
(\V_{+,k_0} \hatt R_+(z,\cdot,k_0))(k_0) = \hatt P_+(z,k_0,k_0) \lb{3.47}
\end{equation}
is equivalent to
\begin{equation}
\hatt P_+(z,k_0,k_0) = \hatt R_+(z,k_0,k_0). \lb{3.48}
\end{equation}
Thus, one infers the equivalence of \eqref{3.36} and \eqref{3.37}
from the equivalence of \eqref{3.41} and \eqref{3.42} with
\eqref{3.44}--\eqref{3.45} and \eqref{3.47}--\eqref{3.48}.

The results for $\{\hatt P_-(z,k,k_0)\}_{k\leq k_0}$ and
$\{\hatt R_-(z,k,k_0)\}_{k\leq k_0}$ are proved analogously.
\end{proof}

Analogous comments to those made right after the proof of Lemma
\ref{l3.2} apply in the present context of Lemma \ref{l3.3}.

Next, we denote by
$\Big(\begin{smallmatrix}P_\pm(z,k,k_0) \\
R_\pm(z,k,k_0)\end{smallmatrix}\Big)_{k\in\Z}$ and
$\Big(\begin{smallmatrix} Q_\pm(z,k,k_0) \\ S_\pm(z,k,k_0)
\end{smallmatrix}\Big)_{k\in\Z}$,
$z\in\bbC\backslash\{0\}$, four linearly independent solutions of
\eqref{3.21} satisfying the following initial conditions:
\begin{align}
\binom{P_+(z,k_0,k_0)}{R_+(z,k_0,k_0)} &=
\begin{cases}
\binom{zI_m}{I_m}, & \text{$k_0$ odd,} \\[1mm]
\binom{I_m}{I_m}, & \text{$k_0$ even,}
\end{cases} \quad\;\;
\binom{Q_+(z,k_0,k_0)}{S_+(z,k_0,k_0)} =
\begin{cases}
\binom{zI_m}{-I_m}, & \text{$k_0$ odd,} \\[1mm]
\binom{-I_m}{I_m}, & \text{$k_0$ even.}
\end{cases} \lb{3.49}
\\
\binom{P_-(z,k_0,k_0)}{R_-(z,k_0,k_0)} &=
\begin{cases}
\binom{I_m}{-I_m}, & \text{$k_0$ odd,} \\[1mm]
\binom{-zI_m}{I_m}, & \text{$k_0$ even,}
\end{cases} \quad
\binom{Q_-(z,k_0,k_0)}{S_-(z,k_0,k_0)} =
\begin{cases}
\binom{I_m}{I_m}, & \text{$k_0$ odd, }\\[1mm]
\binom{zI_m}{I_m}, & \text{$k_0$ even.}
\end{cases} \lb{3.50}
\end{align}
Then $P_\pm(z,k,k_0)$, $Q_\pm(z,k,k_0)$,
$R_\pm(z,k,k_0)$, and $S_\pm(z,k,k_0)$, $k,k_0\in\Z$, are
$\Cm$-valued Laurent polynomials in $z$. In particular, one computes

\begin{align}\lb{Table}
\begin{array}{|c|c|c|c|}
\hline k & k_0-1 & k_0 \text{ odd} & k_0+1
\\ \hline \ph{\Bigg|}
\displaystyle \binom{P_+(z,k,k_0)}{R_+(z,k,k_0)} & \displaystyle
\binom{z\rho_{k_0}^{-1}(I_m-\al_{k_0}^*)}
{\wti\rho_{k_0}^{-1}(I_m-\al_{k_0})} & \displaystyle
\binom{zI_m}{I_m} & \displaystyle
\binom{\rho_{k_0+1}^{-1}(I_m+z\al_{k_0+1}^*)}
{\wti\rho_{k_0+1}^{-1}(zI_m+\al_{k_0+1})}
\\ \hline \ph{\Bigg|}
\displaystyle \binom{Q_+(z,k,k_0)}{S_+(z,k,k_0)} & \displaystyle
\binom{z\rho_{k_0}^{-1}(-I_m-\al_{k_0}^*)}
{\wti\rho_{k_0}^{-1}(I_m+\al_{k_0})} & \displaystyle
\binom{zI_m}{-I_m} & \displaystyle
\binom{\rho_{k_0+1}^{-1}(-I_m+z\al_{k_0+1}^*)}
{\wti\rho_{k_0+1}^{-1}(zI_m-\al_{k_0+1})}
\\ \hline \ph{\Bigg|}
\displaystyle \binom{P_-(z,k,k_0)}{R_-(z,k,k_0)} & \displaystyle
\binom{\rho_{k_0}^{-1}(-zI_m-\al_{k_0}^*)}
{\wti\rho_{k_0}^{-1}(\frac1zI_m+\al_{k_0})} & \displaystyle
\binom{I_m}{-I_m} & \displaystyle
\binom{\rho_{k_0+1}^{-1}(-I_m+\al_{k_0+1}^*)}
{\wti\rho_{k_0+1}^{-1}(I_m-\al_{k_0+1})}
\\ \hline \ph{\Bigg|}
\displaystyle \binom{Q_-(z,k,k_0)}{S_-(z,k,k_0)} & \displaystyle
\binom{\rho_{k_0}^{-1}(zI_m-\al_{k_0}^*)}
{\wti\rho_{k_0}^{-1}(\frac1zI_m-\al_{k_0})} & \displaystyle
\binom{I_m}{I_m} & \displaystyle
\binom{\rho_{k_0+1}^{-1}(I_m+\al_{k_0+1}^*)}
{\wti\rho_{k_0+1}^{-1}(I_m+\al_{k_0+1})}
\\ \hline \hline
k & k_0-1 & k_0 \text{ even} & k_0+1
\\ \hline \ph{\Bigg|}
\displaystyle \binom{P_+(z,k,k_0)}{R_+(z,k,k_0)} & \displaystyle
\binom{\wti\rho_{k_0}^{-1}(I_m-\al_{k_0})}{\rho_{k_0}^{-1}(I_m-\al_{k_0}^*)}
& \displaystyle \binom{I_m}{I_m} & \displaystyle
\binom{\wti\rho_{k_0+1}^{-1}(zI_m+\al_{k_0+1})}
{\rho_{k_0+1}^{-1}(\frac1zI_m+\al_{k_0+1}^*)}
\\ \hline \ph{\Bigg|}
\displaystyle \binom{Q_+(z,k,k_0)}{S_+(z,k,k_0)} & \displaystyle
\binom{\wti\rho_{k_0}^{-1}(I_m+\al_{k_0})}
{\rho_{k_0}^{-1}(-I_m-\al_{k_0}^*)} & \displaystyle \binom{-I_m}{I_m}
& \displaystyle \binom{\wti\rho_{k_0+1}^{-1}(zI_m-\al_{k_0+1})}
{\rho_{k_0+1}^{-1}(-\frac1zI_m+\al_{k_0+1}^*)}
\\ \hline \ph{\Bigg|}
\displaystyle \binom{P_-(z,k,k_0)}{R_-(z,k,k_0)} & \displaystyle
\binom{\wti\rho_{k_0}^{-1}(I_m+z\al_{k_0})}
{\rho_{k_0}^{-1}(-zI_m-\al_{k_0}^*)} & \displaystyle
\binom{-zI_m}{I_m} & \displaystyle
\binom{z\wti\rho_{k_0+1}^{-1}(I_m-\al_{k_0+1})}
{\rho_{k_0+1}^{-1}(-I_m+\al_{k_0+1}^*)}
\\ \hline \ph{\Bigg|}
\displaystyle \binom{Q_-(z,k,k_0)}{S_-(z,k,k_0)} & \displaystyle
\binom{\wti\rho_{k_0}^{-1}(I_m-z\al_{k_0})}
{\rho_{k_0}^{-1}(zI_m-\al_{k_0}^*)} & \displaystyle \binom{zI_m}{I_m}
& \displaystyle \binom{z\wti\rho_{k_0+1}^{-1}(I_m+\al_{k_0+1})}
{\rho_{k_0+1}^{-1}(I_m+\al_{k_0+1}^*)}
\\ \hline
\end{array}
\end{align}

\begin{remark}  \lb{r2.4}
Subsequently, we will have to refer to the leading-order terms of certain matrix-valued Laurent polynomials at various occasions. To put this in precise terms we now introduce  the following conventions: We will refer to the terms
\begin{align}
\begin{cases}
z^{-(k+1)/2}, & k \text{ odd},
\\[1mm]
z^{k/2}, & k \text{ even},
\end{cases}
\end{align}
as the leading-order terms of the Laurent polynomials
\begin{align}
\begin{cases}
z^{-1}P_+(z,k_0+k,k_0),\; R_-(z,k_0-k,k_0),
\\
z^{-1}Q_+(z,k_0+k,k_0),\; S_-(z,k_0-k,k_0), & k_0 \text{ odd},
\\
R_+(z,k_0+k,k_0), \; z^{-1}P_-(z,k_0-k,k_0),
\\
S_+(z,k_0+k,k_0), \; z^{-1}Q_-(z,k_0-k,k_0), & k_0 \text{ even},
\end{cases}
\end{align}
and similarly, we will refer to the terms
\begin{align}
\begin{cases}
z^{(k+1)/2}, & k \text{ odd},
\\[1mm]
z^{-k/2}, & k \text{ even},
\end{cases}
\end{align}
as the leading-order term of the Laurent polynomials
\begin{align}
\begin{cases}
R_+(z,k_0+k,k_0), \; P_-(z,k_0-k,k_0),
\\
S_+(z,k_0+k,k_0), \; Q_-(z,k_0-k,k_0), & k_0 \text{ odd},
\\
P_+(z,k_0+k,k_0), \; R_-(z,k_0-k,k_0),
\\
Q_+(z,k_0+k,k_0), \; S_-(z,k_0-k,k_0), & k_0 \text{ even}.
\end{cases}
\end{align}
\end{remark}


\begin{remark}
We note that Lemmas \ref{l3.2} and \ref{l3.3} are crucial for many
of the proofs to follow. For instance, we note that the equivalence
of items $(i)$ and $(iii)$ in Lemma \ref{l3.2} proves that for each
$z\in\bbC\backslash\{0\}$, any solutions $\{U(z,k)\}_{k\in\bbZ}$ of
$\U U(z,\cdot)=zU(z,\cdot)$ can be expressed as a linear
combinations of $P_{+}(z,\cdot,k_0)$ and $Q_{+}(z,\cdot,k_0)$ (or
$P_{-}(z,\cdot,k_0)$ and $Q_{-}(z,\cdot,k_0)$) with $z$-dependent
right-multiple $\Cm$-valued coefficients. This equivalence also
proves that any solution of $\U U(z,\cdot)=zU(z,\cdot)$ is
determined by the values of $U$ and the auxiliary variable $V$ at a
site $k_0$. In the context of Lemma \ref{l3.3}, we remark that its
importance lies in the fact that it shows that in the case of
half-lattice CMV operators, the analogous equations have solutions,
which up to right-multiplication by $z$-dependent $\Cm$-valued
coefficients, are given by $\{P_{\pm}(z,k,k_0)\}_{k\in\bbZ}$ for
each $z\in\bbC\backslash\{0\}$.\ Consequently, the corresponding
solutions are determined by their value at a single site $k_0$.
\end{remark}

Next, we introduce the modified matrix-valued Laurent polynomials
$\wti P_\pm(z,k,k_0)$ and $\wti Q_\pm(z,k,k_0)$, $z\in\Cz$, $k,k_0\in\Z$,
by 
\begin{align}
\wti P_+(z,k,k_0) &=
\begin{cases}
P_+(z,k,k_0)/z, & \text{$k_0$ odd,} \\
P_+(z,k,k_0), & \text{$k_0$ even,}
\end{cases}
\quad \wti P_-(z,k,k_0) =
\begin{cases}
P_-(z,k,k_0), & \text{$k_0$ odd,} \\
-P_-(z,k,k_0)/z, & \text{$k_0$ even,}
\end{cases} \lb{3.49a}
\\
\wti Q_+(z,k,k_0) &=
\begin{cases}
Q_+(z,k,k_0)/z, & \text{$k_0$ odd,} \\
Q_+(z,k,k_0), & \text{$k_0$ even,}
\end{cases}
\quad \wti Q_-(z,k,k_0) =
\begin{cases}
Q_-(z,k,k_0), & \text{$k_0$ odd,} \\
-Q_-(z,k,k_0)/z, & \text{$k_0$ even.}
\end{cases} \lb{3.50a}
\end{align}

In the remainder of this paper we use the following abbreviations for subarcs $A_{\zeta}$ of $\dD$,
\begin{equation}
A_{\zeta}=\big\{e^{i\phi}\in\dD\,\big|\, 0\leq \phi\leq\theta\big\}, \quad \zeta=e^{i\te},
\; \te \in [0,2\pi).
\end{equation}

The next auxiliary result is of importance in proving orthonormality of the matrix-valued Laurent polynomials $P_\pm$ and $R_\pm$.

\begin{lemma} \lb{l3.5}
Let $\{F_\pm(\cdot,k,k_0)\}_{k\gtreqless k_0}$ denote two sequences
of $\Cm$-valued functions of bounded variation with $F_\pm(1,k,k_0)=0$
for all $k\gtreqless k_0$ that satisfy
\begin{align}
(\U_{\pm,k_0} F_\pm(\zeta,\cdot,k_0))(k) = \int_{A_{\zeta}}
dF_\pm(\zeta',k,k_0) \, \zeta', \quad \zeta\in\dD,\; k\gtreqless k_0,
\end{align}
where $\U_{\pm,k_0}$ are understood in the sense of difference
expressions on $\smm{\Z\cap[k_0,\pm\infty)}$ rather than difference
operators on $\ltm{\Z\cap[k_0,\pm\infty)}$ $($cf.\ Remark
\ref{r3.0}$)$. Then, $F_\pm(\cdot,k,k_0)$ also satisfy
\begin{align}
F_\pm(\zeta,k,k_0) = \int_{A_{\zeta}} \wti
P_\pm(\zeta',k,k_0)\,dF_\pm(\zeta',k_0,k_0), \quad \zeta\in\dD,\;
k\gtreqless k_0.
\end{align}
\end{lemma}
\begin{proof}
Let $\{G_\pm(\cdot,k,k_0)\}_{k\gtreqless k_0}$ denote the  
two sequences of $\Cm$-valued functions,
\begin{align}
G_\pm(\zeta,k,k_0) = \int_{A_{\zeta}} \wti
P_\pm(\zeta',k,k_0)dF_\pm(\zeta',k_0,k_0), \quad \zeta\in\dD,\;
k\gtreqless k_0.
\end{align}
Then it suffices to prove that $F_\pm(\zeta,k,k_0) =
G_\pm(\zeta,k,k_0)$, $\zeta\in\dD$, $k\gtreqless k_0$.

First, we note that according to \eqref{3.49}, \eqref{3.50}, and
\eqref{3.49a}, $\wti P_\pm(\zeta,k_0,k_0) = I_m$, and hence,
\begin{align}
G_\pm(\zeta,k_0,k_0) = \int_{A_{\zeta}} dF_\pm(\zeta',k_0,k_0) =
F_\pm(\zeta,k_0,k_0), \quad \zeta\in\dD.
\end{align}
Moreover,
\begin{align}
(\U_{\pm,k_0} G_\pm(\zeta,\cdot,k_0))(k) &= \int_{A_{\zeta}}
(\U_{\pm,k_0}\wti P_\pm(\zeta',\cdot,k_0))(k)dF_\pm(\zeta',k_0,k_0) \no
\\
&= \int_{A_{\zeta}} dG_\pm(\zeta',k,k_0) \, \zeta', \quad \zeta\in\dD,\;
k\gtreqless k_0.
\end{align}

Next, defining $K_\pm(\zeta,k,k_0) =
F_\pm(\zeta,k,k_0)-G_\pm(\zeta,k,k_0)$, $\zeta\in\dD$, $k\gtreqless
k_0$, one obtains
\begin{align*}
K_\pm(\zeta,k_0,k_0) = 0 \quad \text{and}\quad (\U_{\pm,k_0}
K_\pm(\zeta,\cdot,k_0))(k) = \int_{A_{\zeta}} dK_\pm(\zeta',k,k_0) \,  \zeta' , \quad
\zeta\in\dD,\; k\gtreqless k_0,
\end{align*}
or equivalently,
\begin{align}
K_\pm(\zeta,k_0,k_0) = 0 \quad \text{and}\quad (\U_{\pm,k_0}
K_\pm(\zeta,\cdot,k_0))(k) = (\L \, K_\pm(\cdot,k,k_0))(\zeta), \quad
\zeta\in\dD,\; k\gtreqless k_0, \lb{3.53}
\end{align}
where $\L$ denotes the boundedly invertible operator on $\Cm$-valued
functions $K$ of bounded variation defined by 
\begin{align}
(\L \, K)(\zeta) = \int_{A_{\zeta}} dK(\zeta')\,  \zeta', \quad (\L^{-1} K)(\zeta) =
\int_{A_{\zeta}} dK(\zeta') \,  {\zeta'}^{-1}.
\end{align}

Finally, since, $\L$ commutes with all constant $m\times m$ matrices,
one can repeat the proof of Lemma \ref{l3.3} with $z$ replaced by
$\L$ and obtain that \eqref{3.53} has the unique solution
$K_\pm(\zeta,k,k_0)=0$, $\zeta\in\dD$, $k\gtreqless k_0$, and hence,
$F_\pm(\zeta,k,k_0) = G_\pm(\zeta,k,k_0)$, $\zeta\in\dD$, $k\gtreqless
k_0$.
\end{proof}

Next, following \cite{Be68} (see also \cite{BG90}), we prove a matrix-valued version
of the ``orthogonality" relation for matrix-valued Laurent polynomials
$P_\pm$ and $R_\pm$.

Let $\De_k=\{\De_k(\ell)\}_{\ell\in\Z}\in\smm{\Z}$, $k\in\Z$, denote
the sequences of $m\times m$ matrices defined by
\begin{align}
(\De_k)(\ell) =
\begin{cases}
I_m, & \ell=k,\\
0, & \ell\neq k,
\end{cases}
\quad k,\ell\in\Z.
\end{align}
Then using right-multiplication by $m\times m$ matrices on
$\smm{\Z}$ defined in Remark \ref{r3.0}, we get the 
identity 
\begin{align} \lb{B.1a}
(\De_k X)(\ell) =
\begin{cases}
X, & \ell=k,\\
0, & \ell\neq k,
\end{cases}  \quad  X\in\Cm,
\end{align}
and hence consider $\De_k$ as a map $\De_k\colon \Cm\to\smm{\Z}$. In
addition, we introduce the map $\De_k^*\colon \smm{\Z}\to\Cm$,
$k\in\Z$, defined by
\begin{align} \lb{B.1b}
\De_k^* \Phi = \Phi(k), \,\text{ where }\,
\Phi=\{\Phi(k)\}_{k\in\Z}\in\smm{\Z}.
\end{align}
Similarly, one introduces the corresponding maps with $\bbZ$ replaced by
$[k_0,\pm\infty)\cap\Z$, $k_0\in\Z$, which, for notational brevity, we
will also denote by $\De_k$ and $\De_k^*$, respectively.

Next, we call sequences of $C^{m\times m}$-valued
functions $\{\Phi_\pm(\cdot,k,k_0)\}_{k\gtreqless k_0}$
orthonormal\footnote{This is denoted by pseudo-orthonormality in
\cite[Sect.~VII.2.6]{Be68}} on $\dD$ with respect to some
$\Cm$-valued measures $d\Om_\pm(\cdot,k_0)$, defined on $\dD$, if
the following identity holds for all $k,k'\gtreqless k_0$,
\begin{align}
\oint_{\dD}
\Phi_\pm(\zeta,k,k_0)\,d\Om_\pm(\zeta,k_0)\,\Phi_\pm(\zeta,k',k_0)^*
&= \de_{k,k'}I_m.
\end{align}
We will also call the sequences of $\Cm$-valued functions
$\{\Phi_\pm(\cdot,k,k_0)\}_{k\gtreqless k_0}$ complete with respect
to the measures $d\Om_\pm(\cdot,k_0)$ if the collections
of $\C^m$-valued functions
\begin{align}
\left\{\phi_{\pm,j}(\cdot,k,k_0)=
\begin{pmatrix}
(\Phi_\pm(\cdot,k,k_0))_{1,j}\\\vdots\\(\Phi_\pm(\cdot,k,k_0))_{m,j}
\end{pmatrix}\right\}_{ j=1,\dots,m, \; k\gtreqless k_0}
\end{align}
form complete systems in $\Ltm{\pm}$.

\begin{lemma} \lb{l3.6}
Let $k_0\in\bbZ$. The sets of $\Cm$-valued Laurent polynomials
$\{P_\pm(\cdot,k,k_0)^*\}_{k\gtreqless k_0}$ and
$\{R_\pm(\cdot,k,k_0)^*\}_{k\gtreqless k_0}$ form complete
orthonormal systems on $\dD$ with respect to $\Cm$-valued measures
$d\Om_\pm(\cdot,k_0)$ defined by
\begin{equation}
d\Om_\pm(\zeta,k_0) = d(\De_{k_0}^* E_{\U_{\pm,k_0}}(\zeta)\De_{k_0}),
\quad \zeta\in\dD, \lb{3.54}
\end{equation}
where $E_{\U_{\pm,k_0}}(\cdot)$ denotes the family of spectral projections of
the half-lattice unitary operators $\U_{\pm,k_0}$,
\begin{equation}
\U_{\pm,k_0}=\oint_{\dD}dE_{\U_{\pm,k_0}}(\zeta)\,\zeta.  \lb{3.55} 
\end{equation}
Explicitly, $P_\pm$ and $R_\pm$ satisfy,
\begin{align}
\oint_{\dD}
P_\pm(\zeta,k,k_0)\,d\Om_\pm(\zeta,k_0)\,P_\pm(\zeta,k',k_0)^* &=
\de_{k,k'}I_m, \quad k,k' \gtreqless k_0, \lb{3.56}
\\
\oint_{\dD}
R_\pm(\zeta,k,k_0)\,d\Om_\pm(\zeta,k_0)\,R_\pm(\zeta,k',k_0)^* &=
\de_{k,k'}I_m, \quad k,k' \gtreqless k_0. \lb{3.57}
\end{align}
\end{lemma}
\begin{proof}
Fix an integer $k_1\gtreqless k_0$ and let
$\{F_\pm(\cdot,k,k_0)\}_{k\gtreqless k_0}$ denote two $\Cm$-valued
sequences of functions of bounded variation,
\begin{align}
F_\pm(\zeta,k,k_0) = \De_{k}^* E_{\U_{\pm,k_0}}(\zeta)\De_{k_1}, \quad
\zeta\in\dD, \; k\gtreqless k_0.
\end{align}
Then,
\begin{align}
(\U_{\pm,k_0} F_\pm(\zeta,\cdot,k_0))(k) &=
(\U_{\pm,k_0}E_{\U_{\pm,k_0}}(\zeta)\De_{k_1})(k) = \left(\int_{A_{\zeta}}
dE_{\U_{\pm,k_0}}(\zeta')\, \zeta' \De_{k_1}\right)(k)
\\ &=
\int_{A_{\zeta}} d\big(\De_k^* E_{\U_{\pm,k_0}}(\zeta')\De_{k_1}\big) \, \zeta' =
\int_{A_{\zeta}} dF_\pm(\zeta',k,k_0) \, \zeta', \quad \zeta\in\dD,\;
k\gtreqless k_0, \no
\end{align}
and hence it follows from Lemma \ref{l3.5} that
\begin{align}
F_\pm(\zeta,k,k_0) = \int_{A_{\zeta}} \wti
P_\pm(\zeta',k,k_0)\,dF_\pm(\zeta',k_0,k_0), \quad \zeta\in\dD,
\; k\gtreqless k_0,
\end{align}
or equivalently,
\begin{align}
\De_{k}^* E_{\U_{\pm,k_0}}(\zeta)\De_{k_1} = \int_{A_{\zeta}} \wti
P_\pm(\zeta',k,k_0)\,d\big(\De_{k_0}^*
E_{\U_{\pm,k_0}}(\zeta')\De_{k_1}\big), \quad \zeta\in\dD,\;
k\gtreqless k_0.
\end{align}
In particular, taking $k_1=k'$ and $k_1=k_0$, one obtains,
respectively,
\begin{align}
\De_{k}^* E_{\U_{\pm,k_0}}(\zeta)\De_{k'} &= \int_{A_{\zeta}} \wti
P_\pm(\zeta',k,k_0)\,d\big(\De_{k_0}^*
E_{\U_{\pm,k_0}}(\zeta')\De_{k'}\big), \quad \zeta\in\dD,\;
k\gtreqless k_0,   \lb{3.62}
\end{align}
and
\begin{align}
\De_{k'}^* E_{\U_{\pm,k_0}}(\zeta)\De_{k_0} &= \int_{A_{\zeta}} \wti
P_\pm(\zeta',k',k_0)\,d\big(\De_{k_0}^*
E_{\U_{\pm,k_0}}(\zeta')\De_{k_0}\big) \no
\\ &=
\int_{A_{\zeta}} \wti P_\pm(\zeta',k',k_0)\,d\Om_\pm(\zeta',k_0), \quad
\zeta\in\dD,\; k' \gtreqless k_0. \lb{3.64}
\end{align}
Taking adjoints in \eqref{3.64} one also obtains
\begin{align}
\De_{k_0}^* E_{\U_{\pm,k_0}}(\zeta)\De_{k'} &= \int_{A_{\zeta}}
d\Om_\pm(\zeta',k_0)\,\wti P_\pm(\zeta',k',k_0)^*, \quad
\zeta\in\dD,\; k' \gtreqless k_0. \lb{3.65}
\end{align}
Thus, inserting \eqref{3.49a} and \eqref{3.65} into \eqref{3.62} and
letting $\te\to2\pi$, $\zeta=e^{i\te}$, yields \eqref{3.56},
\begin{align}
\de_{k,k'} I_m &= \oint_{\dD} \wti
P_\pm(\zeta,k,k_0)\,d\Om_\pm(\zeta,k_0)\,\wti P_\pm(\zeta,k',k_0)^*
\no
\\
&= \oint_{\dD}
P_\pm(\zeta,k,k_0)\,d\Om_\pm(\zeta,k_0)\,P_\pm(\zeta,k',k_0)^*, \quad
k,k'\gtreqless k_0.
\end{align}
Finally, \eqref{3.57} is a consequence of \eqref{3.56} and the 
relation
\begin{align}
R_\pm(z,k,k_0) = \frac1z (\W_{\pm,k_0}P_\pm(z,\cdot,k_0))(k), \quad
z\in\Cz,\; k\gtreqless k_0,
\end{align}
where $\W_{\pm,k_0}$ are the unitary block diagonal semi-infinite
matrices defined in \eqref{3.34}.

To prove completeness of $\{P_\pm(\cdot,k,k_0)^*\}_{k\gtreqless
k_0}$ and $\{R_\pm(\cdot,k,k_0)^*\}_{k\gtreqless k_0}$ we first note
the subsequent fact that can be inferred from the definitions of $P_\pm$ and
$R_\pm$ and, in particular, from \eqref{3.21}, \eqref{3.22}, \eqref{3.49},
and \eqref{3.50},
\begin{align}
\spn\{P_\pm(\ze,k,k_0)^*\}_{k\gtreqless k_0} &=
\spn\{R_\pm(\ze,k,k_0)^*\}_{k\gtreqless k_0} = \spn\left\{\ze^k
I_m\right\}_{k\in\bbZ}.
\end{align}
Hence, it suffices to prove that $\big\{\ze^k I_m\big\}_{k\in\Z}$
are complete with respect to $d\Om_\pm(\cdot,k_0)$. Suppose
$F\in\Ltm{\pm}$ is orthogonal to all columns of $\ze^kI_m$ for all
$k\in\Z$, that is,
\begin{equation}
\oint_{\dD} \zeta^{-k}
d\Omega(\zeta,k_0)\,F(\zeta)=
\begin{pmatrix}0\\\vdots\\0\end{pmatrix}\in\C^m,
\quad k\in\Z.
\end{equation}
Note that for a scalar complex-valued measure $d\om$ equalities
$\oint d\om(\ze)\,\ze^n=0$, $n\in\Z$, imply that $\oint
d\Re(\om(\ze))\,\ze^n=\oint d\Im(\om(\ze))\,\ze^n = 0$, and hence one 
concludes from \cite[p.\ 24]{Du83}) that $d\om=0$. Applying this
argument to $d(\Omega_{\pm}(\cdot,k_0)F(\cdot))_\ell$,
$\ell=1,\dots,m$, one obtains
\begin{align}
d\Omega_{\pm}(\cdot,k_0)F(\cdot) &=
\begin{pmatrix}0\\\vdots\\0\end{pmatrix}\in\C^m.
\end{align}
Multiplying by $F(\cdot)^*$ on the left and integrating over the
unit circle then yields
\begin{align}
\|F\|_{\Ltm{\pm}}^2=\oint_{\dD} F(\zeta)^* \,d\Om_\pm(\ze,k_0)\,
F(\zeta) = 0.
\end{align}
\end{proof}

We note that $d\Om_\pm(\cdot,k_0)$, $k_0\in\bbZ$, defined in
\eqref{3.54} are normalized, nonnegative, nondegenerate,
$\Cm$-valued measures supported on infinite subsets of $\dD$,
that is, for any $\Cm$-valued Laurent polynomial $P(z)$ the
following properties hold,
\begin{align}
(i)&\; \oint_{\dD} d\Om_\pm(\zeta,k_0) = I_m \;\text{ and }\;
\oint_{\dD} P(\zeta)\,d\Om_\pm(\zeta,k_0)\,P(\zeta)^* \geq 0.
\\
(ii)&\;\text{ If $P(z)= z^{-n}A_{-n}+...+z^n A_n$ and either $A_n$ or $A_{-n}$ is invertible,}  \\
&\; \text{ then } \, \oint_{\dD} P(\zeta)\,d\Om_\pm(\zeta,k_0)\,P(\zeta)^* > 0.
\\
(iii)&\;\text{ If }\, \oint_{\dD}
P(\zeta)\,d\Om_\pm(\zeta,k_0)\,P(\zeta)^* = 0 \, \text{ then }\,
P(z)=0.
\end{align}
The infinite support property of the spectral measure is a consequence of 
the fact that we have infinitely many linearly independent
orthogonal Laurent polynomials $P_\pm$. Property $(i)$ follows
from \eqref{3.54}, and properties $(ii)$ and
$(iii)$ are implied by the orthogonality relations \eqref{3.56}, \eqref{3.57}, and the fact that
the matrix-valued Laurent polynomials $P_\pm$ and $R_\pm$ have invertible leading-order
coefficients (cf.\ Remark \ref{r2.4}).

\begin{corollary}
Let $k_0\in\bbZ$. Then the operators $\U_{\pm,k_0}$ are unitarily
equivalent to the operators of multiplication by $\zeta$ on
$\Ltm{\pm}$. In particular,
\begin{align}
& \si(\U_{\pm,k_0}) = \supp \, (d\Om_\pm(\cdot,k_0)).
\end{align}
\end{corollary}
\begin{proof}
Consider the linear maps $\dot
\cU_\pm\colon\ltzm{[k_0,\pm\infty)\cap\bbZ}\to\Ltm{\pm}$ from the
space of compactly supported sequences
$\ltzm{[k_0,\pm\infty)\cap\bbZ}$ to the set of $\C^m$-valued Laurent
polynomials defined by
\begin{equation}
(\dot \cU_\pm F)(z) = \sum_{k=k_0}^{\pm\infty} \wti
P_\pm(1/\ol{z},k,k_0)^* F(k), \quad F\in
\ltzm{[k_0,\pm\infty)\cap\bbZ}.
\end{equation}
Using \eqref{3.56} one shows that $\hatt F(\zeta) = (\dot \cU_\pm
F)(\zeta)$, $F\in \ltzm{[k_0,\pm\infty)\cap\bbZ}$ has the 
property
\begin{align}
\|\hatt F\|^2_{\Ltm{\pm}} &= \oint_{\dD} \hatt
F(\zeta)^*d\Om_\pm(\zeta,k_0)\,\hatt F(\zeta)
\\ &=
\oint_{\dD} \sum_{k=k_0}^{\pm\infty} F(k)^*\wti P_\pm(\zeta,k,k_0)
\,d\Om_\pm(\zeta,k_0)\!\! \sum_{k'=k_0}^{\pm\infty}\wti
P_\pm(\zeta,k',k_0)^*F(k') \no
\\ &=
\sum_{k,k'=k_0}^{\pm\infty} F(k)^* \left(\oint_{\dD} \wti
P_\pm(\zeta,k,k_0) \,d\Om_\pm(\zeta,k_0)\, \wti
P_\pm(\zeta,k',k_0)^*\right)F(k') \no
\\ &=
\sum_{k=k_0}^{\pm\infty} F(k)^*F(k) =
\|F\|^2_{\ltm{[k_0,\pm\infty)\cap\bbZ}}. \lb{3.73}
\end{align}
Since $\ltzm{[k_0,\pm\infty)\cap\bbZ}$ is dense in
$\ltm{[k_0,\pm\infty)\cap\bbZ}$, $\dot \cU_\pm$ extend to bounded
linear operators $\cU_\pm\colon \ltm{[k_0,\pm\infty)\cap\bbZ} \to
\Ltm{\pm}$, and the identity 
\begin{align} \lb{3.74}
(\cU_\pm(\U_{\pm,k_0}F))(\zeta) &= \sum_{k=k_0}^{\pm\infty} \wti
P_\pm(\zeta,k,k_0)^* (\U_{\pm,k_0}F)(k) =
\sum_{k=k_0}^{\pm\infty} (\U_{\pm,k_0}^*\wti
P_\pm(\zeta,\cdot,k_0))(k)^* F(k)
\\ &=
\sum_{k=k_0}^{\pm\infty} (\zeta^{-1}\wti P_\pm(\zeta,k,k_0))^* F(k)
= \zeta(\cU_\pm F)(\zeta), \quad F\in\ltm{[k_0,\pm\infty)\cap\bbZ}, 
\no
\end{align}
holds. The ranges of the operators $\cU_\pm$ are all of $\Ltm{\pm}$
since the sets of Laurent polynomials $\{\wti
P_\pm(\cdot,k,k_0)^*\}_{k\gtreqless k_0}$ are complete with respect
to $d\Om_\pm(\cdot,k_0)$, and hence $\cU_\pm$ are onto. Finally, one
computes the inverse operators $\cU_\pm^{-1}$,
\begin{equation}
(\cU_\pm^{-1}\hatt F)(k) = \oint_{\dD} \wti
P_\pm(\zeta,k,k_0)\,d\Om_{\pm}(\zeta,k_0)\,\hatt F(\zeta), \quad
\hatt F\in\Ltm{\pm},
\end{equation}
which together with \eqref{3.73} implies that $\cU_\pm$ are unitary.
In addition, \eqref{3.74} implies that the half-lattice
unitary operators $\U_{\pm,k_0}$ on $\ltm{[k_0,\pm\infty)\cap\bbZ}$
are unitarily equivalent to the operators of multiplication by
$\zeta$ on $\Ltm{\pm}$,
\begin{align}
(\cU_\pm \U_{\pm,k_0} \cU_\pm^{-1} \hatt F)(\zeta) = \zeta\hatt
F(\zeta), \quad \hatt F\in \Ltm{\pm}.
\end{align}
\end{proof}

\begin{corollary} \lb{c3.8}
Let $k_0\in\bbZ$. \\ The matrix-valued Laurent polynomials
$\{P_+(\cdot,k,k_0)\}_{k\geq k_0}$ can be constructed by
Gram--Schmidt orthogonalizing
\begin{equation}
\begin{cases}
\zeta I_m,\, I_m,\,\zeta^2 I_m,\, \zeta^{-1}I_m,\, \zeta^3 I_m,\,
\zeta^{-2}I_m,\dots, & \text{$k_0$ odd,}
\\
I_m,\, \zeta I_m,\, \zeta^{-1}I_m,\, \zeta^2 I_m,\,
\zeta^{-2}I_m,\, \zeta^2 I_m, \dots, & \text{$k_0$ even}
\end{cases}
\end{equation}
in the context of matrix-valued Laurent polynomials orthogonal with respect
to $d\Om_+(\cdot,k_0)$.
\\
The matrix-valued Laurent polynomials $\{R_+(\cdot,k,k_0)\}_{k\geq
k_0}$ can be constructed by Gram--Schmidt orthogonalizing
\begin{equation}
\begin{cases}
I_m,\, \zeta I_m,\, \zeta^{-1}I_m,\, \zeta^2 I_m,\,
\zeta^{-2}I_m,\, \zeta^3I_m,\dots, & \text{$k_0$ odd,}
\\
I_m,\, \zeta^{-1}I_m,\, \zeta I_m,\, \zeta^{-2}I_m,\,
\zeta^2 I_m,\, \zeta^{-3}I_m,\dots, & \text{$k_0$ even}
\end{cases}
\end{equation}
in the context of matrix-valued Laurent polynomials orthogonal with respect
to $d\Om_+(\cdot,k_0)$.
\\
The matrix-valued Laurent polynomials $\{P_-(\cdot,k,k_0)\}_{k\leq
k_0}$ can be constructed by Gram--Schmidt orthogonalizing
\begin{equation}
\begin{cases}
I_m,\, -\zeta I_m,\, \zeta^{-1}I_m,\, -\zeta^2 I_m,\,
\zeta^{-2}I_m,\, -\zeta^3 I_m,\dots, & \text{$k_0$ odd,}
\\
-\zeta I_m,\, I_m,\, -\zeta^2 I_m,\, \zeta^{-1}I_m,\,
-\zeta^3I_m,\, \zeta^{-2}I_m, \dots, & \text{$k_0$ even}
\end{cases}
\end{equation}
in the context of matrix-valued Laurent polynomials orthogonal with respect
to $d\Om_-(\cdot,k_0)$.
\\
The matrix-valued Laurent polynomials $\{R_-(\cdot,k,k_0)\}_{k\leq
k_0}$ can be constructed by Gram--Schmidt orthogonalizing
\begin{equation}
\begin{cases}
-I_m,\, \zeta^{-1}I_m,\, -\zeta I_m,\, \zeta^{-2}I_m,\,
-\zeta^2I_m,\, \zeta^{-3}I_m,\dots, & \text{$k_0$ odd,}
\\
I_m,\, -\zeta I_m,\, \zeta^{-1}I_m,\, -\zeta^2 I_m,\,
\zeta^{-2}I_m,\, -\zeta^3 I_m, \dots, & \text{$k_0$ even}
\end{cases}
\end{equation}
in the context of matrix-valued Laurent polynomials orthogonal with respect
to $d\Om_-(\cdot,k_0)$.

Here the Gram--Schmidt orthogonalization procedure employs
left-multiplication by constant $($i.e., $\zeta$-independent\,$)$
$m\times m$ matrices as discussed in \cite[Sect.\ VII.2.8]{Be68}.
\end{corollary}
\begin{proof}
The result is a consequence of the definition of the Laurent polynomials
$P_\pm$ and $R_\pm$ and Lemma \ref{l3.6}.
\end{proof}

We note that the Gram--Schmidt orthogonalization
process implies that all matrix-valued Laurent polynomials constructed in Corollary \ref{c3.8} have
self-adjoint invertible leading-order coefficients (cf.\ Remark \ref{r2.4}).

The next result clarifies which measures arise as spectral
measures of half-lattice CMV operators and it yields the
reconstruction of the matrix-valued Verblunsky coefficients from the spectral measures
and the corresponding orthogonal Laurent polynomials.

\begin{theorem} \lb{t3.9}
Let $k_0\in\Z$ and $d\Om_\pm(\cdot,k_0)$ be nonnegative finite
measures on $\dD$, supported on infinite sets, and normalized by
\begin{align}
\oint_{\dD} d\Om_\pm(\zeta,k_0) = I_m.
\end{align}
Moreover, assume that $d\Om_\pm(\cdot,k_0)$ are nondegenerate in the
sense that expressions of the form
\begin{align}
\oint_{\dD} P(\zeta) d\Om_\pm(\zeta,k_0) P(\zeta)^*
\end{align}
are invertible for all $\Cm$-valued Laurent polynomials $P(z)=z^{-n}A_{-n}+...+z^n A_n$ with either $A_{-n}=I_m$ or $A_{n}=I_m$. Then $d\Om_\pm(\cdot,k_0)$ are necessarily
the spectral measures for some half-lattice CMV operators
$\U_{\pm,k_0}$ with coefficients $\{\al_k\}_{k \geq k_0+1}$,
respectively, $\{\al_k\}_{k \leq k_0}$, defined by 
\begin{equation}
\al_k = -
\begin{cases}
\oint_{\dD} \zeta
R_+(\zeta,k-1,k_0)d\Om_+(\zeta,k_0)P_+(\zeta,k-1,k_0)^*, & k
\text{ odd,}
\\
\oint_{\dD}
P_+(\zeta,k-1,k_0)d\Om_+(\zeta,k_0)R_+(\zeta,k-1,k_0)^*, & k
\text{ even}
\end{cases}   \lb{3.83}
\end{equation}
for all $k \geq k_0+1$, and
\begin{equation}
\al_k = -
\begin{cases}
\oint_{\dD} \zeta
R_-(\zeta,k-1,k_0)d\Om_-(\zeta,k_0)P_-(\zeta,k-1,k_0)^*, & k
\text{ odd,}
\\
\oint_{\dD}
P_-(\zeta,k-1,k_0)d\Om_-(\zeta,k_0)R_-(\zeta,k-1,k_0)^*, & k
\text{ even}
\end{cases} \lb{3.84}
\end{equation}
for all $k \leq k_0$. Here the matrix-valued Laurent polynomials
$\{P_\pm(\cdot,k,k_0)\}_{k\geq k_0}$ and
$\{R_\pm(\cdot,k,k_0)\}_{k\geq k_0}$  denote the orthonormal
Laurent polynomials constructed in Corollary \ref{c3.8}.
\end{theorem}
\begin{proof}
First, using Corollary \ref{c3.8}, one constructs the orthonormal
polynomials $\{P_+(\cdot,k,k_0)\}_{k\geq k_0}$ and
$\{R_+(\cdot,k,k_0)\}_{k\geq k_0}$.

Next, we will establish the recursion relation \eqref{3.37}. Assume
$k$ is odd and consider the matrix-valued Laurent 
polynomials $P$ and $R$,
\begin{align}
P(\zeta) &= \wti\rho_k P_+(\zeta,k,k_0) - \zeta
R_+(\zeta,k-1,k_0),
\\
R(\zeta) &= \rho_k R_+(\zeta,k,k_0) - \zeta^{-1}
P_+(\zeta,k-1,k_0),
\end{align}
where $\rho_k$, $\wti\rho_k \in \Cm$ are self-adjoint invertible
matrices chosen such that the leading-order terms of the Laurent polynomials
$\wti\rho_k P_+(\zeta,k,k_0)$ and $\rho_k R_+(\zeta,k,k_0)$
cancel the leading-order terms of $\zeta R_+(\zeta,k-1,k_0)$ and
$\zeta^{-1} P_+(\zeta,k-1,k_0)$, respectively (cf.\ Remark \ref{r2.4}). Using Corollary
\ref{c3.8} one then checks that the Laurent polynomials $P$ and $R$ are
constant  $m\times m$ matrix left-multiples of $P_+(\cdot,k-1,k_0)$
and $R_+(\cdot,k-1,k_0)$, respectively,
\begin{align}
\al_k P_+(\zeta,k-1,k_0) &= \wti\rho_k P_+(\zeta,k,k_0) -
\zeta R_+(\zeta,k-1,k_0), \lb{3.87}
\\
\wti\al_k R_+(\zeta,k-1,k_0) &= \rho_k R_+(\zeta,k,k_0) -
\zeta^{-1} P_+(\zeta,k-1,k_0), \lb{3.88}
\end{align}
with $\al_k$, $\wti\al_k \in \Cm$ constant $m\times m$ matrices.
Moreover, using \eqref{3.87}, \eqref{3.88}, and Lemma \ref{l3.6} one
computes,
\begin{align}
I_m &= \oint_{\dD} \zeta R_+(\zeta,k-1,k_0)\,
d\Om_\pm(\zeta,k_0)\, \zeta^{-1}R_+(\zeta,k-1,k_0)^* \no
\\ &=
\oint_{\dD} \Big(\wti\rho_k P_+(\zeta,k,k_0)-\al_k
P_+(\zeta,k-1,k_0)\big)d\Om_\pm(\zeta,k_0) \no
\\ &\hspace*{2.45cm}
\times \big(\wti\rho_k P_+(\zeta,k,k_0)-\al_k
P_+(\zeta,k-1,k_0)\big)^* \no
\\ &=
\wti\rho_k^2 + \al_k\al_k^*, \lb{3.89}
\\
I_m &= \oint_{\dD} \zeta^{-1}P_+(\zeta,k-1,k_0)\,
d\Om_\pm(\zeta,k_0)\, \zeta P_+(\zeta,k-1,k_0)^* \no
\\ &=
\oint_{\dD} \Big(\rho_k R_+(\zeta,k,k_0)-\wti\al_k
R_+(\zeta,k-1,k_0)\big)d\Om_\pm(\zeta,k_0) \no
\\ &\hspace*{2.45cm}
\times \big(\rho_k R_+(\zeta,k,k_0)-\wti\al_k
R_+(\zeta,k-1,k_0)\big)^* \no
\\ &=
\rho_k^2 + \wti\al_k\wti\al_k^*, \lb{3.90}
\end{align}
and
\begin{align}
\al_k &= \oint_{\dD} \al_k P_+(\zeta,k-1,k_0)\,
d\Om_\pm(\zeta,k_0)\,P_+(\zeta,k-1,k_0)^* \no
\\ &=
\oint_{\dD} \Big(\wti\rho_k P_+(\zeta,k,k_0) -
\zeta R_+(\zeta,k-1,k_0)\big)d\Om_\pm(\zeta,k_0)\,
P_+(\zeta,k-1,k_0)^* \no
\\ &=
- \oint_{\dD} \zeta R_+(\zeta,k-1,k_0)\,d\Om_\pm(\zeta,k_0)\,
P_+(\zeta,k-1,k_0)^*, \lb{3.91}
\\
\wti\al_k &= \oint_{\dD} \wti\al_k R_+(\zeta,k-1,k_0)\,
d\Om_\pm(\zeta,k_0)\,R_+(\zeta,k-1,k_0)^* \no
\\ &=
\oint_{\dD} \Big(\rho_k R_+(\zeta,k,k_0) -
\zeta^{-1}P_+(\zeta,k-1,k_0)\big)d\Om_\pm(\zeta,k_0)\,
R_+(\zeta,k-1,k_0)^* \no
\\ &=
- \oint_{\dD} \zeta^{-1}P_+(\zeta,k-1,k_0)\,d\Om_\pm(\zeta,k_0)\,
R_+(\zeta,k-1,k_0)^*. \lb{3.92}
\end{align}
Thus, \eqref{3.89}--\eqref{3.92} imply that $\wti\al_k =
\al_k^*$, $\rho_k = \sqrt{I_m-\al_k^*\al_k}$, and $\wti\rho_k =
\sqrt{I_m-\al_k\al_k^*}$, and hence \eqref{3.87} and \eqref{3.88}
yield the recursion relation \eqref{3.37}. A similar argument also
proves the recursion relation \eqref{3.37} for the case $k$ even.

Finally, using Lemma \ref{l3.3} one concludes that the Laurent
polynomials $\{P_+(\cdot,k,k_0)\}_{k\geq k_0}$ form a generalized
eigenvector of the operator $\U_{+,k_0}$ associated with the
coefficients $\al_k,\rho_k,\wti\rho_k$ introduced above. Thus, the
measure $d\Om_+(\cdot,k_0)$ is the spectral measure of the operator
$\U_{+,k_0}$.

Similarly one proves the result for $d\Om_-(\cdot,k_0)$ and
\eqref{3.84} for $k\leq k_0$.
\end{proof}

\begin{lemma} \lb{l3.10}
Let $z\in\C\backslash(\dD\cup\{0\})$ and $k_0\in\bbZ$. Then the
following identity holds,
\begin{align}
\begin{split} \lb{3.93}
& \wti Q_\pm(z,k,k_0) = \pm \oint_{\dD}
\frac{\zeta+z}{\zeta-z} \big(\wti P_\pm(\zeta,k,k_0) - \wti
P_\pm(z,k,k_0)\big)d\Om_{\pm}(\zeta,k_0), \quad k \gtrless k_0,
\\
&S_\pm(z,k,k_0) = \pm \oint_{\dD} \frac{\zeta+z}{\zeta-z}
\big(R_\pm(\zeta,k,k_0)-R_\pm(z,k,k_0)\big) d\Om_{\pm}(\zeta,k_0),
\quad k \gtrless k_0.
\end{split}
\end{align}
\end{lemma}
\begin{proof}
To simplify our further notation we agree to write both equalities in
\eqref{3.93} as a single one,
\begin{align} \lb{3.95}
&\binom{\wti Q_\pm(z,k,k_0)}{S_\pm(z,k,k_0)} = \pm\oint_{\dD}
\frac{\zeta+z}{\zeta-z} \bigg( \binom{\wti
P_\pm(\zeta,k,k_0)}{R_\pm(\zeta,k,k_0)} - \binom{\wti
P_\pm(z,k,k_0)}{R_\pm(z,k,k_0)} \bigg) d\Om_{\pm}(\zeta,k_0), \quad
k\gtrless k_0,
\end{align}
where the integration on the right-hand side is understood
componentwise, that is, an expression of the type $\oint_{\dD}
\left(\begin{smallmatrix} G_1(\zeta) \\ G_2(\zeta)\end{smallmatrix}\right)
d\Om_\pm(\zeta,k_0)$
with $G_1(z)$ and $G_2(z)$ some $\Cm$-valued Laurent polynomials is defined by
\begin{align}
\oint_{\dD} \binom{G_1(\zeta)}{G_2(\zeta)}d\Om_\pm(\zeta,k_0) =
\begin{pmatrix} \oint_{\dD}G_1(\zeta) \, d\Om_\pm(\zeta,k_0)  \\
\oint_{\dD}G_2(\zeta) \, d\Om_\pm(\zeta,k_0) \end{pmatrix}.
\end{align}

First, we prove \eqref{3.95} for the right half-lattice Laurent polynomials
and for $k_0$ even. In this case \eqref{3.49a} and
\eqref{3.50a} imply that \eqref{3.95} is equivalent to
\begin{align}
&\binom{Q_+(z,k,k_0)}{S_+(z,k,k_0)} = \oint_{\dD}
\frac{\zeta+z}{\zeta-z} \bigg(
\binom{P_+(\zeta,k,k_0)}{R_+(\zeta,k,k_0)} -
\binom{P_+(z,k,k_0)}{R_+(z,k,k_0)} \bigg) d\Om_{+}(\zeta,k_0), \quad k
> k_0, \text{ $k_0$ even}.  \lb{3.97}
\end{align}
Let $k_0\in\Z$ be even. It suffices to show that the right-hand side
of \eqref{3.97}, temporarily denoted by the symbol $RHS(z,k,k_0)$,
satisfies
\begin{align}
&  \T(z,k+1)^{-1} RHS(z,k+1,k_0) = RHS(z,k,k_0),
\quad k >  k_0, \lb{3.98} \\
&  \T(z,k_0+1)^{-1}RHS(z,k_0+1,k_0) =
\binom{Q_+(z,k_0,k_0)}{S_+(z,k_0,k_0)} = \binom{-I_m}{I_m}. \lb{3.99}
\end{align}
One verifies these statements using the equality,
\begin{align} \lb{3.100}
& \T(z,k+1)^{-1} RHS(z,k+1,k_0) = RHS(z,k,k_0)
\\ \no
& \quad + \oint_{\dD} \frac{\zeta+z}{\zeta-z} \left(
\T(z,k+1)^{-1} - \T(\zeta,k+1)^{-1}\right)
\binom{P_+(\zeta,k+1,k_0)}{R_+(\zeta,k+1,k_0)}d\Om_{+}(\zeta,k_0),
\quad k\in\Z.
\end{align}
For $k>k_0$, the last term on the right-hand side of \eqref{3.100}
vanishes since for $k$ odd, $\T(z,k+1)$ does not depend on $z$, and
for $k$ even, it follows from Corollary \ref{c3.8} that
$P_+(\cdot,k+1,k_0)$ and $R_+(\cdot,k+1,k_0)$ are orthogonal in
$\Ltm{+}$ to $\spn\{I_m,\zeta I_m\}$ and
$\spn\{I_m,\zeta^{-1} I_m\}$, respectively. Hence one computes
\begin{align}
&\oint_{\dD} \frac{\zeta+z}{\zeta-z} \Big(\T(z,k+1)^{-1}-
\T(\zeta,k+1)^{-1}\Big)
\binom{P_+(\zeta,k+1,k_0)}{R_+(\zeta,k+1,k_0)}d\Om_{+}(\zeta,k_0)
\no
\\
& \quad = \oint_{\dD} \frac{\zeta+z}{\zeta-z}
\begin{pmatrix}
0 & (z-\zeta)\rho_{k+1}^{-1} \\
(z^{-1} - \zeta^{-1})\wti\rho_{k+1}^{-1} & 0
\end{pmatrix}
\binom{P_+(\zeta,k+1,k_0)}{R_+(\zeta,k+1,k_0)}d\Om_{+}(\zeta,k_0)
\no
\\
& \quad = \oint_{\dD}
\begin{pmatrix}
0 & -(\zeta+z)\rho_{k+1}^{-1} \\
(\zeta^{-1} + z^{-1})\wti\rho_{k+1}^{-1} & 0
\end{pmatrix}
\binom{P_+(\zeta,k+1,k_0)}{R_+(\zeta,k+1,k_0)}d\Om_{+}(\zeta,k_0)  \no
\\ & \quad = \oint_{\dD}
\binom{-\rho_{k+1}^{-1}(\zeta+z)R_+(\zeta,k,k_0)}
{\wti\rho_{k+1}^{-1}(\zeta^{-1}+z^{-1})P_+(\zeta,k,k_0)}
d\Om_{+}(\zeta,k_0) = \binom{0}{0}.
\end{align}
Thus, \eqref{3.98} is implied by \eqref{3.100}.

For $k=k_0$ even, one obtains that $RHS(z,k_0,k_0)=0$ since by \eqref{3.49} one has the normalization
$P_+(z,k_0,k_0)=R_+(z,k_0,k_0)=I_m$. Then using the fact that by
Corollary \ref{c3.8}, $P_+(\cdot,k_0+1,k_0)$ and
$R_+(\cdot,k_0+1,k_0)$ are orthogonal to constant $m\times m$
matrices in $\Ltm{+}$ and that by \eqref{Table},
\begin{align}
\begin{split}
P_+(\zeta,k_0+1,k_0) &=
\wti\rho_{k_0+1}^{-1}(\zeta I_m+\al_{k_0+1}),\\
R_+(\zeta,k_0+1,k_0) &=
\rho_{k_0+1}^{-1}(\zeta^{-1}I_m+\al_{k_0+1}^*),
\end{split}
\end{align}
one computes,
\begin{align}
&\oint_{\dD} \zeta^{-1}P_+(\zeta,k_0+1,k_0)d\Om_+(\zeta,k_0) =
\oint_{\dD} P_+(\zeta,k_0+1,k_0)d\Om_+(\zeta,k_0) (\zeta I_m)^*
\no
\\ &\quad =
\oint_{\dD} P_+(\zeta,k_0+1,k_0)d\Om_+(\zeta,k_0)
P_+(\zeta,k_0+1,k_0)^*\wti\rho_{k_0+1} = \wti\rho_{k_0+1},
\lb{3.104}
\\
&\oint_{\dD} \zeta R_+(\zeta,k_0+1,k_0)d\Om_+(\zeta,k_0) =
\oint_{\dD} R_+(\zeta,k_0+1,k_0)d\Om_+(\zeta,k_0) (\zeta^{-1}I_m)^*
\no
\\ &\quad =
\oint_{\dD} R_+(\zeta,k_0+1,k_0)d\Om_+(\zeta,k_0)
R_+(\zeta,k_0+1,k_0)^*\rho_{k_0+1} = \rho_{k_0+1}, \lb{3.105}
\end{align}
and hence,
\begin{align} \lb{3.106}
&\oint_{\dD} \frac{\zeta+z}{\zeta-z}
\Big(\T(z,k_0+1)^{-1}-\T(\zeta,k_0+1)^{-1}\Big)
\binom{P_+(\zeta,k_0+1,k_0)}{R_+(\zeta,k_0+1,k_0)}d\Om_{+}(\zeta,k_0)
\no
\\
& \quad = \oint_{\dD}
\binom{-\rho_{k+1}^{-1}(\zeta+z)R_+(\zeta,k_0+1,k_0)}
{\wti\rho_{k+1}^{-1}(\zeta^{-1}+z^{-1})P_+(\zeta,k_0+1,k_0)}
d\Om_{+}(\zeta,k_0)
\\ & \quad =
\oint_{\dD} \binom{-\rho_{k+1}^{-1}\zeta R_+(\zeta,k_0+1,k_0)}
{\wti\rho_{k+1}^{-1}\zeta^{-1}P_+(\zeta,k_0+1,k_0)}
d\Om_{+}(\zeta,k_0) =
\binom{-\rho_{k+1}^{-1}\rho_{k+1}}{\wti\rho_{k+1}^{-1}\wti\rho_{k+1}}
= \binom{-I_m}{I_m}. \no
\end{align}
Thus, \eqref{3.99} is a consequence of \eqref{3.100}, \eqref{3.106}, and the
fact that $RHS(z,k_0,k_0)=0$.

Next, we prove \eqref{3.95} for the right half-lattice Laurent polynomials
and $k_0$ odd,
\begin{align}
&\binom{S_+(z,k,k_0)}{\wti Q_+(z,k,k_0)} = \oint_{\dD}
\frac{\zeta+z}{\zeta-z} \bigg(\binom{R_+(\zeta,k,k_0)}{\wti
P_+(\zeta,k,k_0)} - \binom{R_+(z,k,k_0)}{\wti P_+(z,k,k_0)} \bigg)
d\Om_{+}(\zeta,k_0), \quad k > k_0, \text{ $k_0$ odd}. \lb{3.107}
\end{align}
Let $k_0\in\Z$ be odd. We note that for $U(z,k),V(z,k)\in\Cm$,
$k\in\bbZ$, $z\in\Cz$,
\begin{align}
\binom{U(z,k)}{V(z,k)} &=  \T(z,k) \binom{U(z,k-1)}{V(z,k-1)}
\\
\intertext{is equivalent to } \binom{V(z,k)}{\wti U(z,k)} &= \wti
\T(z,k) \binom{V(z,k-1)}{\wti U(z,k-1)},
\end{align}
where
\begin{equation}
\wti U(z,k) = z^{-1} U(z,k) \,\text{ and }\, \wti \T(z,k) =
\begin{pmatrix}0 & I_m \\ z^{-1}I_m & 0\end{pmatrix} \T(z,k)
\begin{pmatrix}0 & zI_m \\ I_m & 0\end{pmatrix}.
\end{equation}
Thus, it suffices to show that the right-hand side of \eqref{3.107},
temporarily denoted by $\wti{RHS}(z,k,k_0)$, satisfies
\begin{align}
& \wti \T(z,k+1)^{-1} \wti{RHS} (z,k+1,k_0) = \wti{RHS}(z,k,k_0),
\quad k > k_0,
\\
& \wti \T(z,k_0+1)^{-1} \wti{RHS}(z,k_0+1,k_0) =
\binom{S_+(z,k_0,k_0)}{\wti Q_+(z,k_0,k_0)} = \binom{-I_m}{I_m}.
\end{align}
At this point one can follow the first part of the proof replacing
$\T$ by $\wti \T$, $\Big(\begin{smallmatrix} P_+ \\[1mm]
R_+\end{smallmatrix}\Big)$ by $\Big(\begin{smallmatrix} R_+ \\[.5mm]
\wti P_+\end{smallmatrix}\Big)$, $\Big(\begin{smallmatrix} Q_+
\\[1mm] S_+\end{smallmatrix}\Big)$ by $\Big(\begin{smallmatrix}
S_+ \\[.5mm] \wti Q_+\end{smallmatrix}\Big)$, etc.

The result for the remaining Laurent polynomials $\wti P_-(z,k,k_0)$,
$R_-(z,k,k_0)$, $\wti Q_-(z,k,k_0)$, and $S_-(z,k,k_0)$ is proved 
similarly.
\end{proof}

\begin{lemma} \lb{l3.11}
Let $k_0\in\bbZ$ and let $m_\pm(\cdot,k_0)$ denote the 
$\Cm$-valued Caratheodory and anti-Caratheodory functions
\begin{align}
m_\pm(z,k_0) &= \pm \De_{k_0}^*(\U_{\pm,k_0}+zI)(\U_{\pm,k_0}-zI)^{-1}
\De_{k_0}  \lb{3.110}
\\
& =\pm
\oint_{\dD}d\Om_{\pm}(\zeta,k_0)\,\frac{\zeta+z}{\zeta-z},
\quad z\in\bbC\backslash\dD,   \lb{3.111}
\end{align}
with
\begin{equation}
 \oint_{\dD}d\Om_{\pm}(\zeta,k_0)= I_m. \lb{3.111a}
\end{equation}
Then the following relations hold,
\begin{align}
Q_\pm(z,\cdot,k_0) + P_\pm(z,\cdot,k_0)m_\pm(z,k_0) \in
\ltmm{[k_0,\pm\infty)\cap\Z}, \quad z\in\bbC\backslash(\dD\cup\{0\}),
\lb{3.112}
\\
S_\pm(z,\cdot,k_0) + R_\pm(z,\cdot,k_0)m_\pm(z,k_0) \in
\ltmm{[k_0,\pm\infty)\cap\Z}, \quad z\in\bbC\backslash(\dD\cup\{0\}).
\lb{3.113}
\end{align}
\end{lemma}
\begin{proof}
Equality \eqref{3.111} is implied by \eqref{3.54} and \eqref{3.55}. \\
Next, let $\bbB_{\pm,k_0}(z)$ denote operators defined on
$\ltm{[k_0,\pm\infty)\cap\Z}$ by
\begin{align}
\bbB_{\pm,k_0}(z) = (\U_{\pm,k_0}+zI)(\U_{\pm,k_0}-zI)^{-1}, \quad
z\in\bbC\backslash\dD.
\end{align}
Since $\U_{\pm,k_0}$ are unitary, the operators $\bbB_{\pm,k_0}(z)$
are bounded for all $z\in\bbC\backslash\dD$, and hence one has
\begin{align}
\bbB_{\pm,k_0}(z)\De_{k_0}=
\big\{\De_{k}^*\bbB_{\pm,k_0}(z)\De_{k_0}\big\}_{k\in[k_0,\pm\infty)\cap\Z}
\in\ltmm{[k_0,\pm\infty)\cap\Z}. \lb{3.115}
\end{align}
Using the spectral representation for the operators
$\bbB_{\pm,k_0}(z)$ and equalities \eqref{3.64}, \eqref{3.93}, and \eqref{3.111}, one obtains
\begin{align}
\De_{k}^*\bbB_{\pm,k_0}(z)\De_{k_0} &=
\oint_{\dD}\frac{\zeta+z}{\zeta-z} \wti
P_\pm(\zeta,k,k_0) \, d\Om_{\pm}(\zeta,k_0) \no
\\
&= \pm\big[\wti Q_\pm(z,k,k_0)+ \wti P_\pm(z,k,k_0)
m_\pm(z,k_0)\big], \quad k \gtrless k_0. \lb{3.116}
\end{align}
Thus, \eqref{3.112} is a consequence of \eqref{3.49a}, \eqref{3.50a},
\eqref{3.115}, and \eqref{3.116}.

Moreover, since $\Big(\begin{smallmatrix}P_\pm(z,k,k_0)\\[1mm]
R_\pm(z,k,k_0)\end{smallmatrix}\Big)_{k\in\Z}$ and
$\Big(\begin{smallmatrix} Q_\pm(z,k,k_0)\\[1mm] S_\pm(z,k,k_0)
\end{smallmatrix}\Big)_{k\in\Z}$,
$z\in\bbC\backslash\{0\}$, satisfy \eqref{3.21}, Lemma \ref{l3.2} implies that
\begin{align}
(\W (Q_\pm (z,\cdot,k_0)+P_\pm (z,\cdot,k_0)m_\pm(z,k_0)))(k) =
z[S_\pm (z,k,k_0)+R_\pm (z,k,k_0)m_\pm(z,k_0)], \quad k\in\Z, \lb{3.117}
\end{align}
and hence \eqref{3.113} follows from \eqref{3.112} and \eqref{3.117}.
\end{proof}

\begin{lemma} \lb{l3.12}
Let $k_0\in\bbZ$. Then the relations in \eqref{3.112} $($equivalently, those in
\eqref{3.113}$)$ uniquely determine the $\Cm$-valued functions
$m_\pm(\cdot,k_0)$ on $\bbC\backslash(\dD\cup\{0\})$.
\end{lemma}
\begin{proof}
We will prove the lemma by contradiction. Assume that there are two
$\Cm$-valued functions $m_+(z,k_0)$ and $\wti m_+(z,k_0)$ satisfying
\eqref{3.112} such that $m_+(z_0,k_0) \neq \wti m_+(z_0,k_0)$ for
some $z_0\in\C\backslash(\dD\cup\{0\})$. Then there is a vector
$x\in\C^m$ such that $(m_+(z_0,k_0) - \wti m_+(z_0,k_0))x \neq 0$ and
by \eqref{3.112},
\begin{align}
p_+(z_0,\cdot,k_0)=P_+(z_0,\cdot,k_0)[m_+(z_0,k_0)-\wti
m_+(z_0,k_0)]x \in \ltm{[k_0,\pm\infty)\cap\Z}, \quad
z\in\bbC\backslash(\dD\cup\{0\}).
\end{align}
Since $P_+(z_0,k_0,k_0)\neq 0$, the sequence of vectors
$\{p_+(z,k,k_0)\}_{k\geq k_0}$ is not identically zero, and hence, by
Lemma \ref{l3.3}, $p_+(z_0,\cdot,k_0)$ is an eigenvector of the
operator $\U_{+,k_0}$ corresponding to the eigenvalue
$z_0\in\C\backslash\dD$. This contradicts unitarity of $\U_{+,k_0}$.

Similarly, one proves the result for $m_-(z,k_0)$. Moreover, one
easily supplies a proof that utilizes \eqref{3.113} instead of
\eqref{3.112}.
\end{proof}

\begin{corollary} \lb{c3.13}
There are solutions $\Big(\begin{smallmatrix}\psi_\pm(z,k)\\[1mm]
\chi_\pm(z,k)\end{smallmatrix}\Big)$, $k\in\Z$, of \eqref{3.21},
unique up to right-multiplication by constant $m\times m$ matrices, so that for some $($and hence for
all\,$)$ $k_1\in\Z$,
\begin{align}
\psi_\pm(z,\cdot),\,\chi_\pm(z,\cdot) \in
\ell^2 ([k_1,\pm\infty)\cap\Z)^{m\times m}, \quad
z\in\C\backslash(\dD\cup\{0\}).
\end{align}
\end{corollary}
\begin{proof}
Since any solution of \eqref{3.21} can be expressed as a linear
combination of the Laurent polynomials
$\Big(\begin{smallmatrix}P_\pm(z,\cdot,k_0)\\[1mm]
R_\pm(z,\cdot,k_0)\end{smallmatrix}\Big)$ and
$\Big(\begin{smallmatrix}Q_\pm(z,\cdot,k_0)\\[1mm]
S_\pm(z,\cdot,k_0)\end{smallmatrix}\Big)$, existence and uniqueness
of the solutions
$\Big(\begin{smallmatrix}\psi_\pm(z,\cdot)\\[1mm]
\chi_\pm(z,\cdot)\end{smallmatrix}\Big)$ is implied by Lemmas  
\ref{l3.11} and \ref{l3.12}, respectively.
\end{proof}

For the next result we recall the definition of $a_k$ and $b_k$ in \eqref{3.10} 
and \eqref{3.11}.

\begin{lemma} \lb{l3.14}
Let $z\in\bbC\backslash\{0\}$ and $k_0\in\bbZ$. Then the following
relations hold for all $k\in\Z$,
\begin{align}
\binom{P_-(z,k,k_0-1)}{R_-(z,k,k_0-1)} &=
\binom{P_+(z,k,k_0)}{R_+(z,k,k_0)}
\frac{\wti\rho_{k_0}^{-1}b_{k_0}-\rho_{k_0}^{-1}b_{k_0}^*}{2} +
\binom{Q_+(z,k,k_0)}{S_+(z,k,k_0)}
\frac{\wti \rho_{k_0}^{-1}b_{k_0}+\rho_{k_0}^{-1}b_{k_0}^*}{2}, \lb{3.120}
\\
\binom{Q_-(z,k,k_0-1)}{S_-(z,k,k_0-1)} &=
\binom{P_+(z,k,k_0)}{R_+(z,k,k_0)}
\frac{\wti\rho_{k_0}^{-1}a_{k_0}+\rho_{k_0}^{-1}a_{k_0}^*}{2} +
\binom{Q_+(z,k,k_0)}{S_+(z,k,k_0)}
\frac{\wti \rho_{k_0}^{-1}a_{k_0}-\rho_{k_0}^{-1}a_{k_0}^*}{2}, \lb{3.121}
\end{align}
and
\begin{align}
\binom{P_-(z,k,k_0)}{R_-(z,k,k_0)} &=
\binom{P_+(z,k,k_0)}{R_+(z,k,k_0)} c(z,k_0) +
\binom{Q_+(z,k,k_0)}{S_+(z,k,k_0)} d(z,k_0), \lb{3.122}
\\
\binom{Q_-(z,k,k_0)}{S_-(z,k,k_0)} &=
\binom{P_+(z,k,k_0)}{R_+(z,k,k_0)} d(z,k_0) +
\binom{Q_+(z,k,k_0)}{S_+(z,k,k_0)} c(z,k_0), \lb{3.123}
\end{align}
where
\begin{align} \lb{3.124}
c(z,k_0) =
\begin{cases}
\frac{1-z}{2z}, & k_0 \text{ odd},\\
\frac{1-z}{2}, & k_0 \text{ even}
\end{cases}
\,\text{ and }\, d(z,k_0) =
\begin{cases}
\frac{1+z}{2z}, & k_0 \text{ odd},\\
\frac{1+z}{2}, & k_0 \text{ even}.
\end{cases}
\end{align}
\end{lemma}
\begin{proof}
Since the left and right-hand sides of \eqref{3.120}--\eqref{3.123}
satisfy the same recursion relation \eqref{3.21}, it suffices to
check \eqref{3.120}--\eqref{3.123} at only one point, say, at the
point $k=k_0$. The latter is easily seen to be a consequence of \eqref{Table}.
\end{proof}

\begin{theorem} \lb{t3.15}
Let $k_0\in\bbZ$. Then there exist unique $\Cm$-valued functions
$M_\pm(\cdot,k_0)$ such that for all
$z\in\bbC\backslash(\dD\cup\{0\})$
\begin{align}
&U_\pm(z,\cdot,k_0) = Q_+(z,\cdot,k_0) + P_+(z,\cdot,k_0)M_\pm(z,k_0)
\in \ltmm{[k_0,\pm\infty)\cap\Z}, \lb{3.125}
\\
&V_\pm(z,\cdot,k_0) = S_+(z,\cdot,k_0) + R_+(z,\cdot,k_0)M_\pm(z,k_0)
\in \ltmm{[k_0,\pm\infty)\cap\Z}. \lb{3.126}
\end{align}
\end{theorem}
\begin{proof}
The assertions \eqref{3.125} and \eqref{3.126} follow from
Lemma \ref{l3.11}, Corollary \ref{c3.13}, and Lemma \ref{l3.14}.
\end{proof}

We will call $U_\pm(z,\cdot,k_0)$ the {\it Weyl--Titchmarsh solutions} of
$\U$. By Corollary \ref{c3.13}, $U_\pm(z,\cdot,k_0)$ and
$V_\pm(z,\cdot,k_0)$ are unique up to right-multiplication by constant $m\times m$ matrices. Similarly,
we will call $m_\pm(z,k_0)$ as well as $M_\pm(z,k_0)$ the {\it
half-lattice Weyl--Titchmarsh $m$-functions} associated with
$\U_{\pm,k_0}$. (See also \cite{Si04a} for a comparison of various
alternative notions of Weyl--Titchmarsh $m$-functions for
$\U_{+,k_0}$ with scalar-valued Verblunsky coefficients.)

Lemma \ref{l3.11}, Corollary \ref{c3.13}, and Lemma
\ref{l3.14} imply that for all $z\in\bbC\backslash\dD$,
\begin{align}
M_+(z,k_0) &= m_+(z,k_0), \lb{3.127}
\\
M_+(0,k_0) &=I_m, \lb{3.128}
\\
M_-(z,k_0) &=[(1+z)I_m + (1-z)m_-(z,k_0)][(1-z)I_m +
(1+z)m_-(z,k_0)]^{-1}, \lb{3.129}
\\
M_-(z,k_0) &=
\big[(\wti\rho_{k_0}^{-1}a_{k_0}+\rho_{k_0}^{-1}a_{k_0}^*) +
(\wti\rho_{k_0}^{-1}b_{k_0}-\rho_{k_0}^{-1}b_{k_0}^*)m_-(z,k_0-1)\big]
\no
\\
&\quad\times
\big[(\wti\rho_{k_0}^{-1}a_{k_0}-\rho_{k_0}^{-1}a_{k_0}^*) +
(\wti\rho_{k_0}^{-1}b_{k_0}+\rho_{k_0}^{-1}b_{k_0}^*)m_-(z,k_0-1)\big]^{-1},
\lb{3.130}
\\
M_-(0,k_0) &=(\al_{k_0}+I_m)(\al_{k_0}-I_m)^{-1}. \lb{3.131}
\end{align}
In addition, it follows from \eqref{3.111} and Theorem \ref{tA.2}
that $m_\pm(z,k_0)$ are $\Cm$-valued Caratheodory and
anti-Caratheodory functions, respectively. From \eqref{3.127} one
concludes that $M_+(z,k_0)$ is also a Caratheodory function. Using
\eqref{3.129} one verifies that $M_-(z,k_0)$ is analytic in $\D$
since the anti-Caratheodory function $m_-(\cdot,k_0)$ satisfies
$\Re(m_-(z,k_0))=(m_-(z,k_0)+m_-(z,k_0)^*)/2 < 0$ for $z\in\D$.
Moreover, utilizing \eqref{3.12}, \eqref{3.13}, and \eqref{3.130},
one computes,
\begin{align}
\Re(M_-(z,k_0)) &= [M_-(z,k_0)+M_-(z,k_0)^*]/2 \no
\\
&= \big[(a_{k_0}^*\wti\rho_{k_0}^{-1}-a_{k_0}\rho_{k_0}^{-1}) +
m_-(z,k_0-1)^*(b_{k_0}^*\wti\rho_{k_0}^{-1}+b_{k_0}\rho_{k_0}^{-1})\big]^{-1}
\Re(m_-(z,k_0-1)) \no
\\
&\quad\times
\big[(\wti\rho_{k_0}^{-1}a_{k_0}-\rho_{k_0}^{-1}a_{k_0}^*) +
(\wti\rho_{k_0}^{-1}b_{k_0}+\rho_{k_0}^{-1}b_{k_0}^*)m_-(z,k_0-1)\big]^{-1},
\end{align}
and hence, $M_-(\cdot,k_0)$ is an anti-Caratheodory matrix.

Next, we introduce the $\Cm$-valued functions $\Phi_\pm(\cdot,k)$,
$k\in\bbZ$, by
\begin{align}
\Phi_\pm(z,k) = [M_\pm(z,k)-I_m][M_\pm(z,k)+I_m]^{-1}, \quad
z\in\C\backslash\dD. \lb{3.133}
\end{align}
Then \eqref{3.128} and \eqref{3.131} imply that
\begin{equation} \lb{3.133a}
\Phi_+(0,k_0) = 0 \,\text{ and }\, \Phi_-(0,k_0)^{-1}=\al_{k_0}.
\end{equation}
Moreover, one verifies that
\begin{align}
M_\pm(z,k) &= [I_m-\Phi_\pm(z,k)]^{-1}[I_m+\Phi_\pm(z,k)], \quad
z\in\C\backslash\dD, \lb{3.134}
\\
m_-(z,k) &= [zI_m+\Phi_-(z,k)]^{-1}[zI_m-\Phi_-(z,k)], \quad
z\in\C\backslash\dD \lb{3.135}
\end{align}
(cf.\ Remark \ref{r3.18}). In addition, we extend these functions to the unit circle $\dD$ by
taking the radial limits which exist and are finite for Lebesgue
almost every $\ze\in\dD$,
\begin{align}
M_\pm(\ze,k) &= \lim_{r \uparrow 1} M_\pm(r\ze,k),
\\
\Phi_\pm(\ze,k) &= \lim_{r \uparrow 1} \Phi_\pm(r\ze,k), \quad
k\in\Z.
\end{align}

\begin{lemma} \lb{l3.16}
Let $z\in\C\backslash(\dD\cup\{0\})$, $k_0, k\in\Z$. Then the
functions $\Phi_\pm(\cdot,k)$ satisfy
\begin{equation}
\Phi_\pm(z,k) =
\begin{cases}
zV_\pm(z,k,k_0)U_\pm(z,k,k_0)^{-1}, &\text{$k$ odd,}
\\
U_\pm(z,k,k_0)V_\pm(z,k,k_0)^{-1}, & \text{$k$ even,}
\end{cases} 
\end{equation}
where $U_\pm(\cdot,k,k_0)$ and $V_\pm(\cdot,k,k_0)$ are the
$\Cm$-valued functions defined in \eqref{3.125} and \eqref{3.126},
respectively.
\end{lemma}
\begin{proof}
Using Corollary \ref{c3.13} it suffices to assume $k=k_0$. Then the
statement is immediately implied by \eqref{3.49}, \eqref{3.125},
\eqref{3.126}, and \eqref{3.133}.
\end{proof}

\begin{lemma} \lb{l3.17}
Let $k\in\bbZ$. Then the $\Cm$-valued functions
$\Phi_+(\cdot,k)|_{\D}$ $($resp., $\Phi_-(\cdot,k)|_{\D}$$)$ are
Schur $($resp., anti-Schur\,$)$ matrices. Moreover, $\Phi_\pm$
satisfy the Riccati-type equation
\begin{equation}
\Phi_\pm(z,k)\wti\rho_k^{-1}\al_k\Phi_\pm(z,k-1) +
z\Phi_\pm(z,k)\wti\rho_k^{-1} - \rho_k^{-1}\Phi_\pm(z,k-1)
=z\rho_k^{-1}\al_k^*, \quad z\in\bbC\backslash\dD,\; k\in\Z.
\lb{3.138}
\end{equation}
\end{lemma}
\begin{proof}
 Lemma \ref{l3.16} and \eqref{3.133} imply that the
functions $\Phi_+(\cdot,k)|_{\D}$ $($resp.,
$\Phi_-(\cdot,k)|_{\D}$$)$ are Schur $($resp., anti-Schur\,$)$
matrices.

Let $k$ be odd. Then applying Lemma \ref{l3.16} and the recursion
relation \eqref{3.21} one obtains
\begin{align}
&\Phi_\pm(z,k) = zV_\pm(z,k,k_0)U_\pm(z,k,k_0)^{-1} \no
\\
&\quad =
\rho_k^{-1}\big[U_\pm(z,k-1,k_0)+z\al_k^*V_\pm(z,k-1,k_0)\big]
\big[\al_kU_\pm(z,k-1,k_0)+zV_\pm(z,k-1,k_0)\big]^{-1}\wti\rho_k \no
\\
&\quad = \rho_k^{-1}\big[\Phi_\pm(z,k-1)+z\al_k^*\big]
\big[\al_k\Phi_\pm(z,k-1)+zI_m\big]^{-1}\wti\rho_k.   \lb{3.139}
\end{align}
For $k$ even, one similarly obtains
\begin{align}
&\Phi_\pm(z,k) = U_\pm(z,k,k_0)V_\pm(z,k,k_0)^{-1} \no
\\
&\quad =
\rho_k^{-1}\big[\al_k^*U_\pm(z,k-1,k_0)+V_\pm(z,k-1,k_0)\big]
\big[U_\pm(z,k-1,k_0)+\al_kV_\pm(z,k-1,k_0)\big]^{-1}\wti\rho_k  \no
\\
&\quad = \rho_k^{-1}\big[z\al_k^*+\Phi_\pm(z,k-1)\big]
\big[zI_m+\al_k\Phi_\pm(z,k-1)\big]^{-1}\wti\rho_k.
\end{align}
\end{proof}

\begin{remark} \lb{r3.18}
$(i)$ In the special case $\al=\{\al_k\}_{k\in\Z}=0$, one obtains
\begin{equation}
M_\pm(z,k) = \pm I_m, \quad \Phi_+(z,k)=0, \quad \Phi_-(z,k)^{-1}=0,
\quad z\in\C, \; k\in\Z.
\end{equation}
Thus, strictly speaking, one should always consider $\Phi_-^{-1}$
rather than $\Phi_-$ and hence refer to the Riccati-type equation of
$\Phi_-^{-1}$,
\begin{equation}
z\Phi_-(z,k)^{-1}\rho_k^{-1}\al_k^*\Phi_-(z,k-1)^{-1} +
\Phi_-(z,k)^{-1}\rho_k^{-1} - z\wti\rho_k^{-1}\Phi_-(z,k-1)^{-1} =
\wti\rho_k^{-1}\al_k, \quad z\in\C\backslash\dD, \; k\in\bbZ,
\lb{3.142}
\end{equation}
rather than that of $\Phi_-$, etc. In fact, since $M_-(\cdot,k)$ is an anti-Caratheodory matrix and hence $[M_-(z,k)-I_m]$ is invertible (cf.\ \cite[p.\ 137]{SF70}), we should have introduced the Schur matrix
\begin{equation}
\Phi_-(z,k)^{-1} = [M_-(z,k) + I_m] [M_-(z,k) - I_m]^{-1},  \quad
z\in\C\backslash\dD,
\end{equation}
rather than the anti-Schur matrix $\Phi_-(\cdot,k)$. For simplicity of notation, we
will typically avoid this complication with $\Phi_-$ and still invoke
$\Phi_-$ rather than $\Phi_-^{-1}$ whenever confusions are unlikely. \\
$(ii)$ We note that $\Phi_\pm(z,k)^{\pm 1}$, $z\in\dD$, $k\in\bbZ$, have
nontangential limits to $\dD$ Lebesgue almost everywhere. In
particular, the Riccati-type equations \eqref{3.138}, and
\eqref{3.142} extend to $\dD$ Lebesgue almost everywhere.
\end{remark}

The Riccati-type equation \eqref{3.138} for the Schur matrix $\Phi_+$ implies the 
norm convergent expansion,
\begin{align}
\Phi_+(z,k)&=\sum_{j=1}^\infty \phi_{+,j}(k) z^j, \quad z\in\D,
\; k\in\Z, \\
\phi_{+,1}(k)&=-\al_{k+1}^*, \no \\
\phi_{+,2}(k)&=-\rho_{k+1} \al_{k+2}^* \wti\rho_{k+1}, \lb{3.144} \\
\phi_{+,j}(k)&=\sum_{\ell=1}^{j-1}
\rho_{k+1}\phi_{+,j-\ell}(k+1)\wti\rho_{k+1}^{-1}\al_{k+1}\phi_{+,\ell}(k)
+ \rho_{k+1}\phi_{+,j-1}(k+1)\wti\rho_{k+1}^{-1}, \; j\geq 3. \no
\end{align}
Similarly, the corresponding Riccati-type equation \eqref{3.142} for the Schur matrix
$\Phi_-^{-1}$ implies the norm convergent expansion
\begin{align}
\Phi_-(z,k)^{-1} &=\sum_{j=0}^\infty \phi_{-,j}(k) z^j, \quad
z\in\D, \; k\in\Z, \\
\phi_{-,0}(k)&=\al_{k}, \no \\
\phi_{-,1}(k)&=\wti\rho_{k}\al_{k-1}\rho_{k},
\lb{3.146} \\
\phi_{-,j}(k)&=-\sum_{\ell=0}^{j-1}
\phi_{-,j-1-\ell}(k-1)\rho_k^{-1}\al_{k}^*\phi_{-,\ell}(k)\rho_k +
\wti\rho_k^{-1}\phi_{-,j-1}(k-1)\rho_k, \; j\geq 2. \no
\end{align}

\section{Weyl--Titchmarsh Theory for CMV Operators on $\bbZ$ with Matrix-valued Verblunsky Coefficients} \lb{s4}

In this section we present the basics of Weyl--Titchmarsh
theory for CMV operators on the lattice $\bbZ$ with matrix-valued Verblunsky coefficients. The corresponding case of scalar-valued Verblunsky coefficients was dealt with in detail in \cite{GZ06}.

We start by introducing the $\Cm$-valued CMV Wronskian of two $\Cm$-valued sequences $U_j(z,\cdot)$, $j=1,2$,
\begin{align}
W(U_1(1/\ol{z},k),U_2(z,k)) &= \frac{(-1)^{k+1}}{2} \Big[
U_1(1/\ol{z},k)^*U_2(z,k)-(\V^*U_1(1/\ol{z},\dott))(k)^*
(\V^*U_2(z,\dott))(k)\Big],   \no
\\
& \hspace*{7cm}   k\in\bbZ, \; z\in\Cz,  \lb{4.1A}
\end{align}
Next we verify that the Wronskian of any
two solutions of $\U U(z,\cdot)=zU(z,\cdot)$ is indeed $k$-independent as expected:

\begin{lemma}
Suppose $U_j(z,\cdot)$ satisfy $\U U_j(z,\cdot)=zU_j(z,\cdot)$,
$j=1,2$, where $\U$ is understood as a difference expression $($rather
then an operator in $\ell^2(\bbZ)^{m\times m}$$)$.\ Then the Wronskian in \eqref{4.1A} is independent of $k\in\bbZ$ and the following identities hold:
\begin{align}
W(U_1(1/\ol{z},k),U_2(z,k)) &= \frac{(-1)^{k+1}}{2}
\Big[U_1(1/\ol{z},k)^*U_2(z,k)-V_1(1/\ol{z},k)^*V_2(z,k)\Big] \no
\\
&= \frac{(-1)^{k+1}}{2}
\begin{pmatrix}
U_1(1/\ol{z},k) \\ V_1(1/\ol{z},k)
\end{pmatrix}^*
\begin{pmatrix}
I_m & 0 \\ 0 & -I_m
\end{pmatrix}
\begin{pmatrix}
U_2(z,k) \\ V_2(z,k)
\end{pmatrix}, \quad k\in\bbZ, \; z\in\Cz, \lb{4.2A}
\end{align}
where $V_j(z,\cdot)=\V^* U_j(z,\cdot)$, $j=1,2$, and
\begin{align}
W(P_+(1/\ol{z},k,k_0),Q_+(z,k,k_0)) &= I_m, \lb{4.3A}
\\
W(U_+(1/\ol{z},k,k_0),U_-(z,k,k_0)) &= M_+(z,k_0)-M_-(z,k_0), \quad
k,k_0\in\bbZ, \; z\in\Cz. \lb{4.4A}
\end{align}
\end{lemma}
\begin{proof}
First, we note that \eqref{4.2A} is implied by 
\eqref{4.1A}. Next, employing Lemma \ref{l3.2},
$U_j$ and $V_j$, $j=1,2$, satisfy the recursion relation
\begin{align}
\binom{U_j(z,k)}{V_j(z,k)} = \T(z,k) \binom{U_j(z,k-1)}{V_j(z,k-1)},
\quad j=1,2,\; k\in\Z,\; z\in\Cz,
\end{align}
and hence
\begin{align}
&W(U_1(1/\ol{z},k),U_2(z,k)) = \frac{(-1)^{k+1}}{2}
\begin{pmatrix}
U_1(1/\ol{z},k) \\ V_1(1/\ol{z},k)
\end{pmatrix}^*
\begin{pmatrix}
I_m & 0 \\ 0 & -I_m
\end{pmatrix}
\begin{pmatrix}
U_2(z,k) \\ V_2(z,k)
\end{pmatrix} \no
\\
&\qquad = \frac{(-1)^{k+1}}{2}
\begin{pmatrix}
U_1(1/\ol{z},k-1) \\ V_1(1/\ol{z},k-1)
\end{pmatrix}^*
\T(1/\ol{z},k)^*
\begin{pmatrix}
I_m & 0 \\ 0 & -I_m
\end{pmatrix}
\T(z,k)
\begin{pmatrix}
U_2(z,k-1) \\ V_2(z,k-1)
\end{pmatrix} \no
\\
&\qquad = -\frac{(-1)^{k}}{2}
\begin{pmatrix}
U_1(1/\ol{z},k-1) \\ V_1(1/\ol{z},k-1)
\end{pmatrix}^*
\begin{pmatrix}
-I_m & 0 \\ 0 & I_m
\end{pmatrix}
\begin{pmatrix}
U_2(z,k-1) \\ V_2(z,k-1)
\end{pmatrix}
\\
&\qquad = W(U_1(1/\ol{z},k-1),U_2(z,k-1)), \quad k\in\bbZ,\;
z\in\Cz. \no
\end{align}
Here we used the following identity which is implied by \eqref{3.12}
and \eqref{3.22}
\begin{align}
\T(1/\ol{z},k)^*
\begin{pmatrix}
I_m & 0 \\ 0 & -I_m
\end{pmatrix}
\T(z,k)
=
\begin{pmatrix}
-I_m & 0 \\ 0 & I_m
\end{pmatrix}, \quad k\in\bbZ, \; z\in\Cz.
\end{align}

Finally, taking $k=k_0$ and utilizing \eqref{3.49}, \eqref{3.50},
\eqref{3.125}, \eqref{3.126}, and \eqref{A.7}, one obtains
\eqref{4.3A} and \eqref{4.4A} from \eqref{4.2A}.
\end{proof}

For notational simplicity we abbreviate the Wronskian of $U_+$ and
$U_-$ by 
\begin{equation}
W(z,k_0)=W(U_+(1/\ol{z},k,k_0),U_-(z,k,k_0)). 
\end{equation}
Then, using \eqref{3.128}, \eqref{3.131}, and \eqref{4.4A}, one analytically continues $W(z,k_0)$ to $z=0$ and obtains
\begin{align}
W(z,k_0) = M_+(z,k_0)-M_-(z,k_0), \quad k\in\bbZ, \; z\in\bbC.
\lb{4.1}
\end{align}
Moreover, one verifies a certain symmetry property of the Wronskian $W(z,k_0)$,
\begin{align}
M_+(z,k_0)W(z,k_0)^{-1}M_-(z,k_0) =
M_-(z,k_0)W(z,k_0)^{-1}M_+(z,k_0), \quad k\in\Z, \; z\in\C.
\lb{4.1a}
\end{align}

Next we prove an auxiliary lemma that will play a crucial role in our
computation of the resolvent for the CMV operator $\U$.

\begin{lemma} \lb{l4.1}
Let $k,k_0\in\Z$ and $z\in\Cz$. The the following identities hold,
\begin{align}
& P_+(z,k,k_0)Q_+(1/\ol{z},k,k_0)^* +
Q_+(z,k,k_0)P_+(1/\ol{z},k,k_0)^* = 2(-1)^{k+1}I_m, \lb{4.2}
\\
& R_+(z,k,k_0)S_+(1/\ol{z},k,k_0)^* +
S_+(z,k,k_0)R_+(1/\ol{z},k,k_0)^* = 2(-1)^{k}I_m, \lb{4.3}
\\
& P_+(z,k,k_0)S_+(1/\ol{z},k,k_0)^* +
Q_+(z,k,k_0)R_+(1/\ol{z},k,k_0)^* = 0, \lb{4.4}
\\
& R_+(z,k,k_0)Q_+(1/\ol{z},k,k_0)^* +
S_+(z,k,k_0)P_+(1/\ol{z},k,k_0)^* = 0, \lb{4.5}
\end{align}
and
\begin{align}
& U_+(z,k,k_0)W(z,k_0)^{-1}U_-(1/\ol{z},k,k_0)^* -
U_-(z,k,k_0)W(z,k_0)^{-1}U_+(1/\ol{z},k,k_0)^* = 2(-1)^{k+1}I_m,
\lb{4.6}
\\
& V_+(z,k,k_0)W(z,k_0)^{-1}U_-(1/\ol{z},k,k_0)^* -
V_-(z,k,k_0)W(z,k_0)^{-1}U_+(1/\ol{z},k,k_0)^* = 0. \lb{4.7}
\end{align}
\end{lemma}
\begin{proof}
First, we note that for $k=k_0$ equalities \eqref{4.2}--\eqref{4.5}
follow from \eqref{3.49}. Then one uses an induction argument
to prove \eqref{4.2}--\eqref{4.5} for $k\neq k_0$. This involves a
consideration of a number of cases all of which follow the same
pattern. Therefore, we limit out attention to just one of these cases.
Suppose \eqref{4.2}--\eqref{4.5} hold for some $k\in\Z$ even.
Then utilizing \eqref{3.21} together with \eqref{3.8} and \eqref{3.9},
one computes
\begin{align}
& P_+(z,k+1,k_0)Q_+(1/\ol{z},k+1,k_0)^* +
Q_+(z,k+1,k_0)P_+(1/\ol{z},k+1,k_0)^* \no
\\
&\quad = \wti\rho_{k+1}^{-1}\al_{k+1}
\big[P_+(z,k,k_0)Q_+(1/\ol{z},k,k_0)^* +
Q_+(z,k,k_0)P_+(1/\ol{z},k,k_0)^*\big] \al_{k+1}^*\wti\rho_{k+1}^{-1}
\no
\\
&\qquad + \wti\rho_{k+1}^{-1} \big[R_+(z,k,k_0)S_+(1/\ol{z},k,k_0)^*
+ S_+(z,k,k_0)R_+(1/\ol{z},k,k_0)^*\big] \wti\rho_{k+1}^{-1}
\\
&\qquad + z\wti\rho_{k+1}^{-1} \big[R_+(z,k,k_0)Q_+(1/\ol{z},k,k_0)^*
+ S_+(z,k,k_0)P_+(1/\ol{z},k,k_0)^*\big]
\al_{k_0}^*\wti\rho_{k+1}^{-1} \no
\\
&\qquad + \wti\rho_{k+1}^{-1}\al_{k_0}
\big[P_+(z,k,k_0)S_+(1/\ol{z},k,k_0)^* +
Q_+(z,k,k_0)R_+(1/\ol{z},k,k_0)^*\big] \wti\rho_{k+1}^{-1}z^{-1} \no
\\
&\quad = 2(-1)^{k+1}
\big[\wti\rho_{k+1}^{-1}\al_{k+1}\al_{k+1}^*\wti\rho_{k+1}^{-1} -
\wti\rho_{k+1}^{-2}\big] = 2(-1)^{(k+1)+1}I_m. \no
\end{align}
Similarly, one checks equalities \eqref{4.3}--\eqref{4.5} at the
point $k+1$. Then inverting the matrix $\T(z,k)$ and utilizing
\eqref{3.21} in the form,
\begin{align}
\binom{P_-(z,k-1),k_0}{R_-(z,k-1,k_0)}  = \T(z,k)^{-1}
\binom{P_-(z,k,k_0)}{R_-(z,k,k_0)},
\end{align}
one verifies \eqref{4.2}--\eqref{4.5} at the point $k-1$. Similarly,
one verifies \eqref{4.2}--\eqref{4.5} at the points $k+1$ and $k-1$
under the assumption of $k$ odd.

Next, using \eqref{3.125}, \eqref{3.126}, \eqref{4.1}, \eqref{4.1a},
\eqref{4.2}, and \eqref{4.5}, one verifies \eqref{4.6} and
\eqref{4.7} as follows:
\begin{align}
& U_+(z,k,k_0)W(z,k_0)^{-1}U_-(1/\ol{z},k,k_0)^* -
U_-(z,k,k_0)W(z,k_0)^{-1}U_+(1/\ol{z},k,k_0)^*  \no
\\
&\quad =
Q_+(z,k,k_0)W(z,k_0)^{-1}\big[M_+(z,k_0)-M_-(z,k_0)\big]P_+(1/\ol{z},k,k_0)^*
\no
\\
&\qquad +
P_+(z,k,k_0)\big[M_+(z,k_0)-M_-(z,k_0)\big]W(z,k_0)^{-1}Q_+(1/\ol{z},k,k_0)^*
\no
\\
&\qquad + P_+(z,k,k_0)\big[M_-(z,k_0)W(z,k_0)^{-1}M_+(z,k_0) \no
\\
&\qquad\quad -
M_+(z,k_0)W(z,k_0)^{-1}M_-(z,k_0)\big]P_+(1/\ol{z},k,k_0)^* \no
\\
&\quad = Q_+(z,k,k_0)P_+(1/\ol{z},k,k_0)^* +
P_+(z,k,k_0)Q_+(1/\ol{z},k,k_0)^* = 2(-1)^{k+1}I_m,
\\
& V_+(z,k,k_0)W(z,k_0)^{-1}U_-(1/\ol{z},k,k_0)^* -
V_-(z,k,k_0)W(z,k_0)^{-1}U_+(1/\ol{z},k,k_0)^*  \no
\\
&\quad =
S_+(z,k,k_0)W(z,k_0)^{-1}\big[M_+(z,k_0)-M_-(z,k_0)\big]P_+(1/\ol{z},k,k_0)^*
\no
\\
&\qquad +
R_+(z,k,k_0)\big[M_+(z,k_0)-M_-(z,k_0)\big]W(z,k_0)^{-1}Q_+(1/\ol{z},k,k_0)^*
\no
\\
&\qquad + R_+(z,k,k_0)\big[M_+(z,k_0)W(z,k_0)^{-1}M_-(z,k_0) \no
\\
&\qquad\quad -
M_-(z,k_0)W(z,k_0)^{-1}M_+(z,k_0)\big]P_+(1/\ol{z},k,k_0)^* \no
\\
&\quad = S_+(z,k,k_0)P_+(1/\ol{z},k,k_0)^* +
R_+(z,k,k_0)Q_+(1/\ol{z},k,k_0)^* = 0.
\end{align}
\end{proof}

\begin{lemma} \lb{l4.2}
Let $z\in\bbC\backslash(\dD\cup\{0\})$ and fix $k_0\in\bbZ$. Then the
resolvent $(\U-zI)^{-1}$ of the unitary CMV operator $\U$  on
$\ltm{\bbZ}$ is given in terms of its matrix representation in the
standard basis of $\ltm{\bbZ}$ by
\begin{align}
(\U-zI)^{-1}(k,k') = \frac{1}{2z}
\begin{cases}
U_-(z,k,k_0)W(z,k_0)^{-1}U_+(1/\ol{z},k',k_0)^*, & k < k' \text{ or
} k = k' \text{ odd},
\\
U_+(z,k,k_0)W(z,k_0)^{-1}U_-(1/\ol{z},k',k_0)^*, & k > k' \text{ or 
} k = k' \text{ even},
\end{cases} \lb{4.13}
\\
\quad k,k' \in\Z. \no
\end{align}
Moreover, since $0\in\bbC\backslash\sigma(\U)$, \eqref{4.13}
analytically extends to $z=0$.

In particular, one obtains for any $z\in\bbC\backslash\dD$,
\begin{align}
&(\U-zI)^{-1}(k,k) \no
\\
&\quad = \frac{1}{2z}
\begin{cases}
[I_m+M_-(z,k)] W(z,k)^{-1} [I_m-M_+(z,k)], & k \text{ odd},\\
[I_m-M_+(z,k)] W(z,k)^{-1} [I_m+M_-(z,k)], & k \text{ even},
\end{cases} \lb{4.14}
\\
&(\U-zI)^{-1}(k-1,k-1) \no
\\
&\quad = \frac{1}{2z}
\begin{cases}
\rho_{k}^{-1}[a_{k}^*-b_{k}^*M_+(z,k)] W(z,k)^{-1}
[a_{k}+M_-(z,k)b_{k}]\rho_{k}^{-1}, & k \text{ odd},\\
\wti\rho_{k}^{-1}[a_{k}+b_{k}M_-(z,k)] W(z,k)^{-1}
[a_{k}^*-M_+(z,k)b_{k}^*]\wti\rho_{k}^{-1}, & k \text{
even},
\end{cases} \lb{4.15}
\\
&(\U-zI)^{-1}(k-1,k) \no
\\
&\quad = \frac{-1}{2z}
\begin{cases}
\rho_{k}^{-1} [a_{k}^*-b_{k}^*M_-(z,k)] W(z,k)^{-1}
[I_m-M_+(z,k)], & k \text{ odd}, \\
\wti\rho_{k}^{-1} [a_{k}+b_{k}M_-(z,k)] W(z,k)^{-1}
[I_m+M_+(z,k)], & k \text{ even},
\end{cases} \lb{4.16}
\\
&(\U-zI)^{-1}(k,k-1) \no
\\
&\quad = \frac{-1}{2z}
\begin{cases}
[I_m+M_+(z,k)] W(z,k)^{-1} [a_{k}+M_-(z,k)b_{k}]
\rho_{k}^{-1}, & k \text{ odd},
\\
[I_m-M_+(z,k)] W(z,k)^{-1} [a_{k}^*-M_-(z,k)b_{k}^*]
\wti\rho_{k}^{-1}, & k \text{ even}.
\end{cases} \lb{4.17}
\end{align}
\end{lemma}
\begin{proof}
Let
\begin{align}
G(z,k,k',k_0) =
\begin{cases}
U_-(z,k,k_0)W(z,k_0)^{-1}U_+(1/\ol{z},k',k_0)^*, & k < k' \text{ or
} k = k' \text{ odd},
\\
U_+(z,k,k_0)W(z,k_0)^{-1}U_-(1/\ol{z},k',k_0)^*, & k > k' \text{ or 
} k = k' \text{ even},
\end{cases}
\\
\quad k,k' \in\Z. \no
\end{align}
Then \eqref{4.13} is equivalent to
\begin{align}
&(\U-zI)G(z,\cdot,k',k_0) = 2z\De_{k'},\quad k',k_0\in\Z. \lb{4.19}
\end{align}

First, assume $k'$ to be odd. Then,
\begin{align}
((\U-zI)G(z,\cdot,k',k_0))(\ell) =
((\V\W-zI)G&(z,\cdot,k',k_0))(\ell) = 0, \quad \ell \in \Z\backslash
\{k',k'+1\},   \lb{4.20}
\end{align}
and by \eqref{4.6}, \eqref{4.7},
\begin{align}
&\binom{((\U-zI)G(z,\cdot,k',k_0))(k')}
{((\U-zI)G(z,\cdot,k',k_0))(k'+1)} =
\binom{((\V\W-zI)G(z,\cdot,k',k_0))(k')}
{((\V\W-zI)G(z,\cdot,k',k_0))(k'+1)} \no
\\
&\quad = \Te_{k'+1}
\binom{zV_-(z,k',k_0)W(z,k_0)^{-1}U_+(1/\ol{z},k',k_0)^*}
{zV_+(z,k'+1,k_0)W(z,k_0)^{-1}U_-(1/\ol{z},k',k_0)^*} -
z\binom{G(z,k',k',k_0)}{G(z,k'+1,k',k_0)} \no
\\
&\quad = z\Te_{k'+1}
\binom{V_+(z,k',k_0)W(z,k_0)^{-1}U_-(1/\ol{z},k',k_0)^*}
{V_+(z,k'+1,k_0)W(z,k_0)^{-1}U_-(1/\ol{z},k',k_0)^*} -
z\binom{G(z,k',k',k_0)}{G(z,k'+1,k',k_0)} \lb{4.21}
\\
&\quad = z\binom{U_+(z,k',k_0)W(z,k_0)^{-1}U_-(1/\ol{z},k',k_0)^*}
{U_+(z,k'+1,k_0)W(z,k_0)^{-1}U_-(1/\ol{z},k',k_0)^*} -
z\binom{G(z,k',k',k_0)}{G(z,k'+1,k',k_0)} \no
\\
&\quad = z\binom{2(-1)^{k'+1}I_m}{0} = \binom{2zI_m}{0}. \no
\end{align}
Thus, for $k'$ odd, \eqref{4.19} is a consequence of \eqref{4.20} and
\eqref{4.21}.

Next, assume $k'$ to be even. Then,
\begin{align}
((\U-zI)G(z,\cdot,k',k_0))(\ell) = ((\V\W-zI)G(z,\cdot,k',k_0))(\ell)
= 0, \quad \ell \in \Z\backslash \{k'-1,k'\},  \lb{4.22}
\end{align}
and by \eqref{4.6}, \eqref{4.7},
\begin{align}
&\binom{((\U-zI)G(z,\cdot,k',k_0))(k'-1)}
{((\U-zI)G(z,\cdot,k',k_0))(k')} =
\binom{((\V\W-zI)G(z,\cdot,k',k_0))(k'-1)}
{((\V\W-zI)G(z,\cdot,k',k_0))(k')} \no
\\
& \quad = \Te_{k'}
\binom{zV_-(z,k'-1,k_0)W(z,k_0)^{-1}U_+(1/\ol{z},k',k_0)^*}
{zV_+(z,k',k_0)W(z,k_0)^{-1}U_-(1/\ol{z},k',k_0)^*} -
z\binom{G(z,k'-1,k',k_0)}{G(z,k',k',k_0)} \no
\\
& \quad = z\Te_{k'}
\binom{V_-(z,k'-1,k_0)W(z,k_0)^{-1}U_+(1/\ol{z},k',k_0)^*}
{V_-(z,k',k_0)W(z,k_0)^{-1}U_+(1/\ol{z},k',k_0)^*} -
z\binom{G(z,k'-1,k',k_0)}{G(z,k',k',k_0)} \lb{4.23}
\\
& \quad =
z\binom{U_-(z,k'-1,k_0)W(z,k_0)^{-1}U_+(1/\ol{z},k',k_0)^*}
{U_-(z,k',k_0)W(z,k_0)^{-1}U_+(1/\ol{z},k',k_0)^*} -
z\binom{G(z,k'-1,k',k_0)}{G(z,k',k',k_0)} \no
\\
&\quad = z\binom{0}{2(-1)^{k'}I_m} = \binom{0}{2zI_m}. \no
\end{align}
Thus, for $k'$ even, \eqref{4.19} follows from \eqref{4.22} and
\eqref{4.23}, and hence one obtains \eqref{4.13}.

Finally, using \eqref{Table} and \eqref{3.125} one verifies the identities
\begin{align}
U_\pm(z,k,k) &=
\begin{cases}
z[I_m+M_\pm(z,k)], & k \text{ odd},\\
-I_m+M_\pm(z,k), & k \text{ even},
\end{cases} \lb{4.24}
\\
U_\pm(z,k-1,k) &=
\begin{cases}
-z\rho_{k}^{-1}[a_{k}^*-b_{k}^*M_\pm(z,k)], & k \text{ odd},\\
\wti\rho_{k}^{-1}[a_{k}+b_{k}M_\pm(z,k)], & k \text{ even}.
\end{cases} \lb{4.25}
\end{align}
Inserting \eqref{4.24} and \eqref{4.25} into \eqref{4.13} and
utilizing the fact that (anti-)Caratheodory matrices satisfy
$M_\pm(1/\ol{z},k)^* = -M_\pm(z,k)$, $k\in\Z$, $z\in\C$, one obtains
\eqref{4.14}--\eqref{4.17}.
\end{proof}

Next, we briefly turn to Weyl--Titchmarsh theory for CMV operators
with matrix-valued Verblunsky coefficients on $\bbZ$. We denote by
$d\Omega(\cdot,k)$, $k\in\Z$, the $2m \times 2m$ matrix-valued
measure,
\begin{align}
d\Omega(\ze,k) &= d
\begin{pmatrix}
\Omega_{0,0}(\ze,k) & \Omega_{0,1}(\ze,k)
\\
\Omega_{1,0}(\ze,k) & \Omega_{1,1}(\ze,k)
\end{pmatrix} \no
\\ &= d
\begin{pmatrix}
\De_{k-1}^*E_{\U}(\ze)\De_{k-1} & \De_{k-1}^*E_{\U}(\ze)\De_{k}
\\
\De_{k}^*E_{\U}(\ze)\De_{k-1} & \De_{k}^*E_{\U}(\ze)\De_{k}
\end{pmatrix}, \quad \ze \in\dD, \lb{4.26}
\end{align}
where $E_{\U}(\cdot)$ denotes the family of spectral projections
of the unitary CMV operator $\U$ on $\ltm{\bbZ}$,
\begin{equation}
\U=\oint_\dD dE_{\U}(\ze)\,\ze.
\end{equation}

We also introduce the $2m \times 2m$ matrix-valued function
$\cM(\cdot,k)$, $k\in\Z$, by
\begin{align}
&\cM(z,k) = \begin{pmatrix} M_{0,0}(z,k) & M_{0,1}(z,k) \\
M_{1,0}(z,k) & M_{1,1}(z,k) \end{pmatrix} \no
\\ &\quad =
\begin{pmatrix}
\De_{k-1}^*(\U+zI)(\U-zI)^{-1}\De_{k-1}
&\De_{k-1}^*(\U+zI)(\U-zI)^{-1}\De_{k}
\\
\De_{k}^*(\U+zI)(\U-zI)^{-1}\De_{k-1} &
\De_{k}^*(\U+zI)(\U-zI)^{-1}\De_{k}
\end{pmatrix} \lb{4.28}
\\ &\quad =
\oint_\dD d\Omega(\ze,k)\, \frac{\ze+z}{\ze-z}, \quad
z\in\bbC\backslash\dD. \no
\end{align}
We note that
\begin{align}
M_{0,0}(\cdot,k+1) = M_{1,1}(\cdot,k), \quad k\in\bbZ \lb{4.29}
\end{align}
and
\begin{align}
\begin{split}
M_{1,1}(z,k) &= \De_{k}^*(\U+zI)(\U-zI)^{-1}\De_{k} \lb{4.30}
\\ & = \oint_\dD d\Omega_{1,1}(\ze,k) \,
\frac{\ze+z}{\ze-z}, \quad z\in\bbC\backslash\dD,\; k\in\Z,
\end{split}
\end{align}
where
\begin{equation}
d\Omega_{1,1}(\ze,k)=d\De_{k}^*E_{\U}(\ze)\De_{k}, \quad \ze\in\dD.
\lb{4.32}
\end{equation}
Thus, $M_{0,0}|_{\D}$ and $M_{1,1}|_{\D}$ are $m\times m$
Caratheodory matrices. Moreover, by \eqref{4.30} one
infers that
\begin{equation}
M_{1,1}(0,k)=I_m, \quad k\in\Z. \lb{4.33}
\end{equation}

\begin{lemma}
Let $z\in\bbC\backslash\dD$. Then the functions
$M_{\ell,\ell'}(\cdot,k)$, $\ell,\ell'=0,1$, and $M_\pm(\cdot,k)$,
$k\in\bbZ$, satisfy the following relations
\begin{align}
&M_{0,0}(z,k) = I_m +
\begin{cases}
\rho_{k}^{-1}[a_{k}^*-b_{k}^*M_+(z,k)] W(z,k)^{-1}
[a_{k}+M_-(z,k)b_{k}]\rho_{k}^{-1}, & k \text{ odd},\\
\wti\rho_{k}^{-1}[a_{k}+b_{k}M_-(z,k)] W(z,k)^{-1}
[a_{k}^*-M_+(z,k)b_{k}^*] \wti\rho_{k}^{-1}, & k \text{ even},
\end{cases} \lb{4.34}
\\
&M_{1,1}(z,k) = I_m +
\begin{cases}
[I_m+M_-(z,k)] W(z,k)^{-1} [I_m-M_+(z,k)], & k \text{ odd},\\
[I_m-M_+(z,k)] W(z,k)^{-1} [I_m+M_-(z,k)], & k \text{ even},
\end{cases} \lb{4.35}
\\
&M_{0,1}(z,k) = -
\begin{cases}
\rho_{k}^{-1} [a_{k}^*-b_{k}^*M_-(z,k)] W(z,k)^{-1}
[I_m-M_+(z,k)], & k \text{ odd}, \\
\wti\rho_{k}^{-1} [a_{k}+b_{k}M_-(z,k)] W(z,k)^{-1} [I_m+M_+(z,k)],
& k \text{ even},
\end{cases} \lb{4.36}
\\
&M_{1,0}(z,k) = -
\begin{cases}
[I_m+M_+(z,k)] W(z,k)^{-1} [a_{k}+M_-(z,k)b_{k}] \rho_{k}^{-1}, & k
\text{ odd},
\\
[I_m-M_+(z,k)] W(z,k)^{-1} [a_{k}^*-M_-(z,k)b_{k}^*]
\wti\rho_{k}^{-1}, & k \text{ even},
\end{cases} \lb{4.37}
\end{align}
where $a_k=I_m+\al_k$ and $b_k=I_m-\al_k$, $k\in\Z$.
\end{lemma}
\begin{proof}
The result is a consequence of Lemma \ref{l4.2} since by \eqref{4.28} one
has
\begin{align}
M_{\ell,\ell'}(z,k) &= \De_{k-1+\ell}^*(I +
2z(\U-zI)^{-1})\De_{k-1+\ell'} \no
\\
&= I_m\de_{\ell,\ell'} + (\U-zI)^{-1}(k-1+\ell,k-1+\ell').
\end{align}
\end{proof}

Finally, introducing the $m\times m$ matrix-valued functions
$\Phi_{1,1}(\cdot,k)$, $k\in\bbZ$, by
\begin{align} \lb{4.39}
\Phi_{1,1}(z,k) &= [M_{1,1}(z,k)-I_m][M_{1,1}(z,k)+I_m]^{-1} \no
\\
&= I_m - 2[M_{1,1}(z,k)+I_m]^{-1}, \quad z\in\C\backslash\dD,
\end{align}
then,
\begin{align}
M_{1,1}(z,k) &= [I_m+\Phi_{1,1}(z,k)][I_m-\Phi_{1,1}(z,k)]^{-1} \no
\\
&=2[I_m-\Phi_{1,1}(z,k)]^{-1} - I_m, \quad z\in\C\backslash\dD.
\lb{4.40}
\end{align}

\begin{lemma}
The $\Cm$-valued function $\Phi_{1,1}|_{\D}$ is a Schur matrix and $\Phi_{1,1}$
is related to $\Phi_\pm$ by
\begin{equation}
\Phi_{1,1}(z,k) =
\begin{cases}
\Phi_-(z,k)^{-1}\Phi_+(z,k), & k\text{ odd},
\\
\Phi_+(z,k)\Phi_-(z,k)^{-1}, & k\text{ even},
\end{cases}
\quad z\in\C\backslash\dD,\; k\in\bbZ. \lb{4.41}
\end{equation}
\end{lemma}
\begin{proof}
Suppose $k$ is odd. Then\eqref{4.1}, \eqref{4.1a},
and \eqref{4.35} imply that
\begin{align}
M_{1,1}(z,k)+I_m &= [I_m+M_-(z,k)] W(z,k)^{-1} [I_m-M_+(z,k)]  \no
\\
&\quad + [M_+(z,k)-M_-(z,k)] W(z,k)^{-1} +
W(z,k)^{-1} [M_+(z,k)-M_-(z,k)] \no
\\
&= [I_m+M_+(z,k)] W(z,k)^{-1} [I_m-M_-(z,k)]. \lb{4.42}
\end{align}
Using \eqref{3.133}, \eqref{4.1}, \eqref{4.39}, and \eqref{4.42},
one computes
\begin{align}
\Phi_{1,1}(z,k) &= I_m-2[I_m-M_-(z,k)]^{-1}W(z,k)[I_m+M_+(z,k)]^{-1}
\no
\\
&=[I_m-M_-(z,k)]^{-1} \big[[I_m-M_-(z,k)][I_m+M_+(z,k)]-2W(z,k)\big]
[I_m+M_+(z,k)]^{-1} \no
\\
&=[I_m-M_-(z,k)]^{-1}[I_m+M_-(z,k)]
[I_m-M_+(z,k)][I_m+M_+(z,k)]^{-1}
\\
&=\Phi_-(z,k)^{-1}\Phi_+(z,k), \quad z\in\C\backslash\dD,\;
k\in\bbZ. \no
\end{align}
The result for $k$ even is proved similarly.
\end{proof}

Next we introduce a sequence of $\C^{m\times 2m}$-valued 
Laurent polynomials $\{P(z,k,k_0)\}_{k\in\Z}$ by
\begin{align} \lb{4.44}
P(z,k,k_0) &= \big(P_0(z,k,k_0),P_1(z,k,k_0)\big) \no
\\
&=
\begin{cases}
\big(P_+(z,k,k_0),Q_+(z,k,k_0)\big)
\begin{pmatrix}
\frac{1}{2z}\rho_{k_0} & \frac{1}{2z}a_{k_0}^*
\\
-\frac{1}{2z}\rho_{k_0} & \frac{1}{2z}b_{k_0}^*
\end{pmatrix}, & k_0 \text{ odd},
\\[4mm]
\big(P_+(z,k,k_0),Q_+(z,k,k_0)\big)
\begin{pmatrix}
\frac{1}{2}\wti\rho_{k_0} & \frac{1}{2}a_{k_0}
\\ \frac{1}{2}\wti\rho_{k_0} & -\frac{1}{2}b_{k_0}
\end{pmatrix}, & k_0 \text{ even}.
\end{cases}
\end{align}
Then it is easy to see that $P_j(z,\cdot,k_0)$, $j=0,1$ are linear
combinations of $P_+(z,\cdot,k_0)$ and $Q_+(z,\cdot,k_0)$, and hence
satisfy $\U P_j(z,\cdot,k_0) = zP_j(z,\cdot,k_0)$, $j=0,1$.
Moreover, \eqref{Table} and \eqref{4.44} imply that
\begin{align}
\begin{split}
P(z,k_0-1,k_0) &= (P_0(z,k_0-1,k_0),P_1(z,k_0-1,k_0)) = (I_m,0),
\\
P(z,k_0,k_0) &= (P_0(z,k_0,k_0),P_1(z,k_0,k_0)) = (0,I_m),
\end{split} \lb{4.45}
\end{align}
and hence any solution $U(z,\cdot)$ of $\U U(z,\cdot) = zU(z,\cdot)$
can be expressed as
\begin{align} \lb{4.46}
U(z,\cdot) = P_0(z,\cdot,k_0)U(z,k_0-1)+P_1(z,\cdot,k_0)U(z,k_0).
\end{align}

Our next goal is to show that the Laurent polynomials
$\{P(z,k,k_0)^*\}_{k\in\Z}$ form complete orthonormal system in
$\Ltm{}$. To do that we first prove an auxiliary result analogous to
Lemma \ref{l3.5}.

\begin{lemma} \lb{l4.5}
Suppose $\{F(\cdot,k)\}_{k\in\Z}$ is a sequence of $\Cm$-valued
functions of bounded variation with $F(1,k)=0$ for all $k\in\Z$ 
that satisfies
\begin{align}
(\U F(\zeta,\cdot))(k) = \int_{A_{\zeta}} dF(\zeta',k) \, \zeta',
\quad \zeta\in\dD,\; k\in\Z,
\end{align}
where $\U$ are understood in the sense of difference expressions
rather than difference operators on $\ltm{\Z}$. Then, $F(\cdot,k)$ also
satisfies
\begin{align}
F(\zeta,k) = \int_{A_{\zeta}}P_0(\zeta',k,k_0)\,dF(\zeta',k_0-1) +
\int_{A_{\zeta}}P_1(\zeta',k,k_0)\,dF(\zeta',k_0), \quad
\zeta\in\dD,\; k, k_0\in\Z. \lb{4.48}
\end{align}
\end{lemma}
\begin{proof}
Let $\{G(\cdot,k,k_0)\}_{k\in\Z}$ denote the sequence of
$\Cm$-valued functions,
\begin{align}
G(\zeta,k,k_0) = \int_{A_{\zeta}}P_0(\zeta',k,k_0)\,dF(\zeta',k_0-1)
+ \int_{A_{\zeta}}P_1(\zeta',k,k_0)\,dF(\zeta',k_0), \quad
\zeta\in\dD,\; k, k_0\in\Z. \lb{4.49}
\end{align}
Then it suffices to prove that $F(\zeta,k) = G(\zeta,k,k_0)$,
$\zeta\in\dD$, $k,k_0\in\Z$.

First, we note that \eqref{4.45} and \eqref{4.49} imply that
\begin{align}
\begin{split}
G(\zeta,k_0-1,k_0) &=
\int_{A_{\zeta}}dF(\zeta',k_0-1)=F(\zeta,k_0-1),
\\
G(\zeta,k_0,k_0) &= \int_{A_{\zeta}}dF(\zeta',k_0)=F(\zeta,k_0),
\quad \zeta\in\dD,\; k_0\in\Z, 
\end{split}
\end{align}
and
\begin{align}
(\U G(\zeta,\cdot,k_0))(k) &= \int_{A_{\zeta}} (\U
P_0(\zeta',\cdot,k_0))(k)\,dF(\zeta',k_0-1) + \int_{A_{\zeta}} (\U
P_1(\zeta',\cdot,k_0))(k)\,dF(\zeta',k_0) \no
\\
&= \int_{A_{\zeta}} dG(\zeta',k,k_0) \, \zeta', \quad \zeta\in\dD,\;
k,k_0\in\Z.
\end{align}

Next, defining $K(\zeta,k,k_0) = F(\zeta,k)-G(\zeta,k,k_0)$,
$\zeta\in\dD$, $k,k_0\in\Z$, one obtains
\begin{align*}
\begin{split}
&K(\zeta,k_0-1,k_0) = K(\zeta,k_0,k_0) = 0,
\\
&(\U K(\zeta,\cdot,k_0))(k) = \int_{A_{\zeta}} dK(\zeta',k,k_0) \,
\zeta' , \quad \zeta\in\dD,\; k,k_0\in\Z,
\end{split}
\end{align*}
or equivalently,
\begin{align}
\begin{split}
&K(\zeta,k_0-1,k_0) = K(\zeta,k_0,k_0) = 0, \lb{4.52}
\\
&(\U K(\zeta,\cdot,k_0))(k) = (\L \, K(\cdot,k,k_0))(\zeta), \quad
\zeta\in\dD,\; k,k_0\in\Z,
\end{split}
\end{align}
where $\L$ denotes the boundedly invertible operator on $\Cm$-valued
functions $K$ of bounded variation defined by 
\begin{align}
(\L \, K)(\zeta) = \int_{A_{\zeta}} dK(\zeta')\,  \zeta', \quad
(\L^{-1} K)(\zeta) = \int_{A_{\zeta}} dK(\zeta') \,  {\zeta'}^{-1}.
\end{align}

Finally, since, $\L$ commutes with all constant $m\times m$
matrices, one can repeat the proof of Lemma \ref{l3.2} with $z$
replaced by $\L$ and using \eqref{4.46} obtain that
\eqref{4.52} has the unique solution $K(\zeta,k,k_0)=0$,
$\zeta\in\dD$, $k,k_0\in\Z$, and hence, $F(\zeta,k) =
G(\zeta,k,k_0)$, $\zeta\in\dD$, $k,k_0\in\Z$.
\end{proof}

\begin{lemma} \lb{l4.6}
Let $k_0\in\Z$. Then the set of $\C^{2m\times m}$-valued Laurent polynomials
$\{P(\cdot,k,k_0)^*\}_{k\in\Z}$ forms a complete orthonormal system
on $\dD$ with respect to $\C^{2m\times2m}$-valued measure
$d\Om(\cdot,k_0)$. Explicitly, $P(\cdot,k,k_0)$, $k\in\Z$, satisfy,
\begin{align}
\oint_{\dD} P(\zeta,k,k_0)\,d\Om(\zeta,k_0)\,P(\zeta,k',k_0)^* &=
\de_{k,k'}I_m, \quad k,k'\in\Z \lb{4.54}
\end{align}
and the collection of $\C^{2m}$-valued Laurent polynomials
\begin{align}
\left\{
\begin{pmatrix}
(P(\cdot,k,k_0))_{1,j}\\\vdots\\(P(\cdot,k,k_0))_{2m,j}
\end{pmatrix}\right\}_{ j=1,\dots,2m, \, k\in\Z}
\end{align}
form complete systems in $\Ltm{}$.
\end{lemma}
\begin{proof}
Fix an integer $k'\in\Z$ and let $\{F(\cdot,k,k')\}_{k\in\Z}$ denote
the $\Cm$-valued sequences of functions of bounded
variation defined by 
\begin{align}
F(\zeta,k,k') = \De_{k}^* E_{\U}(\zeta)\De_{k'}, \quad \zeta\in\dD,
\; k\in\Z. \lb{4.55}
\end{align}
Then,
\begin{align}
(\U F(\zeta,\cdot,k'))(k) &= (\U E_{\U}(\zeta)\De_{k'})(k) =
\left(\int_{A_{\zeta}} dE_{\U}(\zeta')\, \zeta' \De_{k'}\right)(k)
\\ &=
\int_{A_{\zeta}} d\big(\De_k^* E_{\U}(\zeta')\De_{k'}\big) \, \zeta'
= \int_{A_{\zeta}} dF(\zeta',k,k') \, \zeta', \quad \zeta\in\dD,\;
k\in\Z, \no
\end{align}
and hence \eqref{4.48} in Lemma \ref{l4.5} implies that
\begin{align} \lb{4.57}
dF(\zeta,k,k') = P_0(\zeta,k,k_0)\,dF(\zeta,k_0-1,k') +
P_1(\zeta,k,k_0)\,dF(\zeta,k_0,k'), \quad \zeta\in\dD, \; k\in\Z,
\end{align}
or equivalently,
\begin{align} \lb{4.58}
dF(\zeta,k',k) = dF(\zeta,k',k)^* =
dF(\zeta,k',k_0-1)\,P_0(\zeta,k,k_0)^* +
dF(\zeta,k',k_0)\,P_1(\zeta,k,k_0)^*,& \no
\\
\quad \zeta\in\dD, \; k\in\Z.&
\end{align}
In particular, taking $k'=k_0-1$ and $k'=k_0$, one obtains from
\eqref{4.58},  
\begin{align}
dF(\zeta,k_0-1,k) &= dF(\zeta,k_0-1,k_0-1)\,P_0(\zeta,k,k_0)^* +
dF(\zeta,k_0-1,k_0)\,P_1(\zeta,k,k_0)^*, \lb{4.59}
\\
dF(\zeta,k_0,k) &= dF(\zeta,k_0,k_0-1)\,P_0(\zeta,k,k_0)^* +
dF(\zeta,k_0,k_0)\,P_1(\zeta,k,k_0)^*,\quad \zeta\in\dD, \; k\in\Z.
\no
\end{align}
Next, setting $k=k'$ in \eqref{4.59} and plugging it into
\eqref{4.57}, one obtains
\begin{align}
dF(\zeta,k,k') = \sum_{\ell,\ell'=0}^1 P_\ell(\zeta,k,k_0)
\,dF(\zeta,k_0-1+\ell,k_0-1+\ell')\, P_{\ell'}(\zeta,k',k_0)^*,
\quad \zeta\in\dD, \; k,k'\in\Z. \lb{4.60}
\end{align}
Integrating \eqref{4.60} over the unit circle $\dD$ and observing
that by \eqref{4.26} and \eqref{4.55}
$dF(\zeta,k_0-1+\ell,k_0-1+\ell') = d\Om_{\ell,\ell'}(z,k_0)$,
$\ell,\ell'=0,1$, one obtains
\begin{align}
\de_{k,k'}I_m = \oint_\dD\sum_{\ell,\ell'=0}^1 P_\ell(\zeta,k,k_0)
\,d\Om_{\ell,\ell'}(\zeta,k_0)\, P_{\ell'}(\zeta,k',k_0)^*, \quad
\zeta\in\dD, \; k,k'\in\Z,
\end{align}
which is equivalent to \eqref{4.54}.

To prove completeness of $\{P(\cdot,k,k_0)^*\}_{k\in\Z}$ we first note
the fact,
\begin{align}
\spn\{P(z,k,k_0)^*\}_{k\in\bbZ} &= \spn \left\{
\begin{pmatrix} z^k I_m \\ z^{k-1} I_m\end{pmatrix},
\begin{pmatrix} z^{k-1} I_m \\ z^k I_m\end{pmatrix},
\begin{pmatrix} I_m \\ 0\end{pmatrix},
\begin{pmatrix} 0 \\ I_m\end{pmatrix} \right\}_{k\in\Z} \no \\
&= \spn\left\{\begin{pmatrix} z^k I_m\\ 0 \end{pmatrix},
\begin{pmatrix} 0 \\ z^k I_m\end{pmatrix} \right\}_{k\in\bbZ}.
\end{align}
This is a consequence of investigating the leading-order coefficients
of $P_+(z,k,k_0)$ and $Q_+(z,k,k_0)$ (cf.\ Remark \ref{r2.4}) and
\eqref{4.44}). Thus, it suffices to prove that
$\Big\{\Big(\begin{smallmatrix}\ze^k I_m \\ 0\end{smallmatrix}\Big),
\Big(\begin{smallmatrix}0 \\ \ze^k I_m
\end{smallmatrix}\Big)\Big\}_{k\in\Z}$ is a complete system
with respect to $d\Om(\cdot,k_0)$.

Let $F=\Big(\begin{smallmatrix} F_0 \\
F_1 \end{smallmatrix}\Big)\in\Ltm{}$ and suppose $F$ is orthogonal
to all columns of $\Big(\begin{smallmatrix} \ze^k I_m\\ 0
\end{smallmatrix}\Big)$ and $\Big(\begin{smallmatrix} 0 \\
\ze^k I_m\end{smallmatrix}\Big)$ for all $k\in\Z$, that is,
\begin{equation}
\oint_{\dD} \begin{pmatrix} \zeta^k I_m \\ 0\end{pmatrix}^*\,
d\Omega (\zeta,k_0)\, F(\zeta)  = \oint_{\dD} \zeta^{-k} \,
[d\Omega_{0,0}(\zeta,k_0)\,F_0(\zeta)
+d\Omega_{0,1}(\zeta,k_0)\,F_1(\zeta)] =
\begin{pmatrix}0\\\vdots\\0\end{pmatrix}\in\C^{2m} \lb{4.63}
\end{equation}
and
\begin{equation}
\oint_{\dD} \begin{pmatrix} 0 \\ \zeta^k I_m \end{pmatrix}^*\,
d\Omega (\zeta,k_0)\, F(\zeta)  =\oint_{\dD} \zeta^{-k} \,
[d\Omega_{1,0}(\zeta,k_0)\,F_0(\zeta) +
d\Omega_{1,1}(\zeta,k_0)\,F_1(\zeta)] =
\begin{pmatrix}0\\\vdots\\0\end{pmatrix}\in\C^{2m} \lb{4.64}
\end{equation}
for all $k\in\Z$. Note that for a scalar complex-valued measure
$d\om$ equalities $\oint d\om(\ze)\,\ze^n=0$, $n\in\Z$, imply $\oint
d\Re(\om(\ze))\,\ze^n=\oint d\Im(\om(\ze))\,\ze^n = 0$, and hence 
\cite[p.\ 24]{Du83}) implies that $d\om=0$. Applying this
argument to $d(\Omega_{0,0}F_0 + \Omega_{0,1}F_1)_\ell$ and
$d(\Omega_{1,0}F_0 + \Omega_{1,1}F_1)_\ell$, $\ell=1,\dots,2m$, one
obtains
\begin{align}
d\Omega_{0,0}F_0 + d\Omega_{0,1}F_1 &=
\begin{pmatrix}0\\\vdots\\0\end{pmatrix}\in\C^{2m}, \lb{4.65}
\\
d\Omega_{1,0}F_0 + d\Omega_{1,1}F_1 &=
\begin{pmatrix}0\\\vdots\\0\end{pmatrix}\in\C^{2m}. \lb{4.66}
\end{align}
Multiplying \eqref{4.65} by $F_0^*$ on the left and \eqref{4.66} by
$F_1^*$ on the left and adding the results then yields
\begin{align}
\|F\|_{\Ltm{}}^2=\oint_{\dD} F(\zeta)^* \,d\Omega(\zeta,k_0)\,
F(\zeta) = 0.
\end{align}
\end{proof}

\begin{corollary}
The full-lattice CMV operator $\U$ is unitarily equivalent to the
operator of multiplication by $\zeta$ on $\Ltm{}$ for any
$k_0\in\Z$. In particular,
\begin{align}
& \si(\U) = \supp \, (d\Om(\cdot,k_0)), \quad k_0\in\Z.
\end{align}
\end{corollary}
\begin{proof}
Consider the linear map
$\dot\cU\colon\ltzm{\bbZ}\to\Ltm{}$ from the space of compactly
supported sequences $\ltzm{\bbZ}$ to the set of $\C^{2m}$-valued
Laurent polynomials defined by
\begin{equation}
(\dot \cU F)(z) = \sum_{k=-\infty}^{\infty} P(1/\ol{z},k,k_0)^*
F(k), \quad F\in \ltzm{\bbZ}.
\end{equation}
Using \eqref{4.54} one shows that $\hatt F(\zeta) = (\dot \cU
F)(\zeta)$, $F\in \ltzm{\bbZ}$ has the property
\begin{align}
\|\hatt F\|^2_{\Ltm{}} &= \oint_{\dD} \hatt
F(\zeta)^*d\Om(\zeta,k_0)\,\hatt F(\zeta)
\\ &=
\oint_{\dD} \sum_{k=-\infty}^\infty F(k)^* P(\zeta,k,k_0)
\,d\Om_\pm(\zeta,k_0)\!\! \sum_{k'=-\infty}^{\infty}
P_\pm(\zeta,k',k_0)^*F(k') \no
\\ &=
\sum_{k,k'=-\infty}^{\infty} F(k)^* \left(\oint_{\dD} P(\zeta,k,k_0)
\,d\Om(\zeta,k_0)\, P(\zeta,k',k_0)^*\right)F(k') \no
\\ &=
\sum_{k=-\infty}^{\infty} F(k)^*F(k) = \|F\|^2_{\ltm{\Z}}. \lb{4.70}
\end{align}
Since $\ltzm{\bbZ}$ is dense in $\ltm{\bbZ}$, $\dot \cU$ extends by
continuity to a bounded linear operator $\cU\colon \ltm{\bbZ} \to
\Ltm{}$, and the identity 
\begin{align} \lb{4.71}
(\cU(\U F))(\zeta) &= \sum_{k=-\infty}^{\infty} P(\zeta,k,k_0)^* (\U
F)(k) = \sum_{k=-\infty}^{\infty} (\U^* P(\zeta,\cdot,k_0))(k)^*
F(k)
\\ &=
\sum_{k=-\infty}^{\infty} (\zeta^{-1} P(\zeta,k,k_0))^* F(k) =
\zeta(\cU F)(\zeta), \quad F\in\ltm{\bbZ},   \no
\end{align}
holds. The range of the operator $\cU$ is all of $\Ltm{}$ since the
$\C^{2m\times m}$-valued Laurent polynomials $\{P(\cdot,k,k_0)^*\}_{k\in\Z}$
are complete with respect to $d\Om(\cdot,k_0)$. Hence the inverse
operator $\cU^{-1}$ exists on $\Ltm{}$ and is given by
\begin{equation}
(\cU^{-1}\hatt F)(k) = \oint_{\dD}
P(\zeta,k,k_0)\,d\Om(\zeta,k_0)\,\hatt F(\zeta), \quad \hatt F\in
\Ltm{},
\end{equation}
which together with \eqref{4.70} implies that $\cU$ is unitary. In
addition, \eqref{4.71} shows that the full-lattice unitary
operator $\U$ on $\ltm{\bbZ}$ is unitarily equivalent to the
operators of multiplication by $\zeta$ on $\Ltm{}$,
\begin{align}
(\cU \U \cU^{-1} \hatt F)(\zeta) = \zeta\hatt F(\zeta), \quad \hatt
F\in \Ltm{}.
\end{align}
\end{proof}

\section{Borg--Marchenko-type Uniqueness Results for CMV Operators with
Matrix-valued Verblunsky Coefficients} \lb{s5}

In this section we prove (local and global) Borg--Marchenko-type uniqueness results for
CMV operators with matrix-valued Verblunsky coefficients on half-lattices and on
the full lattice $\bbZ$. We freely use the notation established in Sections \ref{s3},
\ref{s4}, and Appendix \ref{sA}.

We start with uniqueness results on half-lattices.

\begin{theorem} \lb{t5.1}
Assume Hypothesis \ref{h3.1} and let $k_0\in\Z$, $N\in\N$. Then for
the right half-lattice problem the following sets of data $(i)$--$(v)$ are
equivalent:
\begin{align}
(i) &\quad \big\{\al_{k_0+k}\big\}_{k=1}^N.
\\
(ii) &\quad \bigg\{\oint_{\dD}d\Om_+(\zeta,k_0)\, \zeta^k \bigg\}_{k=1}^N.
\\
(iii) &\quad \big\{m_{+,k}(k_0)\big\}_{k=1}^N, \;\text{ where
$m_{+,k}(k_0)$, $k\geq 0$, are the Taylor coefficients of
$m_+(z,k_0)$} \no\\ &\hspace*{5mm} \text{at $z=0$, that is, }\;
m_+(z,k_0) = \sum\nolimits_{k=0}^\infty m_{+,k}(k_0)z^k, \; z\in\D.
\\
(iv) &\quad \big\{M_{+,k}(k_0)\big\}_{k=1}^N, \;\text{ where
$M_{+,k}(k_0)$, $k\geq 0$, are the Taylor coefficients of
$M_+(z,k_0)$} \no\\ &\hspace*{5mm} \text{at $z=0$, that is, }\;
M_+(z,k_0) = \sum\nolimits_{k=0}^\infty M_{+,k}(k_0)z^k, \; z\in\D.
\\
(v) &\quad \big\{\phi_{+,k}(k_0)\big\}_{k=1}^N, \;\text{ where
$\phi_{+,k}(k_0)$, $k\geq 0$, are the Taylor coefficients of
$\Phi_+(z,k_0)$} \no\\ &\hspace*{5mm} \text{at $z=0$, that is, }\;
\Phi_+(z,k_0) = \sum\nolimits_{k=0}^\infty \phi_{+,k}(k_0)z^k, \;
z\in\D.
\end{align}
Similarly, for the left half-lattice problem, the following sets of
data $(vi)$--$(x)$ are equivalent:
\begin{align}
(vi) &\quad \big\{\al_{k_0-k}\big\}_{k=0}^{N-1}.
\\
(vii) &\quad \bigg\{\oint_{\dD}d\Om_-(\zeta,k_0)\, \zeta^k \bigg\}_{k=1}^N.
\\
(viii) &\quad \big\{m_{-,k}(k_0)\big\}_{k=1}^{N}, \;\text{ where
$m_{-,k}(k_0)$, $k\geq 0$, are the Taylor coefficients of
$m_-(z,k_0)$} \no\\ &\hspace*{5mm} \text{at $z=0$, that is, }\;
m_-(z,k_0) = \sum\nolimits_{k=0}^\infty m_{-,k}(k_0)z^k.
\\
(ix) &\quad \big\{M_{-,k}(k_0)\big\}_{k=0}^{N-1}, \;\text{ where
$M_{-,k}(k_0)$, $k\geq 0$, are the Taylor coefficients of
$M_-(z,k_0)$} \no\\ &\hspace*{5mm} \text{at $z=0$, that is, }\;
M_-(z,k_0) = \sum\nolimits_{k=0}^\infty M_{-,k}(k_0)z^k.
\\
(x) &\quad \big\{\phi_{-,k}(k_0)\big\}_{k=0}^{N-1}, \;\text{ where
$\phi_{-,k}(k_0)$, $k\geq 0$, are the Taylor coefficients of
$\Phi_-(z,k_0)^{-1}$} \no\\ &\hspace*{5mm} \text{at $z=0$, that is, }\;
\Phi_-(z,k_0)^{-1} = \sum\nolimits_{k=0}^\infty \phi_{-,k}(k_0)z^k.
\end{align}
\end{theorem}
\begin{proof}
$(i)\Rightarrow(ii)$ and $(vi)\Rightarrow(vii)$: First, utilizing
relations \eqref{3.37} and \eqref{3.40} with the initial conditions
\eqref{3.49} and \eqref{3.50}, one constructs $\{P_\pm(z,k_0\pm
k,k_0)\}_{k=1}^{N}$ and $\{R_\pm(z,k_0\pm k,k_0)\big\}_{k=1}^{N}$.
We note that the Laurent polynomials
\begin{align}
&\begin{cases}%
z^{-1}P_+(z,k_0+k,k_0), \; R_-(z,k_0-k,k_0), & k_0 \text{ odd},
\\[1mm]
R_+(z,k_0+k,k_0), \; z^{-1}P_-(z,k_0-k,k_0), & k_0 \text{ even},
\end{cases} \lb{5.15}
\intertext{are linear combinations of }
&\begin{cases}%
I_m,z^{-1}I_m,zI_m,z^{-2}I_m,z^2I_m,\dots,z^{(k-1)/2}I_m,z^{-(k+1)/2}I_m,
& k \text{ odd},
\\[1mm]
I_m,z^{-1}I_m,zI_m,z^{-2}I_m,z^2I_m,\dots,z^{-k/2}I_m,z^{k/2}I_m, & k
\text{ even},
\end{cases} \lb{5.16}
\intertext{and}
&\begin{cases}%
R_+(z,k_0+k,k_0), \; P_-(z,k_0-k,k_0), & k_0 \text{ odd},
\\[1mm]
P_+(z,k_0+k,k_0), \; R_-(z,k_0-k,k_0), & k_0 \text{ even},
\end{cases} \lb{5.17}
\intertext{are linear combinations of }
&\begin{cases}%
I_m,zI_m,z^{-1}I_m,z^2I_m,z^{-2}I_m,\dots,z^{-(k-1)/2}I_m,z^{(k+1)/2}I_m,
& k \text{ odd},
\\[1mm]
I_m,zI_m,z^{-1}I_m,z^2I_m,z^{-2}I_m,\dots,z^{k/2}I_m,z^{-k/2}I_m, & k
\text{ even}.
\end{cases} \lb{5.18}
\end{align}
Moreover, the last elements of the sequences in \eqref{5.16} and
\eqref{5.18} represent the leading-order terms of the Laurent polynomials in
\eqref{5.15} and \eqref{5.17}, respectively, and the corresponding
leading-order coefficients are invertible $m \times m$ matrices (cf.\ Remark \ref{r2.4}).

Next, assume $k_0$ and $k$ to be odd. Then utilizing \eqref{5.17}
and \eqref{5.18} one finds $m\times m$ matrices $C_{\pm,j}$ and
$D_{\pm,j}$, $0\leq j\leq k$, such that
\begin{align}
z^{-(k-1)/2}I_m  &= \sum_{j=0}^{k} C_{+,j}\, R_+(z,k_0+j, k_0),
\quad z^{(k+1)/2}I_m  = \sum_{j=0}^{k} D_{+,j}\, R_+(z,k_0+j,k_0),
\\
z^{-(k-1)/2}I_m  &= \sum_{j=0}^{k} C_{-,j}\, P_-(z,k_0-j, k_0),
\quad z^{(k+1)/2}I_m  = \sum_{j=0}^{k} D_{-,j}\, P_-(z,k_0-j,k_0),
\end{align}
and, using \eqref{3.56} and \eqref{3.57}, computes
\begin{align}
\oint_{\dD}d\Om_\pm(\zeta,k_0)\, \zeta^k  =
\oint_{\dD} \big(\zeta^{(k+1)/2}I_m\big) \, d\Om_\pm(\zeta,k_0)
\, \big(\zeta^{-(k-1)/2}I_m\big)^*
= \sum_{j=0}^k D_{\pm,j}\,C_{\pm,j}^*.
\end{align}
The other cases of $k_0$ and $k$ follow similarly.

$(ii)\Rightarrow(i)$ and $(vii)\Rightarrow(vi)$: Since
$d\Om_\pm(\cdot,k_0)$ are nonnegative normalized measures, one has
\begin{align}
\oint_{\dD}d\Om_\pm(\zeta,k_0)\,\zeta^{-k} =
\left(\oint_{\dD}d\Om_\pm(\zeta,k_0)\,\zeta^{k}\right)^*
\, \text{ and } \, \oint_{\dD}d\Om_\pm(\zeta,k_0) = I_m, \lb{5.22}
\end{align}
that is, the knowledge of positive moments imply the knowledge of
negative ones. Applying Corollary \ref{c3.8} one constructs the matrix-valued
orthonormal Laurent polynomials $\{P_\pm(\ze,k_0\pm k,k_0)\}_{k=1}^{N}$ and
$\{R_\pm(\ze,k_0\pm k,k_0)\big\}_{k=1}^{N}$. Subsequently applying
Theorem \ref{t3.9}, in particular, formulas \eqref{3.83} and
\eqref{3.84}, one obtains the coefficients $(i)$ and $(vi)$.

$(ii)\Leftrightarrow(iii)$ and $(vii)\Leftrightarrow(viii)$: These
follow from \eqref{3.111} and \eqref{5.22},
\begin{align}
m_\pm(z,k_0) &= \pm
\oint_{\dD}d\Om_{\pm}(\zeta,k_0)\,\frac{\zeta+z}{\zeta-z} = \pm
I_m \pm 2\sum_{k=1}^{\infty}
z^k\left(\oint_{\dD}d\Om_\pm(\zeta,k_0)\,\zeta^{k}\right)^*, \quad
z\in\D.
\end{align}

$(iii)\Leftrightarrow(iv)$: This is implied by \eqref{3.127}.

$(iv)\Leftrightarrow(v)$: This is a consequence of \eqref{3.133} and
\eqref{3.134}, together with the facts: For $|z|$ sufficiently small,
$\|M_+(z,k_0)-I_m\|_\Cm<1$ by \eqref{3.128}, and
$\|\Phi_+(z,k_0)\|_\Cm<1$ by \eqref{3.133a}. Hence,
\begin{align}
M_+(z,k_0) &= [I_m-\Phi_+(z,k_0)]^{-1}[I_m+\Phi_+(z,k_0)] \no \\
& \hspace{-1.5mm}
\underset{z\to 0}{=} [I_m+\Phi_+(z,k_0)] \sum_{k=0}^\infty \Phi_+(z,k_0)^k,
\\
\Phi_+(z,k_0) &= \big[2^{-1}[M_+(z,k_0)-I_m]\big]
\big[I_m+2^{-1}[M_+(z,k_0)-I_m]\big]^{-1}
\no \\
& \hspace{-1.5mm}
\underset{z\to 0}{=} -\sum_{k=1}^\infty 2^{-k}[I_m-M_+(z,k_0)]^k.
\end{align}

$(ix)\Leftrightarrow(x)$: This is implied by \eqref{3.131}, \eqref{3.133}, \eqref{3.134},
and the fact that, for $|z|$ sufficiently small, $\|\Phi_-(z,k_0)^{-1}\|_{\Cm}<1$ by
\eqref{3.7} and \eqref{3.133a}. Hence,
\begin{align}
M_-(z,k_0) &= [\Phi_-(z,k_0)^{-1}-I_m]^{-1}[\Phi_-(z,k_0)^{-1}+I_m]
\no
\\
&  \hspace{-1.5mm}
\underset{z\to 0}{=} -[\Phi_-(z,k_0)^{-1}+I_m]\sum_{k=0}^\infty \Phi_-(z,k_0)^{-k},
\\
\Phi_-(z,k_0)^{-1} &=
[M_-(z,k_0)+I_m][M_-(z,k_0)-M_-(0,k_0)+M_-(0,k_0)-I_m]^{-1} \no
\\ &=
\big[[M_-(z,k_0)+I_m][M_-(0,k_0)-I_m]^{-1}\big] \no
\\
&\quad \times \big[[M_-(z,k_0)-M_-(0,k_0)][M_-(0,k_0)-I_m]^{-1}+I_m\big]^{-1}
\\ &  \hspace{-1.5mm}
\underset{z\to 0}{=}
\big[[M_-(z,k_0)+I_m][M_-(0,k_0)-I_m]^{-1}\big] \no
\\
&\quad \times \sum_{k=0}^\infty
\big[[M_-(z,k_0)-M_-(0,k_0)][I_m-M_-(0,k_0)]^{-1}\big]^k. \no
\end{align}

$(viii)\Leftrightarrow(x)$: This follows because \eqref{3.111a}, \eqref{3.135}, and
the fact that $\|\Phi_-(z,k_0)^{-1}\|_{\Cm}\leq1$, $z\in\D$, together imply that
\begin{align}
m_-(z,k_0) &= [z\Phi_-(z,k_0)^{-1}+I_m]^{-1}[z\Phi_-(z,k_0)^{-1}-I_m]
\no \\
&  \hspace{-1.5mm}
\underset{z\to 0}{=} [z\Phi_-(z,k_0)^{-1}-I_m] \sum_{k=0}^\infty
\big[-z\Phi_-(z,k_0)^{-1}\big]^k,
\\
z\Phi_-(z,k_0)^{-1} &= [I_m+m_-(z,k_0)][I_m-m_-(z,k_0)]^{-1} \no
\\
&= 2^{-1} [I_m+m_-(z,k_0)]\big[I_m- 2^{-1}[I_m+m_-(z,k_0)]\big]^{-1}
\\
&  \hspace{-1.5mm}
\underset{z\to 0}{=} \sum_{k=1}^\infty 2^{-k}[I_m+m_-(z,k_0)]^k. \no
\qedhere
\end{align}
\end{proof}

Next, we restate Theorem \ref{t5.1}:

\begin{theorem} \lb{t5.2}
Assume Hypothesis \ref{h3.1} for two sequences $\al^{(1)}$,
$\al^{(2)}$ and let $k_0\in\Z$, $N\in\N$. Then for the right
half-lattice problems associated with $\al^{(1)}$ and $\al^{(2)}$
the following items $(i)$--$(iv)$ are equivalent:
\begin{align}
(i) &\quad \al_k^{(1)} = \al_k^{(2)}, \quad k_0+1\leq k\leq k_0+N.
\\
(ii) &\quad m_+^{(1)}(z,k_0)-m_+^{(2)}(z,k_0) \underset{z\to 0}{=} \oh(z^N).
\\
(iii) &\quad M_+^{(1)}(z,k_0)-M_+^{(2)}(z,k_0) \underset{z\to 0}{=} \oh(z^N).
\\
(iv) &\quad \Phi_+^{(1)}(z,k_0)-\Phi_+^{(2)}(z,k_0) \underset{z\to 0}{=} \oh(z^N).
\end{align}
Similarly, for the left half-lattice problems associated with
$\al^{(1)}$ and $\al^{(2)}$, the following items $(v)$--$(viii)$ are equivalent:
\begin{align}
(v) &\quad \al_k^{(1)} = \al_k^{(2)}, \quad k_0-N+1\leq k\leq k_0.
\\
(vi) &\quad m_-^{(1)}(z,k_0)-m_-^{(2)}(z,k_0) \underset{z\to 0}{=} \oh(z^N).
\\
(vii) &\quad M_-^{(1)}(z,k_0)-M_-^{(2)}(z,k_0) \underset{z\to 0}{=} \oh(z^{N-1}).
\\
(viii) &\quad \Phi_-^{(1)}(z,k_0)^{-1}-\Phi_-^{(2)}(z,k_0)^{-1} \underset{z\to 0}{=}
\oh(z^{N-1}).
\end{align}
\end{theorem}
\begin{proof}
This is an immediate consequence of Theorem \ref{t5.1}.
\end{proof}

Finally, we turn to CMV operators on $\bbZ$ and start with two auxiliary results
that play a role in the proofs of analogous
Borg--Marchenko-type uniqueness results for CMV operators on $\bbZ$.

\begin{lemma} \lb{l5.4}
Let $A,B,C,D$ denote some $m\times m$ matrices. Suppose that
$A\neq0$, $B$ is invertible, and $A,B,C,D$ satisfy
\begin{align}
\big[2\sqrt{\norm{A}\norm{D}}+\norm{C}\big]\|B^{-1}\|<1. \lb{5.45}
\end{align}
Then the matrix-valued Riccati-type equation
\begin{align}
XAX + BX + XC + D = 0, \quad
\norm{X}<\frac{1-\norm{C}\|B^{-1}\|}{2\norm{A}\|B^{-1}\|}, \lb{5.46}
\end{align}
has a unique solution $X\in\Cm$ given by
\begin{align}
X=\lim_{n\to\infty}X_n \,\text{ with }\, \norm{X}\leq
\frac{1-\|C\|\,\|B^{-1}\|}{2\|A\|\,\|B^{-1}\|} -
\sqrt{\left(\frac{1-\|C\|\,\|B^{-1}\|}{2\|A\|\,\|B^{-1}\|}\right)^2
- \frac{\|D\|}{\|A\|}}, \lb{5.47}
\end{align}
where
\begin{equation}
X_0=0, \quad X_n=F(X_{n-1}), \; n\in\bbN, \, \text{ and } \, F(X)=-B^{-1}XAX -
B^{-1}XC - B^{-1}D.
\end{equation}

A similar result also holds if $A\neq0$, $C$ is invertible, and
$A,B,C,D$ satisfying
\begin{align}
\big[2\sqrt{\norm{A}\norm{D}}+\norm{B}\big]\|C^{-1}\|<1. \lb{5.45a}
\end{align}
In this case, the matrix-valued Riccati-type equation
\begin{align}
XAX + BX + XC + D = 0, \quad
\norm{X}<\frac{1-\norm{B}\|C^{-1}\|}{2\norm{A}\|C^{-1}\|},
\lb{5.46a}
\end{align}
has a unique solution $X\in\Cm$ given by
\begin{align}
X=\lim_{n\to\infty}X_n \,\text{ with }\, \norm{X}\leq
\frac{1-\|B\|\,\|C^{-1}\|}{2\|A\|\,\|C^{-1}\|} -
\sqrt{\left(\frac{1-\|B\|\,\|C^{-1}\|}{2\|A\|\,\|C^{-1}\|}\right)^2
- \frac{\|D\|}{\|A\|}}, \lb{5.47a}
\end{align}
where
\begin{equation}
X_0=0, \quad X_n=G(X_{n-1}), \; n\in\bbN, \, \text{ and } \, G(X)=-XAXC^{-1} -
BXC^{-1} - DC^{-1}.
\end{equation}
\end{lemma}
\begin{proof}
Since $B$ is invertible, the equation for $X$ in
\eqref{5.46} is equivalent to $F(X)=X$. Therefore, it suffices to
show that $F(\cdot)$ is a strict contraction on some closed ball of
radius $\la$ centered at the origin, $B_\la=\{X\in\Cm \st
\norm{X}\leq\la\}$, and that $F(\cdot)$ preserves $B_\la$, that is
$\norm{F(X)}\leq\la$ whenever $\norm{X}\leq\la$.

First, we check that for any
$\la<\frac{1-\norm{C}\|B^{-1}\|}{2\norm{A}\|B^{-1}\|}$, the map
$F(\cdot)$ is a strict contraction on $B_\la$. Let $X,Y\in B_\la$,
then
\begin{align}
\norm{F(X)-F(Y)} &\leq \big[\norm{A}\|B^{-1}\|\norm{X} +
\norm{A}\|B^{-1}\|\norm{Y} + \norm{C}\|B^{-1}\|\big] \norm{X-Y}
\\
&\leq \big[2\la\norm{A}\|B^{-1}\|+\norm{C}\|B^{-1}\|\big] \norm{X-Y},
\quad 2\la\norm{A}\|B^{-1}\|+\norm{C}\|B^{-1}\|<1. \no
\end{align}

Next, we check that $F(\cdot)$ preserves $B_\la$ for any $\la$
satisfying
\begin{align}
\frac{1-\|C\|\,\|B^{-1}\|}{2\|A\|\,\|B^{-1}\|} -
\sqrt{\left(\frac{1-\|C\|\,\|B^{-1}\|}{2\|A\|\,\|B^{-1}\|}\right)^2
- \frac{\|D\|}{\|A\|}}
\leq\la<\frac{1-\|C\|\,\|B^{-1}\|}{2\|A\|\,\|B^{-1}\|}. \lb{5.49}
\end{align}
Let $X\in B_\la$, then by \eqref{5.49}
\begin{align}
\norm{F(X)}\leq \norm{A}\|B^{-1}\|\la^2 + \norm{C}\|B^{-1}\|\la +
\norm{D}\|B^{-1}\|\leq\la.
\end{align}
Thus, Banach's contraction mapping principle implies that
$F(\cdot)$ has a unique fixed point $X$ for which \eqref{5.46} and
\eqref{5.47} hold.

The second part of the Lemma is proved similarly.
\end{proof}

\begin{corollary} \lb{c5.4}
Let $A_j$, $B_j$, $C_j$, $D_j$, $j=1,2$, denote some $m\times m$
matrices. Suppose that either $B_1$ and $B_2$ are invertible and
\begin{align}
0<\norm{A_j},\,\|B^{-1}_j\| \leq a, \quad \norm{C_j},\,\norm{D_j}
\leq b, \quad j=1,2, \lb{5.54a}
\end{align}
or $C_1$ and $C_2$ are invertible and
\begin{align}
0<\norm{A_j},\,\|C^{-1}_j\| \leq a, \quad \norm{B_j},\,\norm{D_j}
\leq b, \quad j=1,2, \lb{5.55a}
\end{align}
for some $a,b>0$ satisfying $2ab(1+2a^2)\leq1$. Then there exist
unique solutions $X_j$, $j=1,2$, of the matrix-valued Riccati-type
equations
\begin{align}
X_jA_jX_j + B_jX_j + X_jC_j + D_j = 0, \quad
\norm{X_j}<\frac{1-ab}{2a^2}, \quad j=1,2, \lb{5.56a}
\end{align}
and the following estimate holds
\begin{align}
\norm{X_1-X_2} \leq \la(a,b)\big[\norm{A_1-A_2} + \norm{B_1-B_2} +
\norm{C_1-C_2} + \norm{D_1-D_2}\big], \lb{5.57a}
\end{align}
where $\la(a,b)$ is given by
\begin{align}
\la(a,b) = \frac{\max\left\{a, \frac{2a^2b}{1-ab},
a^2b+\frac{2a^3b^2}{1-ab}+\frac{4a^5b^2}{(1-ab)^2},
\frac{4a^3b^2}{(1-ab)^2}\right\}}
{(1-ab)-\frac{4a^3b}{1-ab}}
> 0. \lb{5.58a}
\end{align}
\end{corollary}
\begin{proof}
Suppose $B_j$, $j=1,2$, are invertible and note that
$b\leq1/(2a(1+2a^2))$ implies
\begin{align}
\Big(2\sqrt{\norm{A_j}\norm{D_j}}+\norm{C_j}\Big)\|B^{-1}_j\| \leq
(2\sqrt{ab}+b)a \leq \frac{2a\sqrt{2+4a^2}+1}{2(1+2a^2)} <
\frac{2a(2a+\frac{1}{2a})+1}{2(1+2a^2)} = 1
\end{align}
and
\begin{align}
\frac{1-\|C_j\|\,\|B^{-1}_j\|}{2\|A_j\|\,\|B^{-1}_j\|} -
\sqrt{\left( \frac{1-\|C_j\|\,\|B^{-1}_j\|}{2\|A_j\|\,\|B^{-1}_j\|}
\right)^2 - \frac{\|D_j\|}{\|A_j\|}} \leq \frac{2ab}{1-ab} <
\frac{1-ab}{2a^2} \leq
\frac{1-\norm{C_j}\|B^{-1}_j\|}{2\norm{A_j}\|B^{-1}_j\|}.
\end{align}
Then Lemma \ref{l5.4} implies that the matrix-valued
Riccati-type equations in \eqref{5.56a} have unique solutions $X_j$
satisfying $\norm{X_j}\leq\frac{2ab}{1-ab}$, $j=1,2$ and
$X_j=F_j(X_j)$, where $F_j(X)=-{B_j}^{-1}XA_jX - B^{-1}_jXC_j -
B^{-1}_jD_j$, $j=1,2$. Hence, one computes
\begin{align} \lb{5.61a}
\norm{X_1-X_2} &= \norm{F_1(X_1)-F_2(X_2)} \no
\\
&\leq \norm{B_1^{-1}X_1A_1X_1-B_2^{-1}X_2A_2X_2} +
\norm{B^{-1}_1X_1C_1-B^{-1}_2X_2C_2} +
\norm{B^{-1}_1D_1-B^{-1}_2D_2} \no
\\
&\leq \big[\norm{A_1}\|B^{-1}_2\|\norm{X_1} +
\norm{A_2}\|B^{-1}_2\|\norm{X_2} +
\|B^{-1}_2\|\norm{C_1}\big]\norm{X_1-X_2} \no
\\
&\quad + \|B^{-1}_2\|\norm{X_1}\norm{X_2}\norm{A_1-A_2} +
\|B^{-1}_2\|\norm{X_2}\norm{C_1-C_2} + \|B^{-1}_2\|\norm{D_1-D_2}
\no
\\
&\quad + \big[\norm{A_1}\norm{X_1}^2 + \norm{C_1}\norm{X_1} +
\norm{D_1}\big]\|B^{-1}_1\|\,\|B^{-1}_2\|\norm{B_1-B_2} \no
\\
&\leq \bigg(\frac{4a^3b}{1-ab}+ab\bigg)\norm{X_1-X_2} +
\bigg(\frac{4a^5b^2}{(1-ab)^2}+\frac{2a^3b^2}{1-ab}+a^2b\bigg)
\norm{B_1-B_2}
\\
&\quad + \frac{4a^3b^2}{(1-ab)^2}\norm{A_1-A_2} +
\frac{2a^2b}{1-ab}\norm{C_1-C_2} + a\norm{D_1-D_2}.  \no
\end{align}

Finally, utilizing $b\leq1/(2a(1+2a^2))$, one verifies that
\begin{align}
1-\bigg(\frac{4a^3b}{1-ab}+ab\bigg) = 1-\frac{4a^3b+ab(1-ab)}{1-ab}
> 1-\frac{ab(1+4a^2)}{1-ab} \geq 1-\frac{1+4a^2}{2(1+2a^2)-1} = 0,
\lb{5.62a}
\end{align}
and hence \eqref{5.57a} and \eqref{5.58a} follow from
\eqref{5.61a}, and \eqref{5.62a}.

The case of $C_j$ being invertible, $j=1,2$, is proved analogously.
\end{proof}

Given these preliminaries, we introduce the following notation for
the diagonal and for the neighboring off-diagonal entries of the
Green's matrix of $\U$ (i.e., the discrete integral kernel of $(\U-zI)^{-1}$),
\begin{align}
g(z,k) &= (\U-Iz)^{-1}(k,k),    \lb{5.62b}
\\
h(z,k) &=
\begin{cases}
(\U-Iz)^{-1}(k-1,k), & k \text{ odd}, \\
(\U-Iz)^{-1}(k,k-1), & k \text{ even},
\end{cases}\quad k\in\Z,\; z\in\D.      \lb{5.62c}
\end{align}

Then the subsequent uniqueness results hold for the full-lattice CMV
operator $\U$:

\begin{theorem}  \lb{t5.4}
Assume Hypothesis \ref{h3.1} and let $k_0\in\Z$. Then any of the
following two sets of data
\begin{enumerate}[$(i)$]
\item $g(z,k_0)$ and $h(z,k_0)$ for all $z$ in some open $($nonempty$)$
neighborhood of the origin under the assumption that $h(0,k_0)$ is invertible;
\item $g(z,k_0-1)$ and $g(z,k_0)$ for all $z$ in some open $($nonempty$)$ neighborhood
of the origin and $\al_{k_0}$ under the assumption $\al_{k_0}$ is invertible;
\end{enumerate}
uniquely determine the matrix-valued Verblunsky coefficients $\{\al_k\}_{k\in\Z}$, and
hence the full-lattice CMV operator $\U$.
\end{theorem}
\begin{proof}
{\it Case} $(i)$. First, we note that \eqref{3.18} implies that
\begin{align}
g(0,k_0) &= (\U^{-1})_{k_0,k_0} = (\U^*)_{k_0,k_0} =
(\U_{k_0,k_0})^* =
\begin{cases}
-\al_{k_0}\al_{k_0+1}^*, & k_0 \text{ odd},\\
-\al_{k_0+1}^*\al_{k_0}, & k_0 \text{ even},
\end{cases}
\\
h(0,k_0) &=
\begin{cases}
(\U^{-1})_{k_0-1,k_0} = (\U_{k_0,k_0-1})^* =
-\rho_{k_0}\al_{k_0+1}^*, & k_0 \text{ odd},
\\
(\U^{-1})_{k_0,k_0-1} = (\U_{k_0-1,k_0})^* =
-\al_{k_0+1}^*\wti\rho_{k_0}, & k_0 \text{ even}.
\end{cases}
\end{align}
Since $h(0,k_0)$ is invertible, one can solve the above equalities
for $\rho_{k_0}$ and $\al_{k_0}$,
\begin{align}
&g(0,k_0)h(0,k_0)^{-1} = \al_{k_0}\rho_{k_0}^{-1}, \quad k_0 \text{
odd},
\\
&h(0,k_0)^{-1}g(0,k_0) = \wti\rho_{k_0}^{-1}\al_{k_0} =
\al_{k_0}\rho_{k_0}^{-1}, \quad k_0 \text{ even},
\end{align}
implying
\begin{align}
&\rho_{k_0} =
\begin{cases}
\big[I_m+[g(0,k_0)h(0,k_0)^{-1}]^*[g(0,k_0)h(0,k_0)^{-1}]\big]^{-1/2}, & k_0
\text{ odd},\\
\big[I_m+[h(0,k_0)^{-1}g(0,k_0)]^*[h(0,k_0)^{-1}g(0,k_0)]\big]^{-1/2}, & k_0
\text{ even},
\end{cases}
\end{align}
and hence,
\begin{align}
\al_{k_0} =
\begin{cases}
g(0,k_0)h(0,k_0)^{-1}\rho_{k_0}, & k_0 \text{ odd},\\
h(0,k_0)^{-1}g(0,k_0)\rho_{k_0}, & k_0 \text{ even}.
\end{cases}
\end{align}
Using \eqref{3.10} and \eqref{3.11}, one also obtains
$a_{k_0}=I_m+\al_{k_0}$ and $b_{k_0}=I_m-\al_{k_0}$.

Next, utilizing \eqref{4.14}, \eqref{4.16}, and \eqref{4.17}, one
computes,
\begin{align}
\begin{split}
g(z,k_0)h(z,k_0)^{-1} &= -[I_m+M_-(z,k_0)][a_{k_0}^* -
b_{k_0}^*M_-(z,k_0)]^{-1}\rho_{k_0}, \quad k_0 \text{ odd},
\\
h(z,k_0)^{-1}g(z,k_0) &= -\wti\rho_{k_0}[a_{k_0}^* -
b_{k_0}^*M_-(z,k_0)]^{-1}[I_m+M_-(z,k_0)], \quad k_0 \text{ even}.
\end{split}
\end{align}
Solving for $M_-(z,k_0)$, one then obtains
\begin{align}
M_-(z,k_0) =
\begin{cases}
2g(z,k_0)[b_{k_0}^*g(z,k_0) - \rho_{k_0}h(z,k_0)]^{-1}-I_m, & k_0
\text{ odd},\\
2[g(z,k_0)b_{k_0}^* - h(z,k_0)\wti\rho_{k_0}]^{-1}g(z,k_0)-I_m, &
k_0 \text{ even}. \lb{5.51}
\end{cases}
\end{align}
The right-hand side of the above formula is well-defined for
sufficiently small $|z|$ since $b_{k_0}^*g(z,k_0) -
\rho_{k_0}h(z,k_0)$ for $k_0$ odd and $g(z,k_0)b_{k_0}^* -
h(z,k_0)\wti\rho_{k_0}$ for $k_0$ even are $\Cm$-valued analytic
functions having invertible values at the origin,
\begin{align}
\begin{split}
b_{k_0}^*g(0,k_0) - \rho_{k_0}h(0,k_0) &=
(\al_{k_0}-I_m)\rho_{k_0}^{-1}h(0,k_0), \quad k_0 \text{ odd},
\\
g(0,k_0)b_{k_0}^* - h(0,k_0)\wti\rho_{k_0} &=
h(0,k_0)\wti\rho_{k_0}^{-1}(\al_{k_0}-I_m), \quad k_0 \text{ even}.
\end{split} \lb{5.52}
\end{align}

Next, having $M_-(z,k_0)$ for sufficiently small $|z|$, one solves the equation
\begin{align}
h(z,k_0) = -\frac{1}{2z}
\begin{cases}
\rho_{k_0}^{-1}
[a_{k_0}^*-b_{k_0}^*M_-(z,k_0)][M_+(z,k_0)-M_-(z,k_0)]^{-1}
[I_m-M_+(z,k_0)],
& k_0 \text{ odd},\\
[I_m-M_+(z,k_0)][M_+(z,k_0)-M_-(z,k_0)]^{-1}
[a_{k_0}^*-M_-(z,k_0)b_{k_0}^*] \wti\rho_{k_0}^{-1}, & k_0 \text{
even},
\end{cases}
\end{align}
for $M_+(z,k_0)$ and obtains,
\begin{align}
M_+(z,k_0) &=
\begin{cases}
2[I_m+zg(z,k_0)]\big[I_m+z[b_{k_0}^*g(z,k_0)-\rho_{k_0} h(z,k_0)]\big]^{-1} -I_m, &
k \text{ odd},\\
2\big[I_m+z[g(z,k_0)b_{k_0}^*-
h(z,k_0)\wti\rho_{k_0}]\big]^{-1}[I_m+zg(z,k_0)] -I_m, & k_0 \text{ even}.
\lb{5.53}
\end{cases}
\end{align}
The right-hand side of \eqref{5.53} is well-defined for
sufficiently small $|z|$ since both $I_m+z(b_{k_0}^*g(z,k_0)
-\rho_{k_0}h(z,k_0))$ and $I_m+z(g(z,k_0)b_{k_0}^*-
h(z,k_0)\wti\rho_{k_0})$ are $\Cm$-valued analytic functions having
invertible values at the origin.

Finally, Theorem \ref{t5.1} (parts $(i)$, $(iv)$ and $(vi)$, $(ix)$)
implies that $M_\pm(z,k_0)$ for $z$ in some small neighborhood of
the origin uniquely determine Verblunsky coefficients
$\{\al_{k}\}_{k\in\Z}$.

{\it Case} $(ii)$. Suppose $k_0$ is odd. Then \eqref{3.133}, \eqref{4.14},
\eqref{4.15}, and
\begin{align}
2[I_m+zg(z,k_0)] &= [I_m+M_-(z,k_0)]W(z,k_0)^{-1}[I_m-M_+(z,k_0)]
\no
\\ &\quad +
[M_+(z,k_0)-M_-(z,k_0)]W(z,k_0)^{-1} +
W(z,k_0)^{-1}[M_+(z,k_0)-M_-(z,k_0)] \no
\\ & =
[I_m+M_+(z,k_0)]W(z,k_0)^{-1}[I_m-M_-(z,k_0)] \lb{5.55}
\end{align}
imply the identity,
\begin{align}
& z\rho_{k_0} g(z,k_0-1)\rho_{k_0} \no
\\
&\quad =
\frac12 \big[(I_m+\al_{k_0}^*)-(I_m-\al_{k_0}^*)M_+(z,k_0)\big]
W(z,k_0)^{-1} \big[(I_m+\al_{k_0})+M_-(z,k_0)(I_m-\al_{k_0})\big] \no
\\
&\quad = \frac12 \big[[I_m-M_+(z,k_0)]+\al_{k_0}^*[I_m+M_+(z,k_0)]\big]
W(z,k_0)^{-1} \lb{5.54}
\\
&\qquad \times \big[[I_m+M_-(z,k_0)]+[I_m-M_-(z,k_0)]\al_{k_0}\big]
\no
\\
&\quad = \frac12 [-\Phi_+(z,k_0)+\al_{k_0}^*]
[I_m+M_+(z,k_0)]W(z,k_0)^{-1}[I_m-M_-(z,k_0)]
[-\Phi_-(z,k_0)^{-1}+\al_{k_0}] \no
\\
&\quad = [\al_{k_0}^*-\Phi_+(z,k_0)] [I_m+zg(z,k_0)]
[\al_{k_0}-\Phi_-(z,k_0)^{-1}]. \no
\end{align}
Moreover, \eqref{5.55} also implies
\begin{align}
&zg(z,k_0)[I_m+zg(z,k_0)]^{-1} = [I_m+zg(z,k_0)]^{-1}zg(z,k_0) =
I_m -[I_m+zg(z,k_0)]^{-1} \no
\\ &\quad =
[I_m-M_-(z,k_0)]^{-1} \big[[I_m-M_-(z,k_0)][I_m+M_+(z,k_0)]-2W(z,k_0)\big]
[I_m+M_+(z,k_0)]^{-1} \no
\\ &\quad =
[I_m-M_-(z,k_0)]^{-1}
\big[I_m+M_-(z,k_0)-M_+(z,k_0)-M_-(z,k_0)M_+(z,k_0)\big]
[I_m+M_+(z,k_0)]^{-1} \no
\\ &\quad =
[I_m-M_-(z,k_0)]^{-1}[I_m+M_-(z,k_0)][I_m-M_+(z,k_0)][I_m+M_+(z,k_0)]^{-1}
\lb{5.56}
\\ &\quad =
\Phi_-(z,k_0)^{-1}\Phi_+(z,k_0). \no
\end{align}
Introducing the $\Cm$-valued analytic functions $A(z,k_0)$ and
$B(z,k_0)$ by
\begin{align}
A(z,k_0)=I_m+zg(z,k_0) \,\text{ and }\, B(z,k_0)=z\rho_{k_0}
g(z,k_0-1)\rho_{k_0} - \al_{k_0}^*A(z,k_0)\al_{k_0}, \lb{5.57}
\end{align}
one rewrites \eqref{5.54} as
\begin{align}
\Phi_+(z,k_0)A(z,k_0)\al_{k_0} + B(z,k_0) -
\Phi_+(z,k_0)A(z,k_0)\Phi_-(z,k_0)^{-1} +
\al_{k_0}^*A(z,k_0)\Phi_-(z,k_0)^{-1} = 0.
\end{align}
Multiplying both sides by $\Phi_+(z,k_0)$ on the right and utilizing \eqref{5.56}
then yields the Riccati-type equation for $\Phi_+(z,k_0)$,
\begin{align}
\Phi_+(z,k_0)A(z,k_0)\al_{k_0}\Phi_+(z,k_0) + B(z,k_0)\Phi_+(z,k_0)
- \Phi_+(z,k_0)zg(z,k_0) + \al_{k_0}^*zg(z,k_0) = 0. \lb{5.59}
\end{align}
Since by \eqref{3.128} and \eqref{3.133} $\Phi_+(0,k_0)=0$ and by
\eqref{5.57}
\begin{align}
&zg(z,k_0) \underset{z\to 0}{\longrightarrow} 0, \quad
A(z,k_0)  \underset{z\to 0}{\longrightarrow} I_m, \quad
B(z,k_0)  \underset{z\to 0}{\longrightarrow} \al_{k_0}^*\al_{k_0}, \lb{5.60}
\end{align}
Lemma \ref{l5.4} implies that equation \eqref{5.59} uniquely
determines the analytic function $\Phi_+(z,k_0)$ for $|z|$
sufficiently small.

Having $\Phi_+(z,k_0)$, one obtains $\Phi_-(z,k_0)^{-1}$ from \eqref{5.54}
for $|z|$ sufficiently small,
\begin{align}
\Phi_-(z,k_0)^{-1} = \al_{k_0} - [I_m+zg(z,k_0)]^{-1}
[\al_{k_0}^*-\Phi_+(z,k_0)]^{-1} z\rho_{k_0}g(z,k_0-1)\rho_{k_0}.
\lb{5.61}
\end{align}
The right-hand side of \eqref{5.61} is well-defined since
$I_m+zg(z,k_0)$ and $\al_{k_0}^*-\Phi_+(z,k_0)$ are $\Cm$-valued
analytic functions invertible at the origin.

Finally, Theorem \ref{t5.1} (parts $(i)$, $(v)$ and $(vi)$, $(x)$)
implies that $\Phi_\pm(z,k_0)^{\pm1}$ for $|z|$ sufficiently small
uniquely determine the Verblunsky coefficients $\{\al_{k}\}_{k\in\Z}$.

The case of $k_0$ even is proved similarly.
\end{proof}

In the subsequent result, $g^{(j)}$ and $h^{(j)}$ denote the corresponding quantities
\eqref{5.62b} and \eqref{5.62c} associated with the Verblunsky coefficients
$\alpha^{(j)}$, $j=1,2$.

\begin{theorem}  \lb{t4.6}
Assume Hypothesis \ref{h3.1} for two sequences $\al^{(1)}$,
$\al^{(2)}$ and let $k_0\in\Z$, $N\in\N$. Then for the full-lattice
problems associated with $\al^{(1)}$ and $\al^{(2)}$ the following
local uniqueness results hold:
\begin{enumerate}[$(i)$]
\item
If either $h^{(1)}(0,k_0)$ or $h^{(2)}(0,k_0)$ is invertible and
\begin{align}
\begin{split}
&\big\|g^{(1)}(z,k_0)-g^{(2)}(z,k_0)\big\|_\Cm
+ \big\|h^{(1)}(z,k_0)-h^{(2)}(z,k_0)\big\|_\Cm \underset{z\to 0}{=} \oh(z^N), \lb{5.71} \\
& \, \text{then } \, \al^{(1)}_k = \al^{(2)}_k \,\text{ for }\,
k_0-N \leq k\leq k_0+N+1.
\end{split}
\end{align}
\item
If $\al^{(1)}_{k_0}=\al^{(2)}_{k_0}$, $\al^{(1)}_{k_0}$ is
invertible, and
\begin{align}
\begin{split}
&\big\|g^{(1)}(z,k_0-1)-g^{(2)}(z,k_0-1)\big\|_\Cm +
\big\|g^{(1)}(z,k_0)-g^{(2)}(z,k_0)\big\|_\Cm \underset{z\to 0}{=} \oh(z^N), \lb{5.72}  \\
& \, \text{then } \, \al^{(1)}_k = \al^{(2)}_k \,\text{ for }\,
k_0-N-1 \leq k\leq k_0+N+1.
\end{split}
\end{align}
\end{enumerate}
\end{theorem}
\begin{proof}
{\it Case} $(i)$. The result is implied by Theorem \ref{t5.2} (parts
$(i)$, $(iii)$ and $(v)$, $(vii)$) upon verifying that \eqref{5.51},
\eqref{5.53}, and \eqref{5.71} imply
\begin{align}
\begin{split}
\big\|M_+^{(1)}(z,k_0)-M_+^{(2)}(z,k_0)\big\|_\Cm &\underset{z\to 0}{=} o(z^{N+1}),
\\
\big\|M_-^{(1)}(z,k_0)-M_-^{(2)}(z,k_0)\big\|_\Cm &\underset{z\to 0}{=} o(z^N).
\end{split}
\end{align}

{\it Case} $(ii)$. The result is a consequence of Theorem \ref{t5.2} (parts
$(i)$, $(iv)$ and $(v)$, $(viii)$) upon verifying that Corollary
\ref{c5.4}, \eqref{5.57}, \eqref{5.59}, \eqref{5.61}, and
\eqref{5.72} imply
\begin{align}
\norm{\Phi_+^{(1)}(z,k_0)-\Phi_+^{(2)}(z,k_0)}_\Cm +
\norm{\Phi_-^{(1)}(z,k_0)^{-1}-\Phi_-^{(2)}(z,k_0)^{-1}}_\Cm \underset{z\to 0}{=} o(z^{N+1}).
\end{align}
\end{proof}

\appendix
\section{Basic Facts on Caratheodory and Schur Functions}
\lb{sA}
\renewcommand{\theequation}{A.\arabic{equation}}
\renewcommand{\thetheorem}{A.\arabic{theorem}}
\setcounter{theorem}{0}
\setcounter{equation}{0}

In this appendix we summarize a few basic properties of matrix-valued Caratheodory
and Schur functions used throughout this manuscript. (For the analogous case of
matrix-valued Herglotz functions we refer to \cite{GT00} and the extensive list of references therein.)

We denote by $\D$ and $\dD$ the open unit disk
and the counterclockwise oriented unit circle in the complex plane $\C$,
\begin{equation}
\D = \{ z\in\C \st \abs{z} < 1 \}, \quad \dD = \{ \ze\in\C \st \abs{\ze} = 1 \}.
\end{equation}
Moreover, we
denote as usual $\Re(A)=(A+A^*)/2$ and $\Im(A)=(A-A^*)/(2i)$
for square matrices $A$ with complex-valued entries.

\begin{definition} \lb{dA.1}
Let $m\in\bbN$ and $F_\pm$, $\Phi_+$, and $\Phi_-^{-1}$ be $m\times m$
matrix-valued analytic functions in $\D$. \\
$(i)$ $F_+$ is called a {\it Caratheodory matrix} if $\Re(F_+(z))\geq 0$ for all
$z\in\D$ and $F_-$ is called an {\it anti-Caratheodory matrix} if $-F_-$ is a
Caratheodory matrix. \\
$(ii)$ $\Phi_+$ is called a {\it Schur matrix} if $\|\Phi_+(z)\|_{\Cm} \leq 1$,
for all $z\in\D$.\  $\Phi_-$ is called an {\it anti-Schur matrix} if $\Phi_-^{-1}$
is a Schur matrix.
\end{definition}

\begin{theorem} \lb{tA.2}
Let $F$ be an $m\times m$ Caratheodory matrix,
$m\in\bbN$. Then $F$ admits the Herglotz representation
\begin{align}
& F(z)=iC+ \oint_{\dD} d\Omega(\zeta) \, \f{\zeta+z}{\zeta-z},
\quad z\in\D, \lb{A.3}
\\
& C=\Im(F(0)), \quad \oint_{\dD} d\Omega(\zeta) = \Re(F(0)),
\end{align}
where $d\Omega$ denotes a nonnegative $m \times m$
matrix-valued measure on $\dD$. The measure $d\Omega$ can
be reconstructed from $F$ by the formula
\begin{equation}
\Omega\big(\Arc\big(\big(e^{i\te_1},e^{i\te_2}\big]\big)\big)
=\lim_{\delta\downarrow 0} \lim_{r\uparrow 1} \f{1}{2\pi}
\oint_{\te_1+\delta}^{\te_2+\delta} d\te \,
\Re\big(F\big(r\zeta\big)\big),
\end{equation}
where
\begin{equation}
\Arc\big(\big(e^{i\theta_1},e^{i\theta_2}\big]\big)
=\big\{\zeta\in\dD\,|\, \theta_1<\te\leq \theta_2\big\}, \quad
\theta_1 \in [0,2\pi), \; \theta_1<\theta_2\leq \theta_1+2\pi. \lb{A.5}
\end{equation}
Conversely, the right-hand side of equation \eqref{A.3} with $C
= C^*$ and $d\Omega$ a finite nonnegative $m \times m$
matrix-valued measure on $\dD$ defines a Caratheodory matrix.
\end{theorem}

We note that additive nonnegative $m\times m$ matrices on the
right-hand side of \eqref{A.3} can be absorbed into the measure $d\Om$
since
\begin{equation}
\oint_\dD d\mu_0(\zeta) \, \f{\zeta+z}{\zeta-z}=1, \quad z\in\D,
\lb{A.5a}
\end{equation}
where
\begin{equation}
d\mu_0(\zeta)=\f{d\te}{2\pi}, \quad \zeta=e^{i\te}, \; \te\in
[0,2\pi) \lb{A.5b}
\end{equation}
denotes the normalized Lebesgue measure on the unit circle $\dD$.

Given a Caratheodory (resp., anti-Caratheodory) matrix $F_+$
(resp. $F_-$) defined in $\D$ as in \eqref{A.3}, one extends $F_\pm$ to
all of $\bbC\backslash\dD$ by
\begin{equation}
F_\pm(z)=iC_\pm \pm \oint_{\dD} d\Om_\pm (\zeta) \,
\f{\zeta+z}{\zeta-z}, \quad z\in\bbC\backslash\dD, \;\;
C_\pm=C_\pm^*. \lb{A.6}
\end{equation}
In particular,
\begin{equation}
F_\pm(z) = -F_\pm(1/\ol{z})^*, \quad z\in\C\backslash\ol{\D}. \lb{A.7}
\end{equation}
Of course, this continuation of $F_\pm|_{\D}$ to
$\bbC\backslash\ol\D$, in general, is not an analytic
continuation of $F_\pm|_\D$.

Next, given the functions $F_\pm$ defined in $\bbC\backslash\dD$
as in \eqref{A.6}, we introduce the functions $\Phi_\pm$
by
\begin{equation}
\Phi_\pm(z)=[F_\pm(z)-I_m][F_\pm(z)+I_m]^{-1}, \quad
z\in\bbC\backslash\dD.  \lb{A.11}
\end{equation}
We recall (cf., e.g., \cite[p.\ 167]{SF70}) that if $\pm \Re(F_\pm) \geq 0$, then
$[F_\pm \pm I_m]$ is invertible.
In particular, $\Phi_+|_{\D}$ and $[\Phi_-]^{-1}|_{\D}$ are Schur matrices
(resp., $\Phi_-|_{\D}$ is an anti-Schur matrix). Moreover,
\begin{equation}
F_\pm(z)= [I_m-\Phi_\pm (z)]^{-1} [I_m+\Phi_\pm (z)], \quad
z\in\bbC\backslash\dD.    \lb{A.12}
\end{equation}

\medskip

\noindent {\bf Acknowledgments.}
We are indebted to Konstantin A.\ Makarov and Eduard Tsekanovskii
for helpful discussions.


\end{document}